\documentclass[reqno]{amsart}

\usepackage{amsmath, amssymb}
\usepackage{amsthm, stmaryrd}
\usepackage{mathrsfs}

\usepackage[top=40truemm,bottom=40truemm,left=30truemm,right=30truemm]{geometry}

\usepackage{hyperref}

\usepackage{url}

\makeatletter
	\@addtoreset{equation}{section}

\theoremstyle{plain}
\newtheorem{theorem}{Theorem}[section]
\newtheorem{proposition}[theorem]{Proposition}
\newtheorem{lemma}[theorem]{Lemma}
\newtheorem{corollary}[theorem]{Corollary}

\theoremstyle{definition}
\newtheorem{definition}[theorem]{Definition}
\newtheorem{remark}[theorem]{Remark}
\newtheorem{example}{Example}[section]

\allowdisplaybreaks[4]

\date{}

\author{Takumu Ooi}
\address{Department of Mathematics,
Faculty of Science and Technology, Tokyo University of Science,
Noda, Chiba, 278-8510, Japan}
\email{ooitaku@rs.tus.ac.jp}

\author{Motohiro Sobajima}
\address{Department of Mathematics,
Faculty of Science and Technology, Tokyo University of Science,
Noda, Chiba, 278-8510, Japan}
\email{msobajima1984@gmail.com}

\title[Subcriticality of subordinated Schr\"{o}dinger op. and wave equations]{Subcriticality of subordinated Schr\"{o}dinger operators and their application to wave equations}

\begin{document}
\maketitle

\begin{abstract}
We provide a probabilistic characterization of criticality, subcriticality, and supercriticality for subordinated Schr\"{o}dinger operators. We also investigate the relationship between the subcriticality of these operators and the uniform boundedness of solutions to the associated wave equation.
\end{abstract}

\section{Introduction}
For a non-positive self-adjoint operator \(\mathcal{L}\) corresponding to a symmetric Hunt process on a locally compact separable metric space \(E\), a signed smooth Radon measure \(\mu\), and a Bernstein function \(\Phi\), we consider subordinated Schr\"{o}dinger operators \(\Phi(-\mathcal{L}+\mu)\) on \(E\). A symmetric Hunt process is a type of Markov process whose infinitesimal generator is a symmetric self-adjoint operator, which is not necessarily local. A Bernstein function is the Laplace exponent of a non-negative increasing L\'{e}vy process. For example, in the case \(\mathcal{L}:=\Delta\) on \(\mathbb{R}^d\), \(\mu:=Vdx\) for a continuous function \(V\), and \(\Phi(\lambda):=\lambda^\beta\) with \(0<\beta<1\), we consider \(\Phi(-\mathcal{L}+\mu)=(-\Delta + V)^\beta\). In previous work by the second author \cite{S26}, although subcriticality is not explicitly defined, subcriticality for \((-\Delta + V)^\beta\) is investigated via the range of \((-\Delta + V)^\beta\), and this approach is applied to establish the boundedness of solutions to the wave equation associated with \((-\Delta + V)^\beta\).

To the best of the authors' knowledge, criticality, subcriticality, and supercriticality for subordinated Schr\"{o}dinger operators \(\Phi(-\mathcal{L}+\mu)\) have not been explicitly defined when \(\mu\) has a nontrivial negative part and \(\Phi\) is not the identity. In this paper, we investigate a probabilistic characterization of criticality, subcriticality, and supercriticality for subordinated Schr\"{o}dinger operators \(\Phi(-\mathcal{L}+\mu)\) on a locally compact separable metric space \(E\) using Dirichlet form theory, which provides a general framework for Markov processes.

Let \((\mathcal{E}, \mathcal{D}(\mathcal{E}))\) be an irreducible regular Dirichlet form on \(L^2(E;m)\), and let \(\mathcal{L}\) be its associated self-adjoint operator. We take a signed smooth Radon measure \(\mu\) on \(E\) such that \((\mathcal{E}^\mu, \mathcal{D}(\mathcal{E}) \cap C_c(E))\) is non-negative and closable, where
\[
\mathcal{E}^\mu(f,g):=\mathcal{E}(f,g)-\int_E fg \, d\mu := \mathcal{E}(f,g)-\int_E fg \, d\mu^+ + \int_E fg \, d\mu^- 
\]
for a positive (resp. negative) part \(\mu^+\) (resp. \(\mu^-\)) for \(\mu\).
We define a Schr\"{o}dinger form \((\mathcal{E}^\mu, \mathcal{D}(\mathcal{E}^\mu))\) on \(L^2(E;m)\) by letting \(\mathcal{D}(\mathcal{E}^\mu)\) be the closure of \(\mathcal{D}(\mathcal{E}) \cap C_c(E)\) with respect to the norm induced by \(\mathcal{E}^\mu_1(f,g):=\mathcal{E}^\mu(f,g)+\langle f, g \rangle_m\), where \(\langle \cdot, \cdot \rangle_m\) denotes the inner product on \(L^2(E;m)\). This form \((\mathcal{E}^\mu, \mathcal{D}(\mathcal{E}^\mu))\) corresponds to \(\mathcal{L}-\mu\), but it does not necessarily correspond to a stochastic process. Takeda \cite{T14} characterized subcriticality (resp.\ criticality) for \(\mathcal{L}-\mu\) as transience (resp.\ recurrence) of the Markov process associated with Doob's \(h\)-transform of \((\mathcal{E}^\mu, \mathcal{D}(\mathcal{E}^\mu))\) for some (equivalently, any) admissible superharmonic function \(h\). He also proved that this probabilistic definition is consistent with the spectral analytic definition.

As in the case of a Schr\"{o}dinger operator, a subordinated Schr\"{o}dinger operator \(\Phi(-\mathcal{L}+\mu)\) does not necessarily correspond to a stochastic process nor Schr\"{o}dinger form in general. Therefore, we introduce a subordinated Schr\"{o}dinger form \((\mathcal{E}^{\mu,\Phi}, \mathcal{D}(\mathcal{E}^{\mu,\Phi}))\) associated with \(\Phi(-\mathcal{L}+\mu)\) using a method similar to the subordination of Dirichlet forms \cite{O02}. Under an additional assumption on \(\Phi\) related to irreducibility, we define \(\Phi(-\mathcal{L}+\mu)\) to be subcritical (resp.\ critical) if there exists an admissible superharmonic function \(h\) such that Doob's \(h\)-transform of \((\mathcal{E}^{\mu,\Phi}, \mathcal{D}(\mathcal{E}^{\mu,\Phi}))\) is transient (resp.\ recurrent). This definition is independent of the choice of \(h\), that is, if the \(h\)-transform is transient (resp.\ recurrent) for some admissible function \(h\), then it is also transient (resp.\ recurrent) for any admissible \(h\) (Theorems \ref{subodinated_subcri}, \ref{subodinated_cri}, \ref{subodinated_supercri}).

We also study the preservation of criticality and subcriticality. If a Schr\"{o}dinger operator \(\mathcal{L}-\mu\) is subcritical, then the subordinated Schr\"{o}dinger operator \(\Phi(-\mathcal{L}+\mu)\) is also subcritical (Corollary \ref{subcrinomama}). Under an additional assumption that suppresses large jumps of the subordinator, criticality is preserved under subordination (Proposition \ref{cri_inv}). We also present examples where a Schr\"{o}dinger operator is critical, but its subordinated operators are critical for some Bernstein functions \(\Phi\) and subcritical for others (Section \ref{sec:Example}). Moreover, we obtain equivalent conditions for subcriticality by determining the range of \(\sqrt{-\Phi(-\mathcal{L}+\mu)}\) (Theorems \ref{thm_range}, \ref{extSoba}, \ref{extSoba2} and Corollaries \ref{cor_extSoba1}, \ref{cor_extSoba2}), which generalize \cite[Theorem 1.1]{S26}.

We further derive properties of solutions to the following wave equation associated with \(\Phi(-\mathcal{L}+\mu)\):
\begin{eqnarray}
\begin{cases}
\frac{\partial^2}{\partial t^2} w(x,t) = 
-\Phi(-\mathcal{L}+\mu) w(x,t) \quad \text{for } (x,t)\in E\times (0,\infty),\\
\frac{\partial}{\partial t} w(x,0) = g(x)\in C_c \quad \text{for } x\in E,\\
w(x,0) = 0 \quad \text{for } x\in E.
\end{cases}
\label{eq:intro}
\end{eqnarray}

Under suitable conditions, if \(\Phi(-\mathcal{L}+\mu)\) is subcritical, then the solution to \eqref{eq:intro} is uniformly bounded in \(L^2(E;m)\) (Theorems \ref{wave_bdd}, \ref{wave_bdd_SF}). Conversely, for the specific class of subordinators given by \(\Phi(\lambda)=\lambda^\beta\), if every solution to the wave equation \eqref{eq:intro} with \(\Phi(-\mathcal{L}+\mu)\) replaced by \(-\mathcal{L}+\mu\) is uniformly bounded, then \((-\mathcal{L}+\mu)^\beta\) is subcritical (Theorem \ref{oppose_wave_bdd}). These relationships were established in \cite{S26} for \(\mathcal{L}=\Delta\) and an absolutely continuous signed measure \(\mu\). In this paper, we generalize these results and provide a rigorous framework for the relationship between subcriticality and wave equations.

The organization of this paper is as follows. In Section \ref{sec:Takeda}, we review previous work \cite{T14, TU23} on probabilistic characterizations of criticality for Schr\"{o}dinger operators using Dirichlet form theory, which provides a general framework for Markov processes. In Section \ref{sec:sub_shr}, we introduce definitions of subcriticality and criticality for subordinated Schr\"{o}dinger operators and examine whether subordination preserves these properties. Section \ref{sec:wave} presents applications to the relationship between subcriticality and the uniform boundedness of solutions to wave equations. In Section \ref{sec:Example}, we provide examples, including classical Hardy inequalities, trace Hardy inequalities on Euclidean spaces, spaces with varying dimension, and fractal spaces. Appendix \ref{Appendix} contains preliminaries on Dirichlet form theory and Markov processes.

Throughout this paper, we use the notation \(a\wedge b := \min\{a,b\}\) and \(a\vee b := \max\{a,b\}\) for \(a,b\in \mathbb{R}\).

\section{Criticality and subcriticality of a Schr\"{o}dinger form}\label{sec:Takeda}
Throughout this paper, we assume that \((\mathcal{E}, \mathcal{D}(\mathcal{E}))\) is an irreducible regular Dirichlet form on \(L^2(E;m)\). More precisely, \(E\) is a locally compact separable metric space, and \(m\) is a positive Radon measure with full support. Moreover \((\mathcal{E}, \mathcal{D}(\mathcal{E}))\) is a non-negative symmetric closed bilinear form satisfying the Markov property, and  \(\mathcal{D}(\mathcal{E}) \cap C_c(E)\) is dense in \(\mathcal{D}(\mathcal{E})\) with respect to \(\sqrt{\mathcal{E}_1}\) and dense in \(C_c(E)\) with respect to \(\|\cdot\|_\infty\). Here and throughout this paper, denote by \(C_c(E)\) the space of continuous functions with compact support, which is equipped with a sup norm \(||\cdot||_{\infty}\), where \(\mathcal{E}_1(f,g):=\mathcal{E}(f,g)+\langle f,g\rangle_m\). Then, there exists a strong Markov process on \(E\), called a Hunt process associated with \((\mathcal{E}, \mathcal{D}(\mathcal{E}))\). Denote by \(\{T_t\}_{t>0}\) the associated strongly continuous contraction semigroup on \(L^2(E;m)\) and \(-\mathcal{L}\) the non-negative definite self-adjoint operator. See Appendix \ref{Appendix} for details on Dirichlet form theory.

We will define subcriticality, criticality and supercriticality of subordinations of Schr\"{o}dinger operators such as \((-\mathcal{L}+\mu)^\alpha\) for a signed measure \(\mu\) from a perspective of probability theory in Section \ref{sec:sub_shr}. We will characterize them similarly to the characterization of Schr\"{o}dinger operators, such as \(-\mathcal{L}+\mu\). Therefore, in this section, we summarize previous works, mainly those of Takeda and Uemura \cite{T14, TU23}.

We consider Schr\"{o}dinger operators perturbed by the following type of singular measures. See Appendix \ref{Appendix} for the definition of a nest and an \(\mathcal{E}\)-polar set.
\begin{definition}[{\cite[Definition 2.3.13]{CF12}}]
A positive Borel measure \(\mu\) on \(E\) is a \textit{smooth measure} if \(\mu\) charges no \(\mathcal{E}\)-polar set and there exists a nest \(\{F_k\}_k\) such that \(\mu(F_k)<\infty\) for every \(k\geq 1.\) Denote by \(\mathcal{S}\) the family of all smooth measures. For subclasses \(\mathcal{T}_1\) and \(\mathcal{T}_2\) of \(\mathcal{S}\), we denote by \(\mathcal{T}_1-\mathcal{T}_2\) the class of all signed smooth measures \(\mu=\mu^+-\mu^-\) for \(\mu^+ \in \mathcal{T}_1\) and \(\mu^- \in \mathcal{T}_2\). 
\end{definition}
For example, \(|f|dm\) is a smooth measure whenever \(f\in L^1_{loc}(E;m)\). We remark that there exist both singular smooth measures and smooth measures that are not Radon. See Appendix \ref{Appendix_PCAF} for details. See also \cite{AM92, OTU25+} for nowhere Radon measures.

Let \(\mathcal{S}_R\) be the set of all smooth Radon measures. For a signed smooth Radon measure \(\mu:=\mu^+-\mu^-\in \mathcal{S}_R-\mathcal{S}_R\), we define a symmetric form \((\mathcal{E}^{\mu}, \mathcal{D}(\mathcal{E})\cap C_c(E))\) on \(L^2(E;m)\) by
\begin{eqnarray*}
\mathcal{E}^{\mu}(f,g)&:=& \mathcal{E}(f,g)+ \int_E f\,g\,d\mu\  := \ \mathcal{E}(f,g)+ \int_E f\,g\,d\mu^+ -  \int_E f\,g\,d\mu^-,\ \ f,g\in \mathcal{D}(\mathcal{E}) \cap C_c(E).
\end{eqnarray*}

We assume that \((\mathcal{E}^{\mu}, \mathcal{D}(\mathcal{E})\cap C_c(E))\) is non-negative definite, that is, \(\mathcal{E}^\mu(f,f) \geq 0\) for any \(f \in \mathcal{D}(\mathcal{E})\cap C_c(E)\), and closable, that is, if \(f_n \in \mathcal{D}(\mathcal{E})\cap C_c(E)\) satisfies \(\mathcal{E}^\mu_1(f_n-f_m, f_n-f_m) \to 0\) and \(\langle f_n, f_n\rangle_m \to 0\) then \(\mathcal{E}^\mu(f_n, f_n)\to 0.\) Here and throughout this paper, we denote by \(\langle \cdot, \cdot \rangle_m\) an inner product of \(L^2(E;m)\) and \(\|\cdot \|_m := \|\cdot \|_{L^2(E;m)}\) is an \(L^2(E;m)\)-norm. We set \(\mathcal{E}^{\mu}_\alpha(f,g):=\mathcal{E}^{\mu}(f,g) + \alpha\langle f,g\rangle_m\).

Denote by \((\mathcal{E}^{\mu}, \mathcal{D}(\mathcal{E}^\mu))\) the closure of \((\mathcal{E}^{\mu}, \mathcal{D}(\mathcal{E})\cap C_c(E))\) and we call this closure a Schr\"{o}dinger form. By the closedness, \(\mathcal{E}^\mu(f,f) \geq 0\) holds for any \(f \in \mathcal{D}(\mathcal{E}^\mu)\). By \cite[Lemma 1.3.4]{D89}, for any \(f\in \mathcal{D}(\mathcal{E})\cap C_c(E)\), it holds that \(|f| \in \mathcal{D}(\mathcal{E}^\mu)\) and \(\mathcal{E}^\mu(|f|,|f|)\leq \mathcal{E}^\mu(f,f)\). Since \((\mathcal{E}^{\mu}, \mathcal{D}(\mathcal{E}^\mu))\) is closed non-negative definite and symmetric, by \cite[Theorem 1.5]{O13}, there exists a strongly continuous contraction semigroup \(\{T_t^\mu\}_t\) on \(L^2(E;m)\) such that \(\mathcal{E}_\alpha^\mu (G_\alpha^\mu f,g)=\langle f, g\rangle_m\) for any \(f\in L^2(E; m)\) and \(g\in \mathcal{D}(\mathcal{E}^\mu)\), where \(G_\alpha^\mu f:=\int_0^\infty e^{-\alpha t} T_t^\mu f\,dt\). Moreover, \(\mathcal{L}^{\mu}= \mathcal{L}-\mu\) is a (non-positive) self-adjoint operator satisfying \(\mathcal{E}^\mu(f,f)=\|\sqrt{-\mathcal{L}^\mu}f\|_m^2\), \(\mathcal{D}(\mathcal{E}^\mu)=\mathcal{D}(\sqrt{-\mathcal{L}^\mu})\) and \(T_t^\mu =e^{\mathcal{L}^\mu t}\). Let \(X\) be an \(m\)-symmetric Hunt process on \(E\)  associated with a regular Dirichlet form \((\mathcal{E}, \mathcal{D}(\mathcal{E}))\) and let \(A^{\mu^+}\) (resp. \(A^{\mu^-}\)) be a PCAF corresponding to \(\mu^+\) (resp. \(\mu^-\)). See Appendix \ref{Appendix} for the definition of PCAFs and the relation between smooth measures and PCAFs. For \(A_t^\mu:=A^{\mu^+}_t-A^{\mu^-}_t\), we set \(P_t^{\mu}f(x):=\mathbb{E}_x[e^{A^\mu_t}f(X_t)]\) for \(f\in L^2(E;m)\) and \(t>0\). Then \(P^\mu_t f = T^\mu_t f\) \(m\)-almost everywhere. See \cite[Proposition 3.1.9]{CF12} for example.

We note that \((\mathcal{E}^{\mu^+}, \mathcal{D}(\mathcal{E}^{\mu^+}))\) is a regular Dirichlet form called a perturbed Dirichlet form, where \(\mathcal{D}(\mathcal{E}^{\mu^+}):=\mathcal{D}(\mathcal{E}) \cap L^2(E;\mu^+)\) and \(\mathcal{E}^{\mu^+}(f,g):=\mathcal{E}(f,g)+ \int fg\,d\mu^+\). See \cite[Section 5.1]{CF12} for details. Without loss of generality, we may assume that
 \((\mathcal{E}^{\mu^+}, \mathcal{D}(\mathcal{E}^{\mu^+}))\) is transient. Indeed, if \((\mathcal{E}^{\mu^+}, \mathcal{D}(\mathcal{E}^{\mu^+}))\) is recurrent, then there exists \(f_n \in \mathcal{D}(\mathcal{E}^{\mu^+})\cap C_c(E) =\mathcal{D}(\mathcal{E})\cap C_c(E)\) such that \(\mathcal{E}^{\mu^+}(f_n,f_n) \to 0\) and \(f_n \to 1\) \(m\)-almost everywhere, so we have \(\mu^- =0\) by the non-negativity of \(\mathcal{E}^{\mu}\). In the case of \(\mu^-\), a Schr\"{o}dinger form \((\mathcal{E}^{\mu^+}, \mathcal{D}(\mathcal{E}^{\mu^+}))\) corresponds to a Hunt process, so we can directly characterize the criticality for \(\mathcal{L}-\mu^+\).

Since \((\mathcal{E}^{\mu^+}, \mathcal{D}(\mathcal{E}^{\mu^+}))\) is a regular Dirichlet form, for any \(f\in \mathcal{D}(\mathcal{E}^{\mu^+})\), there exists \(f_n \in \mathcal{D}(\mathcal{E}) \cap C_c(E)\) such that \(\mathcal{E}_1^{\mu^+}(f_n-f, f_n-f) \to 0\), which implies that 
\[\int_E |f_n-f|^2\,d\mu^- \leq \mathcal{E}_1^{\mu^+}(f_n-f,f_n-f) \to 0\]
and so \(\mathcal{D}(\mathcal{E}^{\mu^+}) \subset \mathcal{D}(\mathcal{E}^{\mu}).\)

We define an extended space \(\mathcal{D}_e(\mathcal{E}^\mu)\) as the set of all \(m\)-measurable functions \(f\) satisfying \(|f|<\infty\) \(m\)-almost everywhere and that admit an approximating sequence \(\{f_n\}_n \subset \mathcal{D}(\mathcal{E}^\mu)\) such that \(\mathcal{E}^\mu(f_n-f_m, f_n-f_m)\to 0\) as \(m,n \to \infty\) and \(f_n \to f\) \(m\)-almost everywhere, and we define \(\mathcal{E}^\mu(f,f):=\lim_{n\to \infty} \mathcal{E}^\mu(f_n,f_n)\). This limit is independent of the choice of approximating sequence. As in \cite[Theorem 2.3.4]{CF12}, any function \( f \in \mathcal{D}_e(\mathcal{E}^\mu)\) has a quasi-continuous version \(\tilde{f}\), that is, \(f=\tilde{f}\) \(m\)-almost everywhere and there exists a nest \(\{F_k\}_k\) such that a restriction of \(\tilde{f}\) to each \(F_k\) is continuous. Hence, without loss of generality, we assume that any function belonging to \(\mathcal{D}_e(\mathcal{E}^\mu)\) is a quasi-continuous function.

\begin{definition}[\cite{TU23}] Let \((\mathcal{E}, \mathcal{D}(\mathcal{E}))\) be an irreducible regular Dirichlet form on \(L^2(E;m)\) and \(\mu\) be a signed smooth Radon measure making \((\mathcal{E}^{\mu}, \mathcal{D}(\mathcal{E}) \cap C_c(E))\) non-negative definite closable. We define criticalities for Schr\"{o}dinger form as follows.\\
\((1)\) A Schr\"{o}dinger form \((\mathcal{E}^{\mu}, \mathcal{D}(\mathcal{E}^{\mu}))\) is critical if there exists a strictly positive function \(h\in  \mathcal{D}_e(\mathcal{E}^{\mu})\) satisfying \(\mathcal{E}^\mu(h,h)=0.\)\\
\((2)\) A Schr\"{o}dinger form \((\mathcal{E}^{\mu}, \mathcal{D}(\mathcal{E}^{\mu}))\) is subcritical if there exists a strictly positive bounded function \(g \in L^1(E;m)\) satisfying, 
for \(f\in \mathcal{D}_e(\mathcal{E}^{\mu})\),
\[\int_E |f|g \,dm \leq \sqrt{\mathcal{E}^\mu(f,f)}.\]
\((3)\) A Schr\"{o}dinger form \((\mathcal{E}^{\mu}, \mathcal{D}(\mathcal{E}^{\mu}))\) is supercritical if neither \((1)\) nor \((2)\) is satisfied. 
\end{definition}

We also say that the Schr\"{o}dinger operator \(\mathcal{L}^\mu\) is subcritical (resp. critical, supercritical) if \((\mathcal{E}^{\mu}, \mathcal{D}(\mathcal{E}^{\mu}))\) is subcritical (resp. critical, supercritical).

We set
\[\mathcal{H}_+^\mu:=\{h : 0<h<\infty,\  T_t^{\mu}h \leq h \text{\ for\ any\ }t>0, m\text{-almost everywhere}\}.\]
If \(\mathcal{H}_+^\mu\) is not empty, for \(h \in \mathcal{H}_+^\mu\), we define a Doob's \(h\)-transform \((\mathcal{E}^{\mu, h}, \mathcal{D}(\mathcal{E}^{\mu, h}))\) of the Schr\"{o}dinger form by
\begin{eqnarray*}
\mathcal{E}^{\mu,h}(f,g)&:=& \mathcal{E}^{\mu}(fh,gh),\\
\mathcal{D}(\mathcal{E}^{\mu,h})&:=& \{f\in L^2(E;h^2m) : fh\in \mathcal{D}(\mathcal{E}^{\mu})\}.
\end{eqnarray*}

An \(h\)-transform \((\mathcal{E}^{\mu,h}, \mathcal{D}(\mathcal{E}^{\mu,h}))\) is an irreducible regular Dirichlet form on \(L^2(E; h^2dm)\), and the following probabilistic characterizations of criticalities for Schr\"{o}dinger operators are known.
\begin{theorem}[{\cite[Theorem 2.13, Lemma 2.11]{TU23}}] \label{defsubcri2}
Let \((\mathcal{E}, \mathcal{D}(\mathcal{E}))\) be an irreducible regular Dirichlet form on \(L^2(E;m)\) and \(\mu\) be a signed smooth Radon measure making \((\mathcal{E}^{\mu}, \mathcal{D}(\mathcal{E}) \cap C_c(E))\) non-negative definite closable. Then the following are equivalent.
\begin{enumerate}
\item A Schr\"{o}dinger form \((\mathcal{E}^{\mu}, \mathcal{D}(\mathcal{E}^{\mu}))\) is subcritical.
\item \(\mathcal{H}_+^\mu\) is not empty and \((\mathcal{E}^{\mu,h}, \mathcal{D}(\mathcal{E}^{\mu,h}))\) on \(L^2(E;h^2m)\) is transient for some \(h\in \mathcal{H}_+^\mu\).
\item \(\mathcal{H}_+^\mu\) is not empty and \((\mathcal{E}^{\mu,h}, \mathcal{D}(\mathcal{E}^{\mu,h}))\) on \(L^2(E;h^2m)\) is transient for any \(h\in \mathcal{H}_+^\mu\).
\end{enumerate}
\end{theorem}

\begin{theorem}[{\cite[Theorem 2.13, Remark 2.14]{TU23}}] \label{defsubcri3}
Let \((\mathcal{E}, \mathcal{D}(\mathcal{E}))\) be an irreducible regular Dirichlet form on \(L^2(E;m)\) and \(\mu\) be a signed smooth Radon measure making \((\mathcal{E}^{\mu}, \mathcal{D}(\mathcal{E}) \cap C_c(E))\) non-negative definite closable. Then the following are equivalent.
\begin{enumerate}
\item A Schr\"{o}dinger form \((\mathcal{E}^{\mu}, \mathcal{D}(\mathcal{E}^{\mu}))\) is critical.
\item \(\mathcal{H}_+^\mu\) is not empty and \((\mathcal{E}^{\mu,h}, \mathcal{D}(\mathcal{E}^{\mu,h}))\) on \(L^2(E;h^2m)\) is recurrent for some \(h\in \mathcal{H}_+^\mu\).
\item \(\mathcal{H}_+^\mu\) is not empty and \((\mathcal{E}^{\mu,h}, \mathcal{D}(\mathcal{E}^{\mu,h}))\) on \(L^2(E;h^2m)\) is recurrent for any \(h\in \mathcal{H}_+^\mu\).
\end{enumerate}
\end{theorem}

By Theorem \ref{defsubcri2}, \ref{defsubcri3}, a Schr\"{o}dinger form \((\mathcal{E}^{\mu}, \mathcal{D}(\mathcal{E}^{\mu}))\) is supercritical if and only if \(\mathcal{H}_+^\mu\) is empty.

In general, it is not easy to check whether \(\mathcal{H}_+^\mu\) is empty or not. However, analytic criteria were obtained in \cite{T14, TU23}. We define the bottom of the spectrum \(\lambda(\mu)\) for the time-changed process of a process associated with \((\mathcal{E}^{\mu^+}, \mathcal{D}(\mathcal{E}^{\mu^+}))\) by \(\mu^-\) as follows.
\[\lambda(\mu):= \inf\{\mathcal{E}(f,f)+\int |f|^2 \, d\mu^+ : f\in \mathcal{D}(\mathcal{E})\cap C_c(E), \int_E |f|^2\, d\mu^- = 1\}.\]

Note that \(\mathcal{D}(\mathcal{E}^{\mu^+}) = \mathcal{D}(\mathcal{E}) \cap L^2(E;\mu^+)\) by \cite[Section 5.1]{CF12}.
\begin{theorem}[{\cite[Theorem 3.5]{TU23}}]\label{Theorem3.5_TU23}
Let \((\mathcal{E}, \mathcal{D}(\mathcal{E}))\) be an irreducible regular Dirichlet form on \(L^2(E;m)\) and \(\mu\) be a signed smooth Radon measure making \((\mathcal{E}^{\mu}, \mathcal{D}(\mathcal{E}) \cap C_c(E))\) non-negative definite closable. We assume that, for each compact set \(K\),
\[\sup_{f\in \mathcal{D}(\mathcal{E}^{\mu^+})} \frac{\int_K |f|\,dm}{\sqrt{\mathcal{E}^{\mu^+}(f,f)}} < \infty. \]
If \(\lambda(\mu)>1\), then \((\mathcal{E}^\mu, \mathcal{D}(\mathcal{E}^\mu))\) is subcritical.
\end{theorem}

For the bottom of the spectrum \(\gamma(\mu)\) where
\[\gamma(\mu):= \inf\{\mathcal{E}^\mu(f,f) : f\in \mathcal{D}(\mathcal{E})\cap C_c(E), \int_E |f|^2\, dm = 1\},\]
it is known that \(\lambda(\mu)\geq 1\) is equivalent to \(\gamma(\mu)\geq 0\). See \cite[Remark 3.6]{TU23} for details.

Takeda \cite{T14} proved the equivalence between probabilistic and analytic characterizations of criticalities for Schr\"{o}dinger operators perturbed by certain types of signed smooth measures under the following additional conditions.
\begin{enumerate}
\item[(SF)] Strong Feller property: For each \(t>0\), \(P_t (\mathcal{B}_b(E)) \subset C_b(E)\), where \(\mathcal{B}_b(E)\) is the set of all bounded Borel measurable functions on \(E\) and \(C_b(E)\) is the set of all bounded continuous functions on \(E\).
\end{enumerate}
By \((\)SF\()\) and the symmetry of the Dirichlet form, the following condition holds.

\begin{enumerate}
\item[(AC)] Absolute continuity condition: There exists a jointly measurable function \(p(t,x,y)\) on \((0,\infty)\times E \times E\) such that \(p(t,x,y)=p(t,y,x)\) and \(P_tf(x)=\int_E p(t,x,y)f(y)dm(y)\).
\end{enumerate}

The above \(p(t,x,y)\) is called a transition density of a stochastic process \(X\). In probability theory, a transition density also tends to be called a heat kernel even when \(\mathcal{L}\) is not necessarily the Laplace operator. In this case, we define the \(\alpha\)-order resolvent \(R_{\alpha}\) for \(\alpha>0\) by \(R_{\alpha}f(x)=\int_E r_{\alpha}(x,y)f(y)dm(y)\), where
\[r_{\alpha}(x,y) := \int_0^{\infty} e^{-\alpha t} p(t,x,y)\,dt.\]
Since \((\mathcal{E}, \mathcal{D}(\mathcal{E}))\) is transient, we can define \(r(x,y):=r_0(x,y):= \int_0^{\infty} p(t,x,y)\,dt.\) We note that \(T_tf=P_tf\) and \(G_{\alpha}f=R_{\alpha}f\) hold \(m\)-almost everywhere for \(f\in L^2(E;m)\).

\begin{definition}\label{defKato} Let \((\mathcal{E}, \mathcal{D}(\mathcal{E}))\) be a regular Dirichlet form on \(L^2(E;m)\) satisfying the condition (AC). We set \(R_{\alpha}\mu := \int_E r_{\alpha}(x,y)\, d\mu(y) \leq \infty\) for \(\mu \in \mathcal{S}\), and then we define the following classes of smooth measures.
\begin{enumerate}
\item A smooth measure \(\mu \in \mathcal{S}\) is called a {\it Kato class measure} if
\[\lim_{\alpha \to \infty} \|R_{\alpha}\mu\|_{\infty}=0.\]
Denote by \(\mathcal{K}:=\mathcal{K}(\mathcal{E})\) the set of all Kato class measures.

\item A smooth measure \(\mu \in \mathcal{S}\) is called a {\it local Kato class measure} if, for any compact set \(K\), the restriction \(1_{K}\, \mu\) of \(\mu\) to \(K\) is a Kato class measure. Denote by \(\mathcal{K}_{loc}:=\mathcal{K}_{loc}(\mathcal{E})\) the set of all local Kato class measures.

\item Suppose that \((\mathcal{E}, \mathcal{D}(\mathcal{E}))\) is transient. A smooth measure \(\mu \in \mathcal{K}\) is called a {\it Green-tight measure} with respect to \((\mathcal{E}, \mathcal{D}(\mathcal{E}))\), if for any \(\varepsilon>0,\) there exists a compact set \(K\) such that \(\|R(1_{K^c}\mu)\|_\infty <\varepsilon\) holds. Denote by \(\mathcal{K}_{\infty}(\mathcal{E})\) the set of all Green-tight measures with respect to \((\mathcal{E}, \mathcal{D}(\mathcal{E}))\).
\end{enumerate}
\end{definition}

\begin{remark}\, 
\begin{enumerate}
\item There are other types of definitions of a Kato class, including the original definition by Kato \cite{K72}. However most definitions are equivalent in the case of \(\Delta\). See \cite{AS82, KT07} for details.

\item The above \(\|\cdot \|_{\infty}\) in Definition \ref{defKato} is the essential supremum with respect to \(m\). However, \(\|R_\alpha \mu \|_{\infty}\) coincides with the essential supremum of \(R_\alpha \mu\) for \(\mu\) and also the capacity of the Dirichlet form. See \cite{ABM91, OTU26+} for details. 

\item It holds that \(\mathcal{K} \subset \mathcal{K}_{loc}.\) By the following Stollmann--Voigt's inequality, \(\mu \in \mathcal{K}_\infty\) is a Radon measure.

\item The above Green-tight measure is also called a Green-tight measure in the sense of Zhao \cite{Z92}. A smooth measure \(\mu \in \mathcal{S}\) is called a Green-tight measure in the sense of Chen \cite{C02}, if for any \(\varepsilon >0,\) there exist a Borel set \(K\) with \(\mu(K) <\infty\) and a constant \(\delta >0\) such that, for any Borel subset \(B \subset K\) satisfying \(\nu(B) \leq \varepsilon\) and \(\|R(1_{K^c \cup B}\mu)\|_{\infty}<\varepsilon.\) Under the assumption (SF), both classes of Green-tight measures coincide (\cite[Lemma 4.1]{KK17}).
\end{enumerate}
\end{remark}

The following inequality is called Stollmann--Voigt's inequality. This inequality is well-known under the condition (AC). However, we provide the following original proof without the condition (AC). 
\begin{theorem}[Stollmann--Voigt's inequality \cite{SV96}]
Let \((\mathcal{E}, \mathcal{D}(\mathcal{E}))\) be a regular Dirichlet form. For any \(\mu \in \mathcal{S}\) and \(f\in \mathcal{D}(\mathcal{E})\), it holds that
\begin{equation}
\int_E |f|^2\,d\mu \leq \|R_\alpha\mu\|_\infty \mathcal{E}_\alpha(f,f),\label{eq:SVineq}
\end{equation}
where \(R_\alpha\mu\) is defined by \(R_\alpha\mu(x) := \mathbb{E}_x[\int_0^\infty e^{-\alpha t}\,dA_t^\mu]\) for a PCAF \(A^\mu\) corresponding to \(\mu\).
\end{theorem}
\begin{proof}
Without loss of generality, we may assume \(\alpha =1\). First, we assume that \(\mu \in \mathcal{S}_{00}\) and \(f\in  \mathcal{D}(\mathcal{E})\) is bounded, where \(\mathcal{S}_{00}\) is the set of all smooth measures \(\mu\) satisfying \(R_1\mu \in \mathcal{D}(\mathcal{E}) \cap L^\infty(E;m)\) and \(\mu(E)<\infty\). See Appendix \ref{Appendix_PCAF} for the equivalent definition. We remark that \(\mathcal{E}_1(R_1\mu,u)=\int_E u\,d\mu\) holds for any \(u \in \mathcal{D}(\mathcal{E})\). See \cite[Section 2.3]{CF12} for details. Then we have
\begin{equation}
\int_E |f|^2\,d\mu = \mathcal{E}_1(|f|, R_1(|f|\mu)) \leq \sqrt{\mathcal{E}_1(f,f)} \sqrt{\mathcal{E}_1(R_1(|f|\mu),R_1(|f|\mu))}.
\label{eq:SVineq_1}
\end{equation}
By \cite[Proposition 1.2]{O26} and H\"{o}lder's inequality, we have
\begin{eqnarray}
\nonumber \mathcal{E}_1(R_1(|f|\mu),R_1(|f|\mu)) &=& \mathbb{E}_{m+\frac{\kappa}{2}+\frac{\nu_0}{2}}\left[\left( \int_0^\infty e^{-s} |f|(X_s)\,dA_s^{\mu} \right)^2\right]\\
\nonumber &\leq & \mathbb{E}_{m+\frac{\kappa}{2}+\frac{\nu_0}{2}}\left[\int_0^\infty e^{-s} |f|^2(X_s)\,dA_s^{\mu}\cdot\, \int_0^\infty e^{-s} \,dA_s^{\mu}\right] \\
\nonumber &=& \mathcal{E}_1(R_1(|f|^2\mu), R_1\mu)\\
\nonumber &=& \int_E |f|^2  R_1\mu \,d\mu \\
 &\leq & \|R_1\mu\|_\infty \, \int_E |f|^2 \,d\mu.
\label{eq:SVineq_2}
\end{eqnarray}
Here \(\kappa\) is a smooth measure called the killing measure, and \(\nu_0\) is the extended energy functional for a function \(\mathbb{E}_x[e^{-2s}1_{\partial}(X_{\zeta-})]\). We note that \(\nu_0\) is not a measure, but the measure-like notation is used in \cite{O26} since the subadditivity and the monotonicity hold. See \cite[Section 4]{O26} for details. 
Hence, by (\ref{eq:SVineq_1}) and (\ref{eq:SVineq_2}), we have
\[
\int_E |f|^2\,d\mu \leq \sqrt{\mathcal{E}_1(f,f)} \, \sqrt{\|R_1\mu\|_\infty} \,  \sqrt{\int_E |f|^2\,d\mu}\]
and so (\ref{eq:SVineq}) holds for \(\mu \in \mathcal{S}_{00}\) and a bounded function \(f\in \mathcal{D}(\mathcal{E})\). For any \(\mu \in \mathcal{S}\) and \(f\in   \mathcal{D}(\mathcal{E})\), by \cite[Theorem 2.3.15]{CF12}, there exists a nest \(\{F_k\}_k\) such that \(1_{F_k}\mu \in \mathcal{S}_{00}\), and we take \((-n)\vee f \wedge n\). Then, by using the monotonicity, we obtain (\ref{eq:SVineq}) for \(\mu \in \mathcal{S}\) and \(f\in   \mathcal{D}(\mathcal{E})\).
\end{proof}

We note that, for \(\mu^- \in \mathcal{K}(\mathcal{E}^{\mu^+})\), the closability of \((\mathcal{E}^\mu, \mathcal{D} \cap C_c(E))\) and \(\mathcal{D}(\mathcal{E}^\mu) =\mathcal{D}(\mathcal{E}^{\mu^+})\) follow from the Stollmann-Voigt inequality. Indeed, for a large \(\alpha\) satisfying \(\|R^{\mu^+}_\alpha \mu^-\|_\infty <1/2\), by the Stollmann--Voigt's inequality, we have
\begin{eqnarray*}
\mathcal{E}^\mu_\alpha(f,f) \geq \mathcal{E}^{\mu^+}_\alpha(f,f) -\|R^{\mu^+}_\alpha \mu^-\|_\infty \mathcal{E}^{\mu^+}_\alpha(f,f)= \frac{1}{2}\mathcal{E}^{\mu^+}_\alpha(f,f),
\end{eqnarray*}
and so \(\mathcal{E}^\mu_\alpha\) and \(\mathcal{E}^{\mu^+}_\alpha\) are comparable. We also note that, for \(\mu^- \in \mathcal{K}(\mathcal{E}^{\mu^+})\), an \(h\)-transform \((\mathcal{E}^{\mu,h}, \mathcal{D}(\mathcal{E}^{\mu,h}))\) is an irreducible regular Dirichlet form on \(L^2(E; h^2dm)\) (\cite[Lemma 2.6]{T14}). 

\begin{theorem}[{\cite[Theorem 5.19]{T14}}] \label{Takedaeigen}
For an irreducible regular Dirichlet form \((\mathcal{E}, \mathcal{D}(\mathcal{E}))\) on \(L^2(E;m)\) and a signed smooth Radon measure \(\mu \in \mathcal{K}_{loc}(\mathcal{E})-\mathcal{K}_{\infty}(\mathcal{E}^{\mu^+})\), the following holds.
\begin{enumerate}
\item \((\mathcal{E}^{\mu}, \mathcal{D}(\mathcal{E}^{\mu}))\) is subcritical if and only if \(\lambda(\mu)>1\)
\item \((\mathcal{E}^{\mu}, \mathcal{D}(\mathcal{E}^{\mu}))\) is critical if and only if \(\lambda(\mu)=1\)
\item \((\mathcal{E}^{\mu}, \mathcal{D}(\mathcal{E}^{\mu}))\) is supercritical if and only if \(\lambda(\mu)<1\)
\end{enumerate}
\end{theorem}

At the end of this section, we discuss the relationship between the notion of subcriticality used in \cite{S26, M86} and subcriticality used in this paper.
\begin{proposition}\label{defsubcri} 
Let \((\mathcal{E}, \mathcal{D}(\mathcal{E}))\) be an irreducible regular Dirichlet form on \(L^2(E;m)\) and \(\mu\) be a signed smooth Radon measure making \((\mathcal{E}^{\mu}, \mathcal{D}(\mathcal{E}) \cap C_c(E))\) non-negative definite closable. Then the Schr\"{o}dinger form \((\mathcal{E}^{\mu}, \mathcal{D}(\mathcal{E}^{\mu}))\) is subcritical if, for any non-negative \(W \in C_c(E)\), there exists \(\delta >0\) such that
\begin{equation}
\delta \int_E |f|^2 W\, dm \leq \mathcal{E}^\mu(f,f) \label{eq:defsubcri1}
\end{equation}
holds for any \(f\in  \mathcal{D}_e(\mathcal{E}^{\mu}).\)

In addition, if we assume the condition \((SF)\) and \(\mu \in \mathcal{K}_{loc}-\mathcal{K}\), then \((\ref{eq:defsubcri1})\) is equivalent to the subcriticality. 
\end{proposition}

\begin{proof}
We assume (\ref{eq:defsubcri1}) for any \(W \in C_c(E)\). Set \(G^\mu W := \int_0^\infty T_t^\mu W\,dt \leq \infty\). By the same argument as \cite[Lemma 2.1.4 (ii)]{CF12}, we have
\[\sup_{f\in \mathcal{D}(\mathcal{E}^\mu)} \frac{\int_E |f| W\,dm}{\sqrt{\mathcal{E}^\mu(f,f)}} = \sqrt{\int_E W G^\mu W \,dm }\leq \infty.\]
By (\ref{eq:defsubcri1}), we have
\begin{eqnarray*}
0\leq \sup_{f\in \mathcal{D}(\mathcal{E}^\mu)} \frac{\int_E |f| W\,dm}{\sqrt{\mathcal{E}^\mu(f,f)}} \leq \sup_{f\in \mathcal{D}(\mathcal{E}^\mu)} \frac{\sqrt{\int_E |f|^2 W\,dm}\,\sqrt{\int_E W\,dm}}{\sqrt{\mathcal{E}^\mu(f,f)}} \leq \frac{\sqrt{\int_E W\,dm}}{\sqrt{\delta}} <\infty,
\end{eqnarray*}
so \(0 < \langle W, G^\mu W\rangle_m <\infty\) holds. Similarly to \cite[Lemma 2.1.4 (i)]{CF12}, we have \(S_t^\mu W :=\int_0^t T_s^\mu W ds \in \mathcal{D}(\mathcal{E}^\mu)\) and \(\{S_n^\mu W\}_n \subset \mathcal{D}(\mathcal{E}^\mu)\) is an \(\mathcal{E}^\mu\)-Cauchy sequence. Since \(S^\mu_n W\) converges to \(G^\mu W\) \(m\)-almost everywhere, we obtain \(G^\mu W \in \mathcal{D}_e(\mathcal{E}^\mu)\) and so \(G^\mu W <\infty\) \(m\)-almost everywhere. Moreover it holds that \(0<G^\mu W <\infty\) \(m\)-almost everywhere. Indeed, if \(m(A)>0\) for \(A:=\{G^\mu W = 0\}\), then \(T_t^\mu W =0 \) for almost every \(t\) and \(G^\mu_1 W = 0\) on \(A\). Hence, for any strictly positive function \(f\in \mathcal{D}(\mathcal{E}^\mu)\) whose support is included in \(A\), we have \(0=\langle G_1^\mu W, f \rangle = \langle W, G_1^\mu f \rangle \) and this contradicts to \(G_1^\mu f >0\). Since \(T_t^\mu G^\mu W = \int_t^\infty T_s^\mu W\,ds \leq G^\mu W \), we have \(G^\mu W \in \mathcal{H}_+^\mu \). Hence \((\mathcal{E}^{\mu}, \mathcal{D}(\mathcal{E}^{\mu}))\) is critical or subcritical. If \((\mathcal{E}^{\mu}, \mathcal{D}(\mathcal{E}^{\mu}))\) is critical, then there exists a strictly positive continuous function \(h\in \mathcal{D}_e(\mathcal{E}^{\mu})\) such that \(\mathcal{E}^\mu(h,h)=0\) and so \(h=0\) holds by (\ref{eq:defsubcri1}). This is a contradiction.

We assume the condition \((SF)\) and \(\mu \in \mathcal{K}_{loc}-\mathcal{K}\) and the subcriticality for the Schr\"{o}dinger form \((\mathcal{E}^{\mu}, \mathcal{D}(\mathcal{E}^{\mu}))\). Then, by \cite[Theorem 3.7]{T14}, there exists a strictly positive continuous function \(g\) such that
\begin{equation*}
\int_E |f|^2 g\, dm \leq \mathcal{E}^\mu(f,f)
\end{equation*}
holds for any \(f\in  \mathcal{D}_e(\mathcal{E}^{\mu}).\) For any \(W\in C_c(E)\), setting \(\delta := 1/\|\frac{W}{g}\|_\infty \), \((\ref{eq:defsubcri1})\) holds.
\end{proof}

\section{Criticality and subcriticality of a subordinated Schr\"{o}dinger form}\label{sec:sub_shr}
Continuing from the previous section, assume that \((\mathcal{E}, \mathcal{D}(\mathcal{E}))\) is an irreducible regular Dirichlet form and that \(\mu\) is a signed smooth Radon measure such that \((\mathcal{E}^{\mu}, \mathcal{D}(\mathcal{E})\cap C_c(E))\) is non-negative definite closable symmetric form. Denote by \((\mathcal{E}^{\mu}, \mathcal{D}(\mathcal{E}^\mu))\) its Schr\"{o}dinger form 
, \(\mathcal{L}^{\mu}=\mathcal{L}-\mu\) its associated non-positive self-adjoint operator, and \(\{T_t^\mu\}_t\) its strongly continuous contraction semigroup on \(L^2(E;m)\).

In this section, we introduce a subordinated Schr\"{o}dinger form and its criticality, subcriticality, and supercriticality in order to consider a subordinated Schr\"{o}dinger operator such as \(-(-\mathcal{L}+\mu)^{\beta}\).

\begin{definition}[{c.f.\cite[Definition 21.4]{S99}}]
A subordinator is a \([0, \infty)\)-valued L\'{e}vy process. More precisely, a \([0,\infty]\)-valued stochastic process \(\{S_t\}_t\) is a subordinator if, almost surely, it is right continuous and has left limits, \(S_0=0\), \(S_t \geq 0\), the distribution of \(S_{t+s}-S_t\) does not depend on \(s\) for any \(t\geq 0\), and \(S_{t_{i+1}}-S_{t_i}\) are independent of each other for \(0\leq t_1 <t_2 < \cdots <t_n\).
\end{definition}

Let \(\{\eta_t\}_t\) be a strongly continuous contraction semigroup of a probability measure of a subordinator \(S\), that is, \(\mathbb{P}(S_t \in ds)=\eta_t(ds)\). Then, by the L\'{e}vy-Khintchine formula (See, for example, \cite[Theorem 8.1]{S99}), \(\{\eta_t\}_t\) is characterized by
\begin{eqnarray*}
\int_0^\infty e^{-\lambda s}\eta_t(ds) = e^{-t \Phi(\lambda)}\\
\Phi(\lambda)= b\lambda + \int_0^\infty (1-e^{-\lambda s})\, \nu(ds),
\end{eqnarray*}
where \(b \geq 0\) is a constant called the drift coefficient and \(\nu\) is a Radon measure on \((0, \infty)\) satisfying \(\int_0^\infty (1\wedge s)\,d\nu(s) < \infty\), called the  L\'{e}vy  measure. \(\Phi\) is called a {\it Bernstein function.}

For a subordinator with a Bernstein function \(\Phi\) and semigroup \(\{\eta_t\}_t\), we define a subordinated semigroup \(\{T_t^{\mu, \Phi}\}_t\) by
\[T_t^{\mu, \Phi}f :=\int_0^\infty T_s^{\mu}f\, \eta_t(ds) \]
for \(f\in L^2(E;m).\) Then \(\{T_t^{\mu, \Phi}\}_t\) is also a strongly continuous contraction semigroup and its generator \(\mathcal{L}^{\mu, \Phi}\) is given by 
\[\mathcal{L}^{\mu, \Phi} = -\Phi(-\mathcal{L}^{\mu}) = b\mathcal{L}^{\mu} - \int_0^\infty (I-T_s^{\mu})\, \nu(ds)\]
and \(\mathcal{D}(\mathcal{L}^{\mu, \Phi}) \supset \mathcal{D}(\mathcal{L}^{\mu}).\) We call \(\mathcal{L}^{\mu, \Phi}\) a subordinated Schr\"{o}dinger operator. Hence, there exists a closed form \((\mathcal{E}^{\mu, \Phi}, \mathcal{D}(\mathcal{E}^{\mu, \Phi}))\) on \(L^2(E;m)\) by \(\mathcal{E}^{\mu, \Phi}(f,g):=\langle -\mathcal{L}^{\mu, \Phi}f,g \rangle_m \) and \(\mathcal{D}(\mathcal{E}^{\mu, \Phi}) = \mathcal{D}(\sqrt{-\mathcal{L}^{\mu, \Phi}})\). Since this is a closed form on \(L^2(E;m)\), as in the case of Dirichlet forms, we define \(\mathcal{E}_\alpha ^{\mu, \Phi}(f,g):= \mathcal{E}^{\mu, \Phi}(f,g) +\alpha \langle f,g \rangle_m \).

We call \((\mathcal{E}^{\mu, \Phi}, \mathcal{D}(\mathcal{E}^{\mu, \Phi}))\) on \(L^2(E;m)\) {\it a subordinated Schr\"{o}dinger form}. See \cite{S98, P52} for details on a subordination of a closed form.

\^Okura obtained the representation of Dirichlet forms of subordinated Markov processes. He used only spectral analysis to obtain the representation without using the Markov property for an original Dirichlet form, and hence we obtain the same representation for a subordinated Schr\"{o}dinger form \((\mathcal{E}^{\mu, \Phi}, \mathcal{D}(\mathcal{E}^{\mu, \Phi}))\). See \cite{O02, AR05} for details on a subordination of a Dirichlet form.

\begin{theorem}[{cf. \cite[Theorem 2.1]{O02}}]\label{Okura}
Let \((\mathcal{E}, \mathcal{D}(\mathcal{E}))\) be an irreducible regular Dirichlet form  on \(L^2(E;m)\), \(\mu\) be a signed smooth Radon measure making \((\mathcal{E}^{\mu}, \mathcal{D}(\mathcal{E}) \cap C_c(E))\) non-negative definite closable, and \(\Phi\) be a Bernstein function. Then it holds that \(\mathcal{D}(\mathcal{E}^{\mu}) \subset \mathcal{D}(\mathcal{E}^{\mu, \Phi})\) and, for \(f,g \in \mathcal{D}(\mathcal{E}^{\mu})\),
\[\mathcal{E}^{\mu, \Phi}(f,g)=b\, \mathcal{E}^{\mu}(f,g) + \int_0^\infty \langle f-T_s^\mu f, g \rangle_m d\nu(s).\]
Moreover, for any \(\mathcal{E}^{\mu}_1\)-dense subspace \(\mathcal{C}^\mu\) of \(\mathcal{D}(\mathcal{E}^{\mu})\), \(\mathcal{C}^\mu\) is also \(\mathcal{E}^{\mu, \Phi}_1\)-dense in \(\mathcal{D}(\mathcal{E}^{\mu, \Phi})\).
\end{theorem}

\begin{example}
We consider a Bernstein function \(\Phi\) with \(b=0\) and \(d\nu(s) = c_{\alpha} s^{-\alpha/2 -1}\, ds\) for \(0< \alpha <2\), then this is called an \(\alpha/2\)-stable subordinator and it holds that \(\Phi(\lambda)= \lambda^{\alpha/2}\).
For example, if \(\mathcal{L} = \Delta\) on \(\mathbb{R}^d\), \(d\mu = V dx \in \mathcal{S}_R-\mathcal{S}_R\) satisfying \(\mathcal{H}_+^\mu \not = \emptyset\), then we have \(\mathcal{L}^{\mu, \Phi} = -(-\Delta + V)^{\alpha/2}\) and
\[\mathcal{E}^{\mu, \Phi}(f,g)= \int_0^\infty \int_{\mathbb{R}^d} (f(x)-\mathbb{E}_x[e^{\int_0^tV(X_s)\,ds} f(X_t)])g(x)\, dx\, d\nu(s),\]
where \(X\) is a \(d\)-dimensional Brownian motion.
\end{example}

By Theorem \ref{Okura}, \(\mathcal{D}(\mathcal{E}^\mu)\) is \(\mathcal{E}^{\mu, \Phi}_1\)-dense in \(\mathcal{D}(\mathcal{E}^{\mu, \Phi})\). For any \(f\in \mathcal{D}(\mathcal{E}^\mu)\), it holds that \(|f| \in \mathcal{D}(\mathcal{E}^\mu)\) and by Theorem \ref{Okura} again, we have
\begin{eqnarray*}
\mathcal{E}^{\mu, \Phi}(|f|, |f|) &=& b\, \mathcal{E}^{\mu}(|f|,|f|) + \int_0^\infty \langle |f|-T_s^\mu |f|, |f| \rangle_m d\nu(s) \\
&\leq & b\, \mathcal{E}^{\mu}(f,f) + \int_0^\infty \langle f-T_s^\mu f, f \rangle_m d\nu(s)\\
&=& \mathcal{E}^{\mu, \Phi}(f, f).
\end{eqnarray*}
Hence by \cite[Lemma 1.3.4]{D89}, for any \(f \in \mathcal{D}(\mathcal{E}^{\mu, \Phi})\), it holds that \(|f|\in \mathcal{D}(\mathcal{E}^{\mu, \Phi})\) and \(\mathcal{E}^{\mu, \Phi}(|f|,|f|) \leq \mathcal{E}^{\mu, \Phi}(f,f)\). By \cite[Proposition 2]{Sc99}, 
\((\mathcal{E}^{\mu, \Phi}, \mathcal{D}(\mathcal{E}^{\mu, \Phi}))\) satisfies the Fatou property, that is, for any \(f, f_n \in \mathcal{D}(\mathcal{E}^{\mu, \Phi})\) satisfying \(\sup_n \mathcal{E}^{\mu, \Phi}(f_n,f_n) <\infty\) and \(f_n\) converges to \(f\) \(m\)-almost everywhere, then \(\mathcal{E}^{\mu, \Phi}(f,f) \leq \varliminf_n \mathcal{E}^{\mu, \Phi}(f_n,f_n) \). Hence we can define the extended space \(\mathcal{D}_e(\mathcal{E}^{\mu, \Phi})\) by the set of all \(m\)-measurable functions \(f\) satisfying \(|f|<\infty\) \(m\)-almost everywhere and possessing an approximating sequence \(\{f_n\}_n \subset \mathcal{D}(\mathcal{E}^{\mu, \Phi})\) such that \(\mathcal{E}^{\mu, \Phi}(f_n-f_m, f_n-f_m)\to 0\) as \(m,n \to \infty\) and \(f_n \to f\) \(m\)-almost everywhere, and we can define \(\mathcal{E}^{\mu, \Phi}(f,f):=\lim_{n\to \infty} \mathcal{E}^{\mu, \Phi}(f_n,f_n)\). Indeed, for an \(\mathcal{E}^{\mu, \Phi}\)-Cauchy sequence \(\{f_n\}_n \subset \mathcal{D}(\mathcal{E}^{\mu, \Phi})\) converging to \(f\in \mathcal{D}(\mathcal{E}^{\mu, \Phi})\) \(m\)-almost everywhere, then \(\{f_n-f_m\}_m \subset \mathcal{D}( \mathcal{E}^{\mu, \Phi})\) is also an \(\mathcal{E}^{\mu, \Phi}\)-Cauchy sequence for each \(n\) and so we have 
\[\sqrt{\mathcal{E}^{\mu, \Phi}(f_n,f_n)} \leq \sqrt{\mathcal{E}^{\mu, \Phi}(f_n-f,f_n-f)} + \sqrt{\mathcal{E}^{\mu, \Phi}(f,f)} \leq \varliminf_m \sqrt{\mathcal{E}^{\mu, \Phi}(f_n-f_m,f_n-f_m)} + \sqrt{\mathcal{E}^{\mu, \Phi}(f,f)}\]
and by letting \(n\) tend to infinity, \(\varlimsup_n \mathcal{E}^{\mu, \Phi}(f_n,f_n) \leq \mathcal{E}^{\mu, \Phi}(f,f)\). Combining this with the Fatou property, we obtain \(\lim_n \mathcal{E}^{\mu, \Phi}(f_n,f_n) = \mathcal{E}^{\mu, \Phi}(f,f)\) and so the definition of the extended space \(\mathcal{D}_e(\mathcal{E}^{\mu, \Phi})\) is well-defined. Similarly, the limit  \(\mathcal{E}^{\mu, \Phi}(f,f)\) for \(f\in \mathcal{D}(\mathcal{E}^{\mu, \Phi})\) is independent of the choice of an approximating sequence. We define the notions of subcriticality and criticality in a similar way to \cite{TU23} as follows.
\begin{definition}
Let \((\mathcal{E}, \mathcal{D}(\mathcal{E}))\) be an irreducible regular Dirichlet form  on \(L^2(E;m)\), \(\mu\) be a signed smooth Radon measure making \((\mathcal{E}^{\mu}, \mathcal{D}(\mathcal{E}) \cap C_c(E))\) non-negative definite closable, and \(\Phi\) be a Bernstein function. We define the notions of criticality for subordinated Schr\"{o}dinger forms as follows.
\begin{enumerate}
\item  A subordinated Schr\"{o}dinger form \((\mathcal{E}^{\mu, \Phi}, \mathcal{D}(\mathcal{E}^{\mu, \Phi}))\) is subcritical if there exists a strictly positive bounded function \(g \in L^1(E;m)\) satisfying that for any \(f\in \mathcal{D}_e(\mathcal{E}^{\mu, \Phi})\),
\begin{equation} \int_E |f|g \,dm \leq \sqrt{\mathcal{E}^{\mu, \Phi}(f,f)}. \label{eq:sub_subcri} \end{equation}
\item A subordinated Schr\"{o}dinger form \((\mathcal{E}^{\mu, \Phi}, \mathcal{D}(\mathcal{E}^{\mu, \Phi}))\) is critical if there exists a strictly positive function \(h\in  \mathcal{D}_e(\mathcal{E}^{\mu, \Phi})\) satisfying \(\mathcal{E}^{\mu, \Phi}(h,h)=0.\)
\item  A subordinated Schr\"{o}dinger form \((\mathcal{E}^{\mu, \Phi}, \mathcal{D}(\mathcal{E}^{\mu, \Phi}))\) is supercritical if neither \((1)\) nor \((2)\) is satisfied.
\end{enumerate}
\end{definition}

We say that \(\mathcal{L}^{\mu, \Phi}\) is subcritical (resp. critical, supercritical) if \((\mathcal{E}^{\mu, \Phi}, \mathcal{D}(\mathcal{E}^{\mu, \Phi}))\) is subcritical (resp. critical, supercritical).

To characterize criticalities for a subordinated Schr\"{o}dinger form from a perspective of probability theory, we set
\[\mathcal{H}^{\mu, \Phi}_+:=\{h : 0<h<\infty,\  T_t^{\mu, \Phi}h \leq h \text{\ for\ any\ }t>0, m\text{-almost everywhere}\}.\]

Since, for \(h \in \mathcal{H}_+^\mu\),
\[T_t^{\mu, \Phi}h =\int_0^\infty T_s^{\mu}h \, \eta_t(ds) \leq h\, \int_0^\infty \eta_t(ds) =h\]
 it holds that \(\mathcal{H}_+^\mu \subset \mathcal{H}^{\mu, \Phi}_+ \).

For any \(h \in \mathcal{H}^{\mu, \Phi}_+\), we define an \(h\)-transform \((\mathcal{E}^{\mu,\Phi, h}, \mathcal{D}(\mathcal{E}^{\mu,\Phi,h}))\) by
\begin{eqnarray*}
\mathcal{E}^{\mu,\Phi,h}(f,g)&:=& \mathcal{E}^{\mu, \Phi}(fh,gh)\\
\mathcal{D}(\mathcal{E}^{\mu,\Phi,h})&:=& \{f\in L^2(E;h^2m) : fh\in \mathcal{D}(\mathcal{E}^{\mu, \Phi})\}.
\end{eqnarray*}

Then the corresponding self-adjoint operator \(\mathcal{L}^{\mu, \Phi, h}\) is represented by \(\mathcal{L}^{\mu, \Phi, h}f = \frac{1}{h}\mathcal{L}^{\mu, \Phi}(fh)\) for \(f \in \mathcal{D}(\mathcal{L}^{\mu, \Phi, h})= \{f \in L^2(E;h^2m) : fh \in \mathcal{D}(\mathcal{L}^{\mu, \Phi})\}\), and the corresponding semigroup \(\{T_t^{\mu, \Phi, h}\}_t\) is \(T_t^{\mu, \Phi, h}f = \frac{1}{h}T_t^{\mu, \Phi}(fh)\) for \(f\in L^2(E;h^2 m)\).

An \(h\)-transform of a subordinated Schr\"{o}dinger form \((\mathcal{E}^{\mu,\Phi, h}, \mathcal{D}(\mathcal{E}^{\mu,\Phi,h}))\) is a Dirichlet form on \(L^2(E;h^2m)\). Indeed, for any \(f \in \mathcal{D}(\mathcal{E}^{\mu,\Phi,h})\) with \(0\leq f \leq 1\), we have \(fh \in \mathcal{D}(\mathcal{E}^{\mu,\Phi})\) and \[0\leq T_t^{\mu, \Phi, h}f =\frac{1}{h} T_t^{\mu, \Phi}(hf) \leq \frac{1}{h} T_t^{\mu, \Phi}h \leq 1.\]

We consider the following condition.\\
(IB): A Bernstein function \(\Phi\) satisfies  either \(b>0\) or \(\inf{{\rm supp}(\nu)}=0\).
\begin{proposition}\label{irr}
Let \((\mathcal{E}, \mathcal{D}(\mathcal{E}))\) be an irreducible regular Dirichlet form on \(L^2(E;m)\) and \(\mu\) be a signed smooth Radon measure making \((\mathcal{E}^{\mu}, \mathcal{D}(\mathcal{E}) \cap C_c(E))\) non-negative definite closable. Then, under the condition \rm{(IB)}, \((\mathcal{E}^{\mu,\Phi}, \mathcal{D}(\mathcal{E}^{\mu,\Phi}))\) is irreducible. Moreover \((\mathcal{E}^{\mu,\Phi, h}, \mathcal{D}(\mathcal{E}^{\mu, \Phi,h}))\) is an irreducible Dirichlet form on \(L^2(E;h^2m)\) for any \(h\in \mathcal{H}_+^{\mu, \Phi}\).
\end{proposition}

\begin{proof}
We take a \(\{T_t^{\mu, \Phi}\}_t\)-invariant set \(A\), then, as in \cite[Theorem 1.6.1]{FOT11}, this is equivalent to \(1_A f, 1_{A^c}g \in \mathcal{D}(\mathcal{E}^{\mu, \Phi})\) and \(\mathcal{E}^{\mu, \Phi}(1_A f, 1_{A^c}g)=0\) for any \(f,g\in \mathcal{D}(\mathcal{E}^{\mu, \Phi})\). For any non-negative functions \(f, g\in \mathcal{D}(\mathcal{E}^{\mu})\), we have
 \[\mathcal{E}^{\mu}(1_Af, 1_{A^c}g)= \lim_{t\searrow 0}\frac{1}{t}\langle1_Af - T_t^\mu(1_Af), 1_{A^c}g \rangle_m =  -\lim_{t\searrow 0}\frac{1}{t}\langle T_t^\mu(1_Af), 1_{A^c}g \rangle_m \leq 0.\]
Similarly, we have \(\mathcal{E}^{\mu, \Phi}(1_Af, 1_{A^c}g) \leq 0\). By Theorem \ref{Okura}, for any non-negative functions \(f, g\in \mathcal{D}(\mathcal{E}^{\mu})\), we have
\begin{eqnarray}
\nonumber 0=\mathcal{E}^{\mu, \Phi}(1_Af, 1_{A^c}g) &=& b\, \mathcal{E}^{\mu}(1_Af, 1_{A^c}g) + \int_0^\infty \langle 1_Af-T_t^{\mu}(1_Af), 1_{A^c} g \rangle_m \,d\nu(t)\\
&=&  b\, \mathcal{E}(1_Af, 1_{A^c}g) - \int_0^\infty \langle T_t^{\mu}(1_Af), 1_{A^c} g \rangle_m\,d\nu(t). \label{eq:irr_1}
\end{eqnarray}
Since each term in \((\ref{eq:irr_1})\) are non positive, if \(b>0\), it holds that \(\mathcal{E}(1_Af, 1_{A^c}g) = 0\), and if \(\inf{{\rm supp}(\nu)}=0\), it holds that
\[\mathcal{E}(1_Af, 1_{A^c}g)  = \lim_{\substack{t\searrow 0,\\ t \in  \inf{{\rm supp}(\nu)}}}\frac{1}{t}\langle1_Af - T_t^\mu(1_Af), 1_{A^c}g \rangle_m = 0.\]
For \(f\in \mathcal{D}(\mathcal{E}^{\mu})\), by considering \(f=f_+-f_-,\) it holds that \(\mathcal{E}(1_Af, 1_{A^c}f)=0\). By the irreducibility of \((\mathcal{E}, \mathcal{D}(\mathcal{E}))\), \(m(A)=0\) or \(m(A^c)=0\) holds. Hence \((\mathcal{E}^{\mu,\Phi}, \mathcal{D}(\mathcal{E}^{\mu,\Phi}))\) is irreducible.

An \(h\)-transform preserves irreducibility, so \((\mathcal{E}^{\mu,\Phi, h}, \mathcal{D}(\mathcal{E}^{\mu, \Phi,h}))\) is an irreducible Dirichlet form on \(L^2(E;h^2m)\) for any \(h\in \mathcal{H}_+^{\mu, \Phi}\).
\end{proof}

Similarly to \cite{TU23}, we provide probabilistic characterizations for criticalities as follows.

\begin{theorem}\label{subodinated_subcri}
Let \((\mathcal{E}, \mathcal{D}(\mathcal{E}))\) be an irreducible regular Dirichlet form  on \(L^2(E;m)\), \(\mu\) be a signed smooth Radon measure making \((\mathcal{E}^{\mu}, \mathcal{D}(\mathcal{E}) \cap C_c(E))\) non-negative definite closable, and \(\Phi\) be a Bernstein function satisfying the condition \rm{(IB)}. Then the following are equivalent.
\begin{enumerate}
\item A subordinated Schr\"{o}dinger form \((\mathcal{E}^{\mu, \Phi}, \mathcal{D}(\mathcal{E}^{\mu, \Phi}))\) is  subcritical.
\item \(\mathcal{H}_+^{\mu, \Phi}\) is not empty and \((\mathcal{E}^{\mu,\Phi, h}, \mathcal{D}(\mathcal{E}^{\mu, \Phi,h}))\) on \(L^2(E;h^2m)\) is transient for some \(h\in \mathcal{H}_+^{\mu, \Phi}\).
\item \(\mathcal{H}_+^{\mu, \Phi}\) is not empty and \((\mathcal{E}^{\mu,\Phi, h}, \mathcal{D}(\mathcal{E}^{\mu, \Phi,h}))\) on \(L^2(E;h^2m)\) is transient for any \(h\in \mathcal{H}_+^{\mu, \Phi}\).
\end{enumerate}
\end{theorem}

\begin{theorem}\label{subodinated_cri}
Let \((\mathcal{E}, \mathcal{D}(\mathcal{E}))\) be an irreducible regular Dirichlet form  on \(L^2(E;m)\), \(\mu\) be a signed smooth Radon measure making \((\mathcal{E}^{\mu}, \mathcal{D}(\mathcal{E}) \cap C_c(E))\) non-negative definite closable, and \(\Phi\) be a Bernstein function satisfying the condition \rm{(IB)}. Then the following are equivalent.
\begin{enumerate}
\item A subordinated Schr\"{o}dinger form \((\mathcal{E}^{\mu, \Phi}, \mathcal{D}(\mathcal{E}^{\mu, \Phi}))\) is  critical.
\item \(\mathcal{H}_+^{\mu, \Phi}\) is not empty and \((\mathcal{E}^{\mu,\Phi, h}, \mathcal{D}(\mathcal{E}^{\mu, \Phi,h}))\) on \(L^2(E;h^2m)\) is recurrent for some \(h\in \mathcal{H}_+^{\mu, \Phi}\).
\item \(\mathcal{H}_+^{\mu, \Phi}\) is not empty and \((\mathcal{E}^{\mu,\Phi, h}, \mathcal{D}(\mathcal{E}^{\mu, \Phi,h}))\) on \(L^2(E;h^2m)\) is recurrent for any \(h\in \mathcal{H}_+^{\mu, \Phi}\).
\end{enumerate}
\end{theorem}

\begin{theorem}\label{subodinated_supercri}
Let \((\mathcal{E}, \mathcal{D}(\mathcal{E}))\) be an irreducible regular Dirichlet form  on \(L^2(E;m)\), \(\mu\) be a signed smooth Radon measure making \((\mathcal{E}^{\mu}, \mathcal{D}(\mathcal{E}) \cap C_c(E))\) non-negative definite closable, and \(\Phi\) be a Bernstein function satisfying the condition \rm{(IB)}. Then the following are equivalent.
\begin{enumerate}
\item A subordinated Schr\"{o}dinger form \((\mathcal{E}^{\mu, \Phi}, \mathcal{D}(\mathcal{E}^{\mu, \Phi}))\) is  supercritical.
\item \(\mathcal{H}_+^{\mu, \Phi}\) is empty.
\end{enumerate}
\end{theorem}

To prove these three theorems, we need the following lemma.
\begin{lemma}\label{Fe_lemma}
Let \((\mathcal{E}, \mathcal{D}(\mathcal{E}))\) be an irreducible regular Dirichlet form  on \(L^2(E;m)\), \(\mu\) be a signed smooth Radon measure making \((\mathcal{E}^{\mu}, \mathcal{D}(\mathcal{E}) \cap C_c(E))\) non-negative definite closable, and \(\Phi\) be a Bernstein function satisfying the condition \rm{(IB)}. For \(h \in \mathcal{H}_+^{\mu, \Phi}\), it holds that \(\mathcal{D}_e(\mathcal{E}^{\mu,\Phi, h}) =\{f : fh \in \mathcal{D}_e(\mathcal{E}^{\mu,\Phi})\} \), and \(\mathcal{E}^{\mu, \Phi, h}(f,f) = \mathcal{E}^{\mu, \Phi}(fh, fh)\) for \(f\in \mathcal{D}_e(\mathcal{E}^{\mu,\Phi, h})\).
\end{lemma}
\begin{proof}
For \(f \in \mathcal{D}_e(\mathcal{E}^{\mu,\Phi, h})\), we take an \(\mathcal{E}^{\mu,\Phi, h}\)-Cauchy sequence \(\{f_n\}_n \subset \mathcal{D}(\mathcal{E}^{\mu,\Phi, h})\) converging to \(f\) \(m\)-almost everywhere. Then \(f_n h \in \mathcal{D}(\mathcal{E}^{\mu,\Phi})\), \(\mathcal{E}^{\mu, \Phi}(f_nh-f_m h, f_nh-f_m h) = \mathcal{E}^{\mu, \Phi, h}(f_n-f_m, f_n-f_m)\) converges to \(0\) as \(n,m\) go to \(\infty\), and \(f_n h \) converges to \(fh\) \(m\)-almost everywhere. Hence \(fh \in \mathcal{D}_e(\mathcal{E}^{\mu,\Phi})\) and 
\[\mathcal{E}^{\mu,\Phi, h}(f, f) = \lim_{n \to \infty}\mathcal{E}^{\mu,\Phi, h}(f_n, f_n)=  \lim_{n \to \infty} \mathcal{E}^{\mu,\Phi}(f_nh, f_nh)= \mathcal{E}^{\mu,\Phi}(fh, fh).\]
Conversely, for \(f\) satisfying \(fh \in \mathcal{D}_e(\mathcal{E}^{\mu,\Phi})\), we take an \(\mathcal{E}^{\mu,\Phi}\)-Cauchy sequence \(\{u_n\}_n \subset \mathcal{D}(\mathcal{E}^{\mu,\Phi})\) converging to \(fh\) \(m\)-almost everywhere. Then \(\{\frac{u_n}{h}\}_n \subset \mathcal{D}(\mathcal{E}^{\mu,\Phi,h})\) is an \(\mathcal{E}^{\mu,\Phi,h}\)-Cauchy sequence and \(\frac{u_n}{h}\) converges to \(f\), so \(f\in \mathcal{D}_e(\mathcal{E}^{\mu,\Phi, h}).\)
\end{proof}

\begin{proof}[Proof of Theorem \ref{subodinated_subcri}]
We assume \((1)\) and take a strictly positive bounded function \(g \in L^1(E;m)\) satisfying \((\ref{eq:sub_subcri})\). Then \((\mathcal{D}_e(\mathcal{E}^{\mu, \Phi}), \mathcal{E}^{\mu, \Phi})\) is a Hilbert space compactly embedded in \(L^1(E;gm)\). By the Riesz representation theorem, there exists a unique function \(G^{\mu,\Phi} g \in \mathcal{D}_e(\mathcal{E}^{\mu,\Phi})\) such that \(\mathcal{E}^{\mu, \Phi}(G^{\mu,\Phi} g, f) = \langle f, g \rangle_m\) for any \(f\in \mathcal{D}_e(\mathcal{E}^{\mu,\Phi})\). For \(\tilde{G}^{\mu, \Phi}g := \int_0^\infty T_t^{\mu, \Phi}g\,dt \leq \infty\), it holds that \(T_t^{\mu, \Phi} \tilde{G}^{\mu,\Phi} g = \int_t^\infty T_s^{\mu, \Phi}g\,ds \leq \tilde{G}^{\mu,\Phi} g\) and so,for any \(f\in \mathcal{D}(\mathcal{E}^{\mu,\Phi})\) we have
\begin{equation}
\mathcal{E}^{\mu, \Phi}(\tilde{G}^{\mu, \Phi}g, f) = \lim_{t\searrow 0} \frac{1}{t} \left\langle \tilde{G}^{\mu,\Phi} g-T_t^{\mu, \Phi} \tilde{G}^{\mu,\Phi} g, f \right\rangle_m =  \lim_{t\searrow 0} \frac{1}{t} \left\langle \int_0^t T_s^{\mu, \Phi}g\,ds, f \right\rangle_m = \langle g,f \rangle_m .
\label{eq:sub_subcri_proof1}
\end{equation}
By taking an approximating sequence, \((\ref{eq:sub_subcri_proof1})\) holds for any \(f\in \mathcal{D}_e(\mathcal{E}^{\mu,\Phi})\), so \(G^{\mu, \Phi}g = \tilde{G}^{\mu, \Phi}g\) and \(G^{\mu, \Phi}g \in \mathcal{H}_+^{\mu, \Phi}\). Since \((\mathcal{D}_e(\mathcal{E}^{\mu, \Phi}), \mathcal{E}^{\mu, \Phi})\)  is a Hilbert space, by Lemma \ref{Fe_lemma}, \((\mathcal{D}_e(\mathcal{E}^{\mu, \Phi, G^{\mu, \Phi}g}), \mathcal{E}^{\mu, \Phi, G^{\mu, \Phi}g})\) is also a Hilbert space, and \((2)\) holds.

We prove the equivalence of \((2)\) and \((3)\). Note that an irreducible Dirichlet form is either transient or recurrent. We assume that there exists \(h,g \in \mathcal{H}_+^{\mu, \Phi}\) such that \((\mathcal{E}^{\mu,\Phi,h}, \mathcal{D}(\mathcal{E}^{\mu,\Phi, h}))\) is transient but \((\mathcal{E}^{\mu,\Phi,g}, \mathcal{D}(\mathcal{E}^{\mu,\Phi, g}))\) is recurrent. Since \(1 \in \mathcal{D}_e(\mathcal{E}^{\mu, \Phi, g})\), by Lemma \ref{Fe_lemma}, we have \(g \in \mathcal{D}_e(\mathcal{E}^{\mu, \Phi})\) and so \(g/h \in \mathcal{D}_e(\mathcal{E}^{\mu, \Phi, h})\). It holds that
\[\mathcal{E}^{\mu, \Phi, h}\left(\frac{g}{h}, \frac{g}{h}\right)= \mathcal{E}^{\mu, \Phi}(g,g) = \mathcal{E}^{\mu, \Phi, g}(1,1)= 0\]
and \(g/h \not =0\), so this is a contradiction to the transience of  \((\mathcal{E}^{\mu,\Phi,h}, \mathcal{D}(\mathcal{E}^{\mu,\Phi, h}))\) and (2) and (3) are equivalent.

We assume \((2)\). For \(h\in \mathcal{H}_+^{\mu, \Phi}\), by \cite[Theorem 2.1.5]{CF12}, there exists a strictly positive bounded function \(v \in L^1(E;h^2m)\) such that \(\langle |u|, v\rangle_{h^2m} \leq \sqrt{\mathcal{E}^{\mu, \Phi, h}(u,u)}\) for any \(u \in \mathcal{D}_e(\mathcal{E}^{\mu, \Phi, h})\). For \(\eta \in L^1(E;m)\) satisfying \(0< \eta \leq 1\), the function \(g:= vh\eta \wedge \eta \in L^1(E;m)\) is strictly positive, bounded, and \((\ref{eq:sub_subcri})\) holds for any \(f\in \mathcal{D}_e(\mathcal{E}^{\mu, \Phi})\), so \((1)\) holds.
\end{proof}

\begin{proof}[Proof of Theorem \ref{subodinated_cri}]
Suppose that \((1)\) holds, that is, there exists a strictly positive function \(h \in \mathcal{D}_e(\mathcal{E}^{\mu, \Phi})\) such that \(\mathcal{E}^{\mu, \Phi}(h,h)=0\). We take an \(\mathcal{E}^{\mu, \Phi}\)-Cauchy  sequence \(\{h_n\}_n\subset \mathcal{
D}(\mathcal{E}^{\mu, \Phi})\) converging to \(h\) \(m\)-almost everywhere. By the spectral decomposition, for any \(t>0\), we have 
\[\frac{1}{t} \|h-T_t ^{\mu, \Phi} h\|_{L^2(E;m)}^2 \leq \varliminf_n \frac{1}{t}\|h_n-T_t^{\mu, \Phi} h_n\|_{L^2(E;m)}^2 \leq  \varliminf_n  \mathcal{E}^{\mu, \Phi}(h_n, h_n) = \mathcal{E}^{\mu, \Phi}(h,h) = 0,\]
Hence \(h \in \mathcal{H}_+^{\mu, \Phi}\) and \((2)\) holds.

The equivalence between \((2)\) and \((3)\) follows similarly to the proof of Theorem \ref{subodinated_subcri} by using the condition (IB).

Since, for a recurrent Dirichlet form, \(1\) belongs to the extended Dirichlet space and its value for the Dirichlet form is \(0\), \((3)\) implies \((1)\).
\end{proof}

\begin{proof}[Proof of Theorem \ref{subodinated_subcri}]
This follows from Theorem \ref{subodinated_cri} and \ref{subodinated_subcri}.
\end{proof}

Under the situation of Lemma \ref{irr}, \((\mathcal{E}^{\mu,\Phi, h}, \mathcal{D}(\mathcal{E}^{\mu, \Phi,h}))\) is not regular in general. The following is a sufficient condition for the regularity.
\begin{proposition}\label{regularDF}
Let \((\mathcal{E}, \mathcal{D}(\mathcal{E}))\) be an irreducible regular Dirichlet form  on \(L^2(E;m)\), \(\mu\) be a signed smooth Radon measure making \((\mathcal{E}^{\mu}, \mathcal{D}(\mathcal{E}) \cap C_c(E))\) non-negative definite closable, and \(\Phi\) be a Bernstein function satisfying the condition \rm{(IB)}. Suppose that \(\mu^- \in \mathcal{K}(\mathcal{E}^{\mu^+})\). Then an \(h\)-transform \((\mathcal{E}^{\mu,\Phi, h}, \mathcal{D}(\mathcal{E}^{\mu,\Phi,h}))\) is a regular Dirichlet form on \(L^2(E;h^2m)\).
\end{proposition}

\begin{proof}
We have already seen  \(\mathcal{D}(\mathcal{E}^{\mu}) = \mathcal{D}(\mathcal{E}^{\mu^+}) =  \mathcal{D}(\mathcal{E}) \cap L^2(E;\mu^+)\) for \(\mu^- \in \mathcal{K}(\mathcal{E}^{\mu^+})\) in Section \ref{sec:Takeda}. Since \(\mu^+\) is a Radon measure, it holds that \(\mathcal{D}(\mathcal{E}^{\mu}) \cap C_c = \mathcal{D}(\mathcal{E}) \cap L^2(E;\mu^+) \cap C_c = \mathcal{D}(\mathcal{E}) \cap C_c\) and, by \cite[Theorem 5.1.6]{CF12}, \(\mathcal{D}(\mathcal{E}) \cap C_c\) is dense in \(\mathcal{D}(\mathcal{E}^{\mu^+})=\mathcal{D}(\mathcal{E}^{\mu})\) with respect to \(\mathcal{E}_1^{\mu^+}\). By Theorem \ref{Okura}, \(\mathcal{D}(\mathcal{E}) \cap C_c\) is also dense in \(\mathcal{D}(\mathcal{E}^{\mu,\Phi})\) with respect to \(\mathcal{E}_1^{\mu^+, \Phi}.\)

We take \(h \in \mathcal{H}_+^{\mu, \Phi}\). Since \((T_s^{\mu^+} -T_s^\mu)u(x) = \mathbb{E}_x[e^{-A_s^{\mu^+}}(1-e^{A_s^{\mu^-}})u(X_s)] \leq 0\) for any \(u \geq 0\), where \(A^{\mu^+}\) (resp. \(A^{\mu^-}\)) is a PCAF corresponding to \(\mu^+\) (resp. \(\mu^-\)) with respect to a Hunt process \(X\) associated with a Dirichlet form \((\mathcal{E}, \mathcal{D}(\mathcal{E}))\). Then we have
\begin{eqnarray}
\nonumber \mathcal{E}^{\mu, \Phi, h}_1(f,f) &=& b\, \mathcal{E}^{\mu}(fh,fh) + \int_0^\infty \langle fh-T_s^\mu (fh), fh \rangle_m d\nu(s) + \int |fh|^2\, dm\\
\nonumber &=& \mathcal{E}^{\mu^+, \Phi}_1(fh,fh) - b \int |fh|^2\, d\mu^+ + \int_0^\infty \langle (T_s^{\mu^+} -T_s^\mu )(fh), fh \rangle_m d\nu(s)\\
&\leq & \mathcal{E}^{\mu^+, \Phi}_1(fh,fh) \label{eq:reg1}
\end{eqnarray}
for \(f \in \mathcal{D}(\mathcal{E}^{\mu, \Phi, h})\) with \(f\geq 0\).

We note that \((\mathcal{E}^{\mu^+}, \mathcal{D}(\mathcal{E}^{\mu^+}))\) is a regular Dirichlet form by \cite[Theorem 5.1.6]{CF12}, and so is \((\mathcal{E}^{\mu^+, \Phi}, \mathcal{D}(\mathcal{E}^{\mu^+, \Phi}))\) by \cite[Theorem 2.1]{O02}. Since \(T_t^{\mu^+, \Phi}h \leq T_t^{\mu, \Phi}h \leq h\), \((\mathcal{E}^{\mu^+,\Phi,h}, \mathcal{D}(\mathcal{E}^{\mu^+,\Phi,h}))\) is also a regular Dirichlet form. For any \(f\in \mathcal{D}(\mathcal{E}^{\mu, \Phi, h})=\mathcal{D}(\mathcal{E}^{{\mu^+}, \Phi, h})\), we take \(u_k \in \mathcal{D}(\mathcal{E})\cap C_c\) such that \(u_k\) converges to \(fh\) in \(\mathcal{E}_1^{\mu^+, \Phi}\). This is equivalent to the convergence of \(\frac{u_k}{h}\) to \(f\) in \(\mathcal{E}_1^{\mu^+, \Phi,h}\). By (\ref{eq:reg1}) and the Markov property of \((\mathcal{E}^{\mu^+, \Phi, h}, \mathcal{D}(\mathcal{E}^{\mu^+, \Phi, h}))\), we have

\begin{eqnarray*}
\mathcal{E}_1^{\mu, \Phi, h}(f-\frac{u_k}{h}, f-\frac{u_k}{h}) & \leq & 2\, \mathcal{E}_1^{\mu, \Phi, h}\left((f-\frac{u_k}{h})_+, (f-\frac{u_k}{h})_+\right) + 2\, \mathcal{E}_1^{\mu, \Phi, h}\left((f-\frac{u_k}{h})_-, (f-\frac{u_k}{h})_-\right)\\
& \leq & 2\, \mathcal{E}_1^{\mu^+, \Phi}\left((fh-u_k)_+, (fh-u_k)_+\right) + 2\, \mathcal{E}_1^{\mu^+, \Phi}\left((fh-u_k)_-, (fh-u_k)_-\right)\\
&=&  2\, \mathcal{E}_1^{\mu^+, \Phi, h}\left((f-\frac{u_k}{h})_+, (f-\frac{u_k}{h})_+\right) + 2\, \mathcal{E}_1^{\mu^+, \Phi, h}\left((f-\frac{u_k}{h})_-, (f-\frac{u_k}{h})_-\right)\\
&\leq & 4\, \mathcal{E}_1^{\mu^+, \Phi, h}\left(f-\frac{u_k}{h}, f-\frac{u_k}{h}\right).
\end{eqnarray*}
Combining this with the non-negativity of \(\mathcal{E}_1^{\mu, \Phi, h}\), \(\frac{u_k}{h}\) converges to \(f\) in \(\mathcal{E}_1^{\mu, \Phi, h}\).
In the same way as \cite[Lemma 2.4]{T14}, \(\frac{u_k}{h} \in \mathcal{D}(\mathcal{E})\cap C_c,\) so \(\mathcal{D}(\mathcal{E})\cap C_c\) is \(\mathcal{E}_1^{\mu,\Phi,h}\)-dense in \(\mathcal{D}(\mathcal{E}^{\mu, \Phi, h})\). Hence \((\mathcal{E}^{\mu,\Phi, h}, \mathcal{D}(\mathcal{E}^{\mu,\Phi,h}))\) is a regular Dirichlet form on \(L^2(E;h^2m)\).
\end{proof}

\begin{remark}
In general, neither \((\mathcal{E}^{\mu}, \mathcal{D}(\mathcal{E}^{\mu}))\) nor \((\mathcal{E}^{\mu,\Phi}, \mathcal{D}(\mathcal{E}^{\mu,\Phi}))\) is a Dirichlet form, so the corresponding stochastic processes for these closed forms do not exist. However, by Fukushima's theorem, there exists an \(h^2m\)-symmetric Hunt process on \(E\) associated with a regular Dirichlet form \((\mathcal{E}^{\mu,\Phi, h}, \mathcal{D}(\mathcal{E}^{\mu,\Phi,h}))\) for \(\mu^- \in \mathcal{K}(\mathcal{E}^{\mu^+})\). Therefore, we characterise the (sub)criticality for \(\mathcal{L}^{\mu, \Phi}\) through the stochastic processes associated with \((\mathcal{E}^{\mu}, \mathcal{D}(\mathcal{E}^{\mu}))\).
\end{remark}

Next we consider the subordinated regular Dirichlet form \((\mathcal{E}^{\mu,h, \Phi}, \mathcal{D}(\mathcal{E}^{\mu,h, \Phi}))\) of an \(h\)-transform of a Schr\"{o}dinger form \((\mathcal{E}^{\mu,h}, \mathcal{D}(\mathcal{E}^{\mu,h}))\). The following ensures that taking an \(h\)-transform and subordination commute for \(h \in \mathcal{H}_+^\mu\).

\begin{lemma}\label{corelemm}
Let \((\mathcal{E}, \mathcal{D}(\mathcal{E}))\) be an irreducible regular Dirichlet form  on \(L^2(E;m)\), \(\mu\) be a signed smooth Radon measure making \((\mathcal{E}^{\mu}, \mathcal{D}(\mathcal{E}) \cap C_c(E))\) non-negative definite closable, and \(\Phi\) be a Bernstein function. If \(\mathcal{H}_+^{\mu}\) is not empty, then \((\mathcal{E}^{\mu, h, \Phi}, \mathcal{D}(\mathcal{E}^{\mu, h, \Phi})) = (\mathcal{E}^{\mu, \Phi, h}, \mathcal{D}(\mathcal{E}^{\mu, \Phi, h}))\) for \(h\in \mathcal{H}_+^\mu\).
\end{lemma}

\begin{proof}
We take \(h\in \mathcal{H}_+^\mu\). By Theorem \ref{Okura}, \(\mathcal{D}(\mathcal{E}^{\mu, h})\) is \(\mathcal{E}_1^{\mu, h, \Phi}\)-dense in \(\mathcal{D}(\mathcal{E}^{\mu, h, \Phi})\) and \(\mathcal{D}(\mathcal{E}^{\mu })\) is \(\mathcal{E}_1^{\mu, \Phi}\)-dense in \(\mathcal{D}(\mathcal{E}^{\mu, \Phi})\). For any \(f\in \mathcal{D}(\mathcal{E}^{\mu, \Phi, h})\), \(fh \in \mathcal{D}(\mathcal{E}^{\mu, \Phi})\) and so there exists \(\{u_n\}_n \subset \mathcal{D}(\mathcal{E}^{\mu})\) such that \(\mathcal{E}_1^{\mu, \Phi}(u_n-fh, u_n-fh)\) converges to \(0\). Then we have \(\{\frac{u_n}{h}\}_n \subset \mathcal{D}(\mathcal{E}^{\mu,h})\) and  \(\mathcal{E}_1^{\mu, \Phi, h}(\frac{u_n}{h}-f, \frac{u_n}{h}-f)\) converges to \(0\), so \(\mathcal{D}(\mathcal{E}^{\mu, h})\) is also \(\mathcal{E}_1^{\mu, \Phi, h}\)-dense in \(\mathcal{D}(\mathcal{E}^{\mu, \Phi, h})\).

For \(f,g \in \mathcal{D}(\mathcal{E}^{\mu, h})\), by Theorem \ref{Okura}, we have
\begin{eqnarray*}
\mathcal{E}^{\mu, \Phi, h}(f, g) &=& \mathcal{E}^{\mu, \Phi}(fh, gh) \\
&=& b\, \mathcal{E}^{\mu}(fh,gh) + \int_0^\infty \langle fh-T_s^\mu (fh), gh \rangle_m \, d\nu(s)\\
&=& b\, \mathcal{E}^{\mu, h}(f,g) + \int_0^\infty \langle f-T_s^{\mu, h} f, g \rangle_{h^2m} \, d\nu(s)\\
&=& \mathcal{E}^{\mu, h, \Phi}(f, g).
\end{eqnarray*}
Since \(\mathcal{D}(\mathcal{E}^{\mu, h})\) is both an \(\mathcal{E}_1^{\mu, h, \Phi}\)-dense subset of \(\mathcal{D}(\mathcal{E}^{\mu, h, \Phi})\) and an \(\mathcal{E}_1^{\mu, \Phi, h}\)-dense subset of \(\mathcal{D}(\mathcal{E}^{\mu, \Phi, h})\), \(\mathcal{E}^{\mu, h, \Phi}\) coincides with \(\mathcal{E}^{\mu, \Phi, h}\) on \(\mathcal{D}(\mathcal{E}^{\mu, h, \Phi})=\mathcal{D}(\mathcal{E}^{\mu, \Phi, h})\).
\end{proof}

\begin{corollary}\label{subcrinomama}
Let \((\mathcal{E}, \mathcal{D}(\mathcal{E}))\) be an irreducible regular Dirichlet form  on \(L^2(E;m)\), \(\mu\) be a signed smooth Radon measure making \((\mathcal{E}^{\mu}, \mathcal{D}(\mathcal{E}) \cap C_c(E))\) non-negative definite closable, and \(\Phi\) be a Bernstein function satisfying the condition \rm{(IB)}. If \((\mathcal{E}^{\mu}, \mathcal{D}(\mathcal{E}^{\mu}))\) is subcritical, then so is \((\mathcal{E}^{\mu, \Phi}, \mathcal{D}(\mathcal{E}^{\mu, \Phi}))\).
\end{corollary}

\begin{proof}
Suppose that  \((\mathcal{E}^{\mu}, \mathcal{D}(\mathcal{E}^{\mu}))\) is subcritical. For \(h\in \mathcal{H}_+^\mu\), \((\mathcal{E}^{\mu,h}, \mathcal{D}(\mathcal{E}^{\mu,h}))\) is transient and, by \cite[Theorem 3.2]{O02}, the subordinated regular Dirichlet form \((\mathcal{E}^{\mu,h, \Phi}, \mathcal{D}(\mathcal{E}^{\mu,h, \Phi}))\) is also transient. Hence, by Lemma \ref{corelemm},  \((\mathcal{E}^{\mu, \Phi, h}, \mathcal{D}(\mathcal{E}^{\mu, \Phi, h}))\) is also transient.
\end{proof}

\begin{example}
A Bernstein function \(\Phi(\lambda)= \lambda^{\alpha/2}\) with \(0<\alpha <2\) satisfies the condition \rm{(IB)}. Hence, if \(\mathcal{L}\) is irreducible, then so is \(-(-\mathcal{L}^{\mu})^{\alpha/2}\). Moreover if \(\mathcal{L}^{\mu}\) is subcritical, then so is \(-(-\mathcal{L}^{\mu})^{\alpha/2}\).
\end{example}

In \cite[Lemma 2.5]{TU23}, for \(f \in \mathcal{D}_e(\mathcal{E}^{\mu})\) and its approximating sequence \(\{f_n\}_n \subset \mathcal{D}_e(\mathcal{E}^{\mu})\), \(f-T_t^\mu f \in \mathcal{D}(\mathcal{E}^{\mu})\) is realized as a limit of \(f_n-T_t^\mu f_n\) in \(L^2(E;m\)), and so \(T_t^\mu f := f-(f-T_t^\mu f) \in \mathcal{D}(\mathcal{E}^{\mu})\) is defined. Moreover, by using the spectral decomposition, we can see that
\begin{equation}
\frac{1}{t} \langle f-T_t^\mu f, f\rangle_m \leq \mathcal{E}^\mu(f,f)
\label{eq:spectral}
\end{equation}
for \(f \in \mathcal{D}(\mathcal{E}^\mu)\) and \(t>0\).

Unlike the case of subcriticality, it is not necessarily true that criticality is preserved by a subordination. The following is a sufficient condition for preserving a criticality.
\begin{proposition}\label{cri_inv}
Let \((\mathcal{E}, \mathcal{D}(\mathcal{E}))\) be an irreducible regular Dirichlet form  on \(L^2(E;m)\), \(\mu\) be a signed smooth Radon measure making \((\mathcal{E}^{\mu}, \mathcal{D}(\mathcal{E}) \cap C_c(E))\) non-negative definite closable, and \(\Phi\) be a Bernstein function satisfying the condition \rm{(IB)}. If \((\mathcal{E}^{\mu}, \mathcal{D}(\mathcal{E}^{\mu}))\) is critical and suppose that
\[\int_0^\infty s\, d\nu(s) < \infty\]
then \((\mathcal{E}^{\mu, \Phi}, \mathcal{D}(\mathcal{E}^{\mu, \Phi}))\) is also critical.
\end{proposition}

\begin{proof}
By \cite[Corollary 2.7]{TU23}, we can take \(h \in \mathcal{D}_e(\mathcal{E}^{\mu})\cap \mathcal{H}_+^\mu\) satisfying \(\mathcal{E}^{\mu}(h,h)=0\) and \(T_t^\mu h = h\) \(m\)-almost everywhere. We assume that the Bernstein function \(\Phi\) satisfies \(\int_0^\infty s\, d\nu(s) < \infty. \) For an \(\mathcal{E}^\mu\)-Cauchy sequence \(\{h_n\}_n \subset \mathcal{D}(\mathcal{E}^{\mu})\) converging to \(h\) \(m\)-almost everywhere, by Theorem \ref{Okura}, we have
\begin{eqnarray*}
\lefteqn{\varlimsup_{n, m\to \infty} \mathcal{E}^{\mu, \Phi}(h_n-h_m,h_n-h_m)}\\
 &\leq  &\varlimsup_{n,m \to \infty} b\, \mathcal{E}^{\mu}(h_n-h_m ,h_n-h_m) + \varlimsup_{n,m\to \infty} \left| \int_0^\infty \langle h_n-h_m-T_s^\mu (h_n-h_m), h_n-h_m \rangle_m d\nu(s) \right|\\
& \leq & 0 + \varlimsup_{n,m\to \infty} \mathcal{E}^\mu(h_n-h_m, h_n-h_m) \, \int_0^\infty s \, d\nu(s)\\
&=& 0,
\end{eqnarray*}
where we used \((\ref{eq:spectral})\) in the second equality. Hence \(\{h_n\}_n \subset \mathcal{D}(\mathcal{E}^\mu) \subset \mathcal{D}(\mathcal{E}^{\mu, \Phi})\) is also an \(\mathcal{E}^{\mu, \Phi}\)-Cauchy sequence and so \(h\in \mathcal{D}_e(\mathcal{E}^{\mu, \Phi})\). Moreover, similarly to the above, we have
\begin{eqnarray*}
\lefteqn{\mathcal{E}^{\mu, \Phi}(h,h) =\lim_{n \to \infty} \mathcal{E}^{\mu, \Phi}(h_n,h_n) }\\
&\leq &\varlimsup_{n \to \infty} b \mathcal{E}^\mu(h_n,h_n) + \varlimsup_{n \to \infty} \mathcal{E}^\mu (h_n, h_n) \, \int_0^\infty s \, d\nu(s) =\left( b+\int_0^\infty s \, d\nu(s) \right) \mathcal{E}^\mu(h,h) = 0,
\end{eqnarray*}
so \((\mathcal{E}^{\mu, \Phi}, \mathcal{D}(\mathcal{E}^{\mu, \Phi}))\) is critical.
\end{proof}

There are examples in which \(\mathcal{L}^\mu\) is critical, but its subordination \(\Phi(\mathcal{L}^\mu)\) becomes subcritical. See Section \ref{sec:Example} for details.

We set 
\[\mathcal{D}^*(\sqrt{-\mathcal{L}^{\mu, \Phi}}):=\{f : |\langle f,u \rangle_m | \leq C \sqrt{\mathcal{E}^{\mu, \Phi}(u,u)} \text{\ for\ }u \in \mathcal{D}(\sqrt{-\mathcal{L}^{\mu, \Phi}})\}\]
and
\[\mathcal{D}_e^*(\sqrt{-\mathcal{L}^{\mu, \Phi}}):=\{f : |\langle f,u \rangle_m | \leq C \sqrt{\mathcal{E}^{\mu, \Phi}(u,u)} \text{\ for\ }u \in \mathcal{D}_e(\sqrt{-\mathcal{L}^{\mu, \Phi}})\}.\]
Here we defined \(\langle f,u \rangle_m\) for \(u\in \mathcal{D}_e(\sqrt{-\mathcal{L}^{\mu, \Phi}})\}\) by the limit of \(\langle f,u_n \rangle_m\) for an approximate sequence \(\{u_n\}_n \subset \mathcal{D}(\sqrt{-\mathcal{L}^{\mu, \Phi}})\). The space \(\mathcal{D}^*_e(\sqrt{-\mathcal{L}^{\mu, \Phi}})\) is well-defined in a similar way to the proof of \cite[Proposition 2.1.5]{CF12}.

For \(v\in L^2(E;m)\), we put
\[G^{\mu, \Phi}v := \int_0^\infty T_t^{\mu, \Phi}v\,dt \leq \infty.\]
Note that for any \(\alpha>0\) and \(v \in L^2(E;m)\), by the Riesz representation theorem, there exists \(G_\alpha^{\mu, \Phi} v \in \mathcal{D}(\mathcal{E}^{\mu, \Phi})\) such that \(\mathcal{E}_\alpha^{\mu, \Phi}(G_\alpha^{\mu, \Phi} v, f)=\langle v, f \rangle_m \) for any \(f\in \mathcal{D}(\mathcal{E}^{\mu, \Phi})\), and \(G_\alpha^{\mu, \Phi} v=(\alpha-\mathcal{L}^{\mu, \Phi})^{-1}v = \int_0 ^\infty e^{-\alpha t} T_t^{\mu, \Phi} v \,dt\).

We consider equivalent conditions for the subcriticality from a perspective of an analysis of operators.
For an operator \(A\), we define the range of \(A\) by \(R(A):=\{Af : f \in \mathcal{D}(A)\}\). 

\begin{theorem}\label{thm_range}
Let \((\mathcal{E}, \mathcal{D}(\mathcal{E}))\) be a regular Dirichlet form on \(L^2(E;m)\), \(\mu\) be a signed smooth Radon measure making \((\mathcal{E}^{\mu}, \mathcal{D}(\mathcal{E}) \cap C_c(E))\) non-negative definite closable, and \(\Phi\) be a Bernstein function. Then it holds that
\begin{eqnarray*}
R(\sqrt{-\mathcal{L}^{\mu, \Phi}}) &=& L^2(E;m) \cap \mathcal{D}_e^*(\sqrt{-\mathcal{L}^{\mu, \Phi}})\\
&=& L^2(E;m) \cap \mathcal{D}^*(\sqrt{-\mathcal{L}^{\mu, \Phi}})\\
&=& \{f \in L^2(E;m) : \langle f, G^{\mu,\Phi}f \rangle_m  < \infty\}.
\end{eqnarray*}
\end{theorem}
\begin{proof}
It is clear that \(L^2(E;m) \cap \mathcal{D}_e^*(\sqrt{-\mathcal{L}^{\mu, \Phi}})\subset L^2(E;m) \cap \mathcal{D}^*(\sqrt{-\mathcal{L}^{\mu, \Phi}}).\) For any \(f \in L^2(E;m) \cap \mathcal{D}^*(\sqrt{-\mathcal{L}^{\mu, \Phi}})\) and any \(u \in \mathcal{D}_e(\sqrt{-\mathcal{L}^{\mu, \Phi}})\), we take an approximate sequence \(\{u_n\}_n \subset \mathcal{D}(\sqrt{-\mathcal{L}^{\mu, \Phi}})\) of \(u\). Then \(\{\langle f, u_n \rangle_m \}_n\) is a Cauchy sequence since
\begin{eqnarray*}
\varlimsup_{n,m\to \infty}|\langle f, u_n-u_m \rangle_m| &=& \varlimsup_{n,m\to \infty}|\langle f_*, \sqrt{-\mathcal{L}^{\mu, \Phi}}(u_n-u_m) \rangle_m|\\
&\leq &\varlimsup_{n,m\to \infty} \|f_*\| \sqrt{\mathcal{E}^{\mu, \Phi}(u_n-u_m,u_n-u_m)}\\
& =& 0.
\end{eqnarray*}
Similarly, \(\langle f, u \rangle_m\) is independent of the choice of an approximating sequence of \(u\). We have
\begin{eqnarray*}
|\langle f, u \rangle_m | := \lim_{n\to \infty}|\langle f, u_n \rangle_m | \leq  \lim_{n\to \infty}C \, \sqrt{\mathcal{E}^{\mu, \Phi}(u_n,u_n)} = C\, \sqrt{\mathcal{E}^{\mu, \Phi}(u,u)}.
\end{eqnarray*}
Hence \(f \in L^2(E;m) \cap \mathcal{D}_e^*(\sqrt{-\mathcal{L}^{\mu, \Phi}})\), and so  \(L^2(E;m) \cap \mathcal{D}_e^*(\sqrt{-\mathcal{L}^{\mu, \Phi}})= L^2(E;m) \cap \mathcal{D}^*(\sqrt{-\mathcal{L}^{\mu, \Phi}}).\)

We take \(f \in R(\sqrt{-\mathcal{L}^{\mu, \Phi}})\) and \(f_* \in \mathcal{D}(\sqrt{-\mathcal{L}^{\mu, \Phi}})\) with \(f= \sqrt{-\mathcal{L}^{\mu, \Phi}}f_*.\) For any \(u \in \mathcal{D}(\sqrt{-\mathcal{L}^{\mu, \Phi}})\), we have
\begin{eqnarray*}
|\langle f, u \rangle_m |&= & |\langle f_*, \sqrt{-\mathcal{L}^{\mu, \Phi}}u \rangle_m| \leq \| f_*\|_{L^2(E;m)} \, \| \sqrt{-\mathcal{L}^{\mu, \Phi}}u \|_{L^2(E;m)} = \| f_*\|_{L^2(E;m)} \, \sqrt{\mathcal{E}^{\mu, \Phi}(u,u)}
\end{eqnarray*}
and so \(f \in L^2(E;m) \cap \mathcal{D}^*(\sqrt{-\mathcal{L}^{\mu, \Phi}}).\) On the other hand, we take \(f \in L^2(E;m) \cap \mathcal{D}_e^*(\sqrt{-\mathcal{L}^{\mu, \Phi}})\). Then there exists \(C\) such that, for  \(G_\alpha^{\mu, \Phi}f = (\alpha-\mathcal{L}^{\mu, \Phi})^{-1}f\),
\begin{eqnarray*}
\langle f, G_\alpha^{\mu, \Phi}f \rangle_m \leq  C \sqrt{\mathcal{E}^{\mu, \Phi}( G_\alpha^{\mu, \Phi}f,  G_\alpha^{\mu, \Phi}f)} \leq C \sqrt{\mathcal{E}^{\mu, \Phi}_\alpha( G_\alpha^{\mu, \Phi}f,  G_\alpha^{\mu, \Phi}f)} =  C \sqrt{\langle f, G_\alpha^{\mu, \Phi}f \rangle_m}
\end{eqnarray*}
and so we have  \(\sup_{\alpha}\langle f, G_\alpha^{\mu, \Phi}f \rangle_m \leq C^2.\) By the spectral decomposition and letting \(\alpha\) tend to \(0\), we have \(f \in \mathcal{D}((-\mathcal{L}^{\mu, \Phi})^{-1/2})\). Since \(f\in L^2(E;m)\), it holds that \((-\mathcal{L}^{\mu, \Phi})^{-1/2}f \in \mathcal{D}(\sqrt{-\mathcal{L}^{\mu, \Phi}})\) and so \(f=(-\mathcal{L}^{\mu, \Phi})^{1/2} (-\mathcal{L}^{\mu, \Phi})^{-1/2}f \in R(\sqrt{-\mathcal{L}^{\mu, \Phi}})\).

As in \cite[Lemma 2.1.4 (ii)]{CF12}, it holds that
\begin{equation}
\sup_{u \in  \mathcal{D}(\sqrt{-\mathcal{L}^{\mu, \Phi}})} \frac{\langle f,|u|\rangle_m}{\sqrt{\mathcal{E}^{\mu, \Phi}(u,u)}} = \langle f, G^{\mu,\Phi}f \rangle_m. \label{eq:gGg}
\end{equation}
Hence we have \(L^2(E;m) \cap \mathcal{D}^*(\sqrt{-\mathcal{L}^{\mu, \Phi}})= \{f \in L^2(E;m) : \langle f, G^{\mu,\Phi}f \rangle_m  < \infty\}\).
\end{proof}

\begin{theorem}\label{extSoba}
Let \((\mathcal{E}, \mathcal{D}(\mathcal{E}))\) be an irreducible regular Dirichlet form  on \(L^2(E;m)\), \(\mu\) be a signed smooth Radon measure making \((\mathcal{E}^{\mu}, \mathcal{D}(\mathcal{E}) \cap C_c(E))\) non-negative definite closable, and \(\Phi\) be a Bernstein function satisfying the condition \rm{(IB)}. Then the following are equivalent.
\begin{enumerate}
\item There exists a strictly positive function \(g \in R(\sqrt{-\mathcal{L}^{\mu, \Phi}})\),
\item  There exists a non-negative function \(g \in R(\sqrt{-\mathcal{L}^{\mu, \Phi}}) \) such that \(g\) is not \(0\),
\item There exists a non-negative function \(g \in  L^2(E;m)\) such that \(  \langle g, G^{\mu,\Phi}g \rangle_m  < \infty\) and \(g\) is not \(0\),
\item \(\mathcal{L}^{\mu, \Phi} \) is subcritical.
\end{enumerate}
\end{theorem}

\begin{proof}
Recall that \(\mathcal{D}(\mathcal{E}^{\mu, \Phi})= \mathcal{D}(\sqrt{-\mathcal{L}^{\mu, \Phi}})\) and \(\sqrt{-\mathcal{L}^{\mu, \Phi}}\, \mathcal{D}(\mathcal{E}^{\mu, \Phi})= R(\sqrt{-\mathcal{L}^{\mu, \Phi}})\).

Suppose (4). We take a strictly positive bounded function \(g \in L^1(E;m)\) satisfying (\ref{eq:sub_subcri}). Then (3) follows from (\ref{eq:gGg}).

The equivalence between (2) and (3) follows from Theorem \ref{thm_range}.

Suppose (2).  We take \(g_0 \geq 0\) with \(g_0 \not \equiv 0\) and \(g_0 \in L^2(E;m)\), and \(u_0 \in \mathcal{D}(\mathcal{E}^{\mu, \Phi})\) such that \(g_0 = \sqrt{-\mathcal{L}^{\mu, \Phi}} u_0\). We set \(g:=G_1^{\mu, \Phi} g_0\), then \(g = \sqrt{-\mathcal{L}^{\mu, \Phi}} (1-\mathcal{L}^{\mu, \Phi})^{-1} u_0 = \sqrt{-\mathcal{L}^{\mu, \Phi}} G_1^{\mu, \Phi}u_0 \in R(\sqrt{-\mathcal{L}^{\mu, \Phi}})\) and \(g\) is strictly positive. Indeed, we set \(B:=\{x \in E : G_1^{\mu, \Phi}g_0=0 \}\) and, for any \(s>0\), we have
\[G_1^{\mu, \Phi}g_0= \int_0 ^s e^{-t} T_t^{\mu, \Phi} g_0 \,dt + \int_s ^\infty e^{-t} T_t^{\mu, \Phi} g_0 \,dt = \int_0 ^s e^{-t} T_t^{\mu, \Phi} g_0 \,dt + e^{-s}T_s^{\mu, \Phi} G_1^{\mu, \Phi} g_0.\]
For \(m\)-almost every \(x\in B\), we have \(0\leq  e^{-s}T_s^{\mu, \Phi} G_1^{\mu, \Phi} g_0(x) \,dt  \leq G_1^{\mu, \Phi}g_0(x)=0 \) and so, 
\[G_1^{\mu, \Phi} g(x) = G_1^{\mu, \Phi} G_1^{\mu, \Phi} g_0 (x)=\int_0^\infty e^{-s}T_s^{\mu, \Phi}  G_1^{\mu, \Phi} g_0(x) \,ds =0 \]
for \(m\)-almost every \(x\in B\). For any non-negative function \(f\in L^2(E;m)\), we have
\[\langle \, G_1^{\mu, \Phi} ( 1_{B} f), g \rangle_m = \langle  \,f, 1_B G_1^{\mu, \Phi} g \rangle_m  =0\]
and so \( 1_{B^c}G_1^{\mu, \Phi} ( 1_{B} f) =  0.\) For \(f \in L^2(E;m)\), we take non-negative functions \(f_+, f_-\) with \(f=f+-f_-\), we have  \( 1_{B^c}G_1^{\mu, \Phi} ( 1_{B} f) =  0.\) By the same argument as \cite[Proposition 2.1.6]{CF12}, \(B\) is \(\{T_t^{\mu, \Phi} \}_t\)-invariant set.  By the irreducibility (Proposition \ref{irr}), \(m(B)=0\) or \(m(B^c)=0\). Since \(g_0 \not \equiv 0\) and \(G_1^{\mu, \Phi}\) is injective, we have \(g\geq 0\) and so \(m(B)=0\). Hence we get (1).

Suppose (1). We take a strictly positive function \(g\in L^2(E;m)\) and \(u_0 \in \mathcal{D}(\mathcal{E}^{\mu, \Phi})\) such that \(g = \sqrt{-\mathcal{L}^{\mu, \Phi}} u_0\). We also take a strictly positive bounded function \(\varphi \in L^1(E;m)\). For any \(f\in \mathcal{D}(\mathcal{E}^{\mu, \Phi})\), we have
\begin{eqnarray*}
\langle g \wedge \varphi,|f| \rangle_m & \leq &\langle g,|f| \rangle_m \ = \ \langle u_0, \sqrt{-\mathcal{L}^{\mu, \Phi}} |f| \rangle_m \\
&\leq & \| u_0 \|_{L^2(E;m)}\, \| \sqrt{-\mathcal{L}^{\mu, \Phi}} |f|  \|_{L^2(E;m)}\\ 
&=& \| u_0 \|_{L^2(E;m)}\, \sqrt{\mathcal{E}^{\mu, \Phi}(|f|,|f|)}\\ 
&\leq &  \| u_0 \|_{L^2(E;m)}\, \sqrt{\mathcal{E}^{\mu, \Phi}(f,f)}\  <\  \infty.
\end{eqnarray*}
By using an approximating sequence and the Fatou lemma, the above inequality also holds for \(f \in \mathcal{D}_e(\mathcal{E}^{\mu, \Phi})\), so we get (4).
\end{proof}

The following is an extension of \cite[Theorem 1.1]{S26}.
\begin{theorem}\label{extSoba2}
Let \((\mathcal{E}, \mathcal{D}(\mathcal{E}))\) be an irreducible regular Dirichlet form  on \(L^2(E;m)\), \(\mu\) be a signed smooth Radon measure making \((\mathcal{E}^{\mu}, \mathcal{D}(\mathcal{E}) \cap C_c(E))\) non-negative definite closable, and \(\Phi\) be a Bernstein function satisfying the condition \rm{(IB)}. Assume that \(\inf_{x\in K} G_1^{\mu, \Phi} g(x)>0\) for any compact set \(K\) and a strictly positive bounded function \(g \in L^1(E;m)\), then the following are equivalent.
\begin{enumerate}
\item There exists a strictly positive function \(g \in R(\sqrt{-\mathcal{L}^{\mu, \Phi}})\),
\item  There exists a non-negative function \(g \in R(\sqrt{-\mathcal{L}^{\mu, \Phi}}) \) such that \(g\) is not \(0\),
\item There exists a non-negative function \(g \in  L^2(E;m)\) such that \(  \langle g, G^{\mu,\Phi}g \rangle_m  < \infty\) and \(g\) is not \(0\),
\item \(\mathcal{L}^{\mu, \Phi} \) is subcritical.
\item Any bounded function with compact support belongs to \(R(\sqrt{-\mathcal{L}^{\mu, \Phi}})\),
\item \(C_c(E) \subset R(\sqrt{-\mathcal{L}^{\mu, \Phi}})\).
\end{enumerate}
\end{theorem}
\begin{proof}
The equivalence of (1) to (4) are shown in Theorem \ref{extSoba}.

Assume that (1) to (4). Then we have a strictly positive bounded function \(g\in L^1(E;m)\) satisfying (\ref{eq:sub_subcri}). In the proof of Theorem \ref{extSoba} that (2) implies (1), we have seen \(G_1^{\mu, \Phi} g \in R(\sqrt{-\mathcal{L}^{\mu, \Phi}})\).
For any bounded function \(f\) with compact support \(K\), it holds that \(|f| \leq C_f G_1^{\mu,\Phi} g\) where \(C_f \leq \|f\|_\infty(\inf_{x\in K} G_1^{\mu, \Phi} g )^{-1}<\infty.\) By Theorem \ref{thm_range}, we have
\[|\langle f, u\rangle_m| \leq \langle |f|, |u|\rangle_m \leq C_f \langle  G_1^{\mu,\Phi} g, |u|\rangle_m \leq C_f C \sqrt{\mathcal{E}^{\mu,\Phi}(u,u)}\]
and so \(f \in R(\sqrt{-\mathcal{L}^{\mu, \Phi}}) = L^2(E;m) \cap \mathcal{D}_e^*(\sqrt{-\mathcal{L}^{\mu, \Phi}}).\) Hence (7) holds.

Clearly (5) implies (6), and (6) implies (1).
\end{proof}

We provide two criteria to ensure \(\inf_{x\in K} G_1^{\mu, \Phi} g(x)\). The first one is the following condition on the lower boundedness of a semigroup.
\begin{enumerate}
\item[(LB) :]Suppose that for any compact set \(K\) and a strictly positive bounded function \(g \in L^1(E;m)\), there exists an open interval \(I\subset [0,\infty)\) such that \(\inf_{s\in I} \inf_{x\in K} T_s^{\mu^+}g(x) >0\).
\end{enumerate}

To prove that (LB) is a sufficient criterion, we describe a property of subordinations. Let \(U_\alpha(ds):=\int_0^\infty e^{-\alpha t}\,d\eta_t(s)\,dt\) for \(\alpha \geq 0\). This is called a \(\alpha\)-potential measure (\cite[Definition 30.9]{S99}).
\begin{lemma}\label{lem_IBGreen}
Under the condition {\rm (IB)}, it holds that \({\rm supp}(U_1)=[0,\infty).\)
\end{lemma}
\begin{proof}
Let \(S\) be a subordinator with \(\Phi\), then it holds that \({\rm supp}(\eta_t) = {\rm supp}(\mathbb{P}(S_t \in ds))\). By \cite[Proposition 1.3]{B99}, we have \(S_t = bt + \sum_{0\leq s\leq t}N_s\) where \(N\) is a Poisson point process with values in \([0,\infty)\) and characteristic measure \(\nu\). Hence, by \cite[Lemma 24.1, Theorem 24.5]{S99}, we have \({\rm supp}(\eta_t) = bt + \overline{\{\sum_{j=1}^n x_j : x_j \in {\rm supp}(\nu), n\in \mathbb{N}\}}\) and so \({\rm supp}(U_1)= \overline{\bigcup_{t>0} {\rm supp}(\eta_t)}=[0,\infty)\).
\end{proof}

\begin{corollary}\label{cor_extSoba1}
Under the condition {\rm (LB)}, it holds that \(\inf_{x\in K} G_1^{\mu, \Phi} g(x)\) for any compact set \(K\) and a strictly positive bounded function \(g \in L^1(E;m)\). Hence the equivalence between (1) to (6) in Theorem \ref{extSoba2} holds.
\end{corollary}
\begin{proof}
For any compact set \(K\) a strictly positive bounded function \(g \in L^1(E;m)\), b Lemma \ref{lem_IBGreen}, we have
\begin{eqnarray*}
\inf_{x\in K} G_1^{\mu, \Phi} g(x) &=& \inf_{x\in K} \int_0^\infty\int_0^\infty e^{-t}T_s^{\mu}g(x)\,d\eta_t(s)\,dt\\
& =& \inf_{x\in K} \int_{[0,\infty)} T_s^{\mu}g(x)\,dU_1(s)\\
&\geq &  \int_{I} \inf_{x\in K}\inf_{s\in I} T_s^{\mu} g(x)\,dU_1(s)\\
&\geq & C \inf_{x\in K} \inf_{s\in I} T_s^{\mu^+} g(x)\\
&>&0.
\end{eqnarray*}
\end{proof}

The second criterion below is a condition for deriving \(\inf_{x\in K} G_1^{\mu, \Phi} g(x)\) that depends only on the original operator \(\mathcal{L}\).
\begin{corollary}\label{cor_extSoba2}
Under the condition {\rm (SF)} for \((\mathcal{E}, \mathcal{D}(\mathcal{E}))\), it holds that \(\inf_{x\in K} G_1^{\mu, \Phi} g(x)\) for any compact set \(K\), a strictly positive bounded function \(g \in L^1(E;m)\) and \(\mu^+ \in \mathcal{K}\). Hence the equivalence between (1) to (6) in Theorem \ref{extSoba2} holds.
\end{corollary}
\begin{proof}
Let \(A^{\mu^+}\) be a PCAF associated with \(\mu^+\). Since \(\mu^+ \) is a Kato class measure, it holds that  \(\lim_{t\to 0} \sup_{x\in K} \mathbb{E}_x[1-e^{-A_t^{\mu^+}}]=0\) for any compact set \(K\), and so \(\{T_t^{\mu^+}\}_t\) also enjoys (SF) by \cite[Theorem 1.1]{CK09} (See also \cite{C85}). Here \(T_t^{\mu^+}f(x):=\mathbb{E}_x[e^{-A_t^{\mu^+}}f(X_t)]\) is the semigroup of the process of \(X\) killed by \(\mu^+\).

For any compact set \(K\), a strictly positive bounded function \(g \in L^1(E;m)\), we have
\[G_1^{\mu, \Phi} g(x) = \int_{[0, \infty)} T_s^{\mu}g(x)\,dU_1(s) \geq \int_{(0, \infty)} T_s^{\mu}g(x)\,dU_1(s) \geq \int_{(0, \infty)} T_s^{\mu^+}g(x)\,dU_1(s). \]
By the dominated convergence theorem and (SF) for \(T_s^{\mu^+}\), a function \(\int_{(0, \infty)} T_s^{\mu^+}g(x)\,dU_1(s)\) is continuous. Since \(g\) is strictly positive and \(\{T_s^{\mu^+}\}_s\) corresponds to a process, \(\int_{(0, \infty)} T_s^{\mu^+}g(x)\,dU_1(s)\) is also strictly positive. Hence we have
\[\inf_{x\in K} G_1^{\mu, \Phi} g(x) \geq \inf_{x\in K} \int_{(0, \infty)} T_s^{\mu^+}g(x)\,dU_1(s) >0.\]
\end{proof}

\begin{example}
Let \(X\) be a Brownian motion on \(\mathbb{R}^d\) and \(\mu \in \mathcal{K}-\mathcal{S}_R\). Brownian motion corresponds to the Laplace operator \(\Delta\) and, since a Dirichlet form associated with Brownian motion is irreducible regular, and Brownian motion satisfies (SF), the equivalence between (1) to (6) in Theorem \ref{extSoba2} holds for a Bernstein function with (IB).

Similarly, the equivalence between (1) to (6) in Theorem \ref{extSoba2} holds for a L\'{e}vy process \(X\). For example, this equivalence holds for the case that the original operator \(-\mathcal{L}\) is \((-\Delta)^{\alpha/2}\) with \(0<\alpha \leq 2\). 
\end{example}

\section{Application to wave equations}\label{sec:wave}
In this section, we discuss the application of subcriticality to wave equations, which was examined in \cite{S26} for Laplace operators on subsets of \(\mathbb{R}^d\), and we extend it to a broader class of subordinated operators.

As in the previous section, let \((\mathcal{E}, \mathcal{D}(\mathcal{E}))\) be an irreducible regular Dirichlet form  on \(L^2(E;m)\), \(\mu\) be a signed smooth Radon measure making \((\mathcal{E}^{\mu}, \mathcal{D}(\mathcal{E}) \cap C_c(E))\) non-negative definite closable, and \(\Phi\) be a Bernstein function satisfying the condition \rm{(IB)}. Recall that \(\mathcal{L}\) is a self-adjoint operator associated with \((\mathcal{E}, \mathcal{D}(\mathcal{E}))\), \(\mathcal{L}^\mu:=\mathcal{L}-\mu\), and \(\mathcal{L}^{\mu,\Phi}:=-\Phi(-\mathcal{L}^\mu)\).

\begin{definition}
We define the norm on \(R(\sqrt{-\mathcal{L}^{\mu, \Phi}})\) by
\[\|f\|_{R(\sqrt{-\mathcal{L}^{\mu, \Phi}})} := \inf\left\{\sqrt{\mathcal{E}_1^{\mu, \Phi}(u,u)} : u\in \mathcal{D}(\sqrt{-\mathcal{L}^{\mu, \Phi}}), \sqrt{-\mathcal{L}^{\mu, \Phi}}u=f\right\}.\]
\end{definition}

\begin{lemma}\label{quolem1}
The norm \(\|f\|_{R(\sqrt{-\mathcal{L}^{\mu, \Phi}})}\) is isometric to the quotient  norm on \(\mathcal{D}(\mathcal{E}^{\mu, \Phi})/\rm{Ker}(\sqrt{-\mathcal{L}^{\mu, \Phi}})\) induced by the \(\mathcal{E}_1^{\mu, \Phi}\)-norm. In particular, \(R(\sqrt{-\mathcal{L}^{\mu, \Phi}})\) is a Hilbert space.
\end{lemma}
\begin{proof}
Note that \(\mathcal{D}(\mathcal{E}^{\mu, \Phi}) = \mathcal{D}(\sqrt{-\mathcal{L}^{\mu, \Phi}})\) and this is a Hilbert space equipped with the \(\mathcal{E}_1^{\mu, \Phi}\)-norm. We consider the quotient norm \(\|\cdot\|_Q\) on \(\mathcal{D}(\mathcal{E}^{\mu, \Phi})/\text{Ker}(\sqrt{-\mathcal{L}^{\mu, \Phi}})\) by
\[\|u\|_Q := \inf\left\{ \sqrt{\mathcal{E}^{\mu, \Phi}_1(u+v,u+v)} : \sqrt{-\mathcal{L}^{\mu, \Phi}}v=0 \right\}.\]
By the fundamental theorem on homomorphisms, the proof is completed.
\end{proof}

If \(\sqrt{-\mathcal{L}^{\mu, \Phi}}\) is injective, then the norm \(\|\cdot \|_{R(\sqrt{-\mathcal{L}^{\mu, \Phi}})}\) on \(R(\sqrt{-\mathcal{L}^{\mu, \Phi}})\) coincides with the norm introduced in \cite{S26}. In particular, if \(\mathcal{L}^{\mu, \Phi}\) is subcritical, then \(\sqrt{-\mathcal{L}^{\mu, \Phi}}\) is injective.

\begin{lemma}
Suppose that \(\mathcal{L}^{\mu, \Phi}\) is subcritical. Then an \(\mathcal{E}^{\mu, \Phi}_1\)-dense subspace \(\mathcal{C}\) of \(\mathcal{D}(\mathcal{E}^{\mu, \Phi})\) is also \(\|\cdot\|_{R(\sqrt{-\mathcal{L}^{\mu, \Phi}})}\)-dense in \(R(\sqrt{-\mathcal{L}^{\mu, \Phi}})\). In particular, for \(\mu^- \in \mathcal{K}_\infty(G^{\mu^+})\), \(\mathcal{D}(\mathcal{E}) \cap C_c\) is \(\|\cdot\|_{R(\sqrt{-\mathcal{L}^{\mu, \Phi}})}\)-dense in \(R(\sqrt{-\mathcal{L}^{\mu, \Phi}})\).
\end{lemma}
\begin{proof}
Suppose that \(\mathcal{L}^{\mu, \Phi}\) is subcritical. Then, by Theorem \ref{extSoba}, \(\mathcal{C} \subset R(\sqrt{-\mathcal{L}^{\mu, \Phi}})\) holds. Suppose that \(g\in R(\sqrt{-\mathcal{L}^{\mu, \Phi}})\) satisfies \(\langle f, g\rangle_{R(\sqrt{-\mathcal{L}^{\mu, \Phi}})} = 0\) for any \(f\in \mathcal{C}\). We take \(f_*, g_* \in \mathcal{D}(\mathcal{E}^{\mu, \Phi})\) satisfying \(f= \sqrt{-\mathcal{L}^{\mu, \Phi}} f_*\) and \(g=\sqrt{-\mathcal{L}^{\mu, \Phi}}g_*\), and we have
\[
0=\langle f, g\rangle_{R(\sqrt{-\mathcal{L}^{\mu, \Phi}})}=\mathcal{E}_1^{\mu, \Phi}(f_*, g_*)=\mathcal{E}_1^{\mu, \Phi}(f, (-\mathcal{L}^{\mu, \Phi})^{-1/2}g_*)= \mathcal{E}_1^{\mu, \Phi}(f, G^{\mu, \Phi} g),
\]
where \(G^{\mu, \Phi} g := (-\mathcal{L}^{\mu, \Phi})^{-1}g.\) We note that, for any \(h\in \mathcal{H}_+^{\mu, \Phi}\), \((\mathcal{E}^{\mu, \Phi, h}, \mathcal{D}(\mathcal{E}^{\mu, \Phi, h}))\) is transient and so there exists a \(0\)-order resolvent \(G^{\mu, \Phi, h} (\frac{g}{h}) \in \mathcal{D}(\mathcal{E}^{\mu, \Phi, h})\) for \(\frac{g}{h} \in L^2(E;h^2dm)\), and it holds that \(G^{\mu, \Phi} g=  hG^{\mu, \Phi, h} (\frac{g}{h}) \in \mathcal{D}(\mathcal{E}^{\mu, \Phi}).\) Since \(\mathcal{C}\) is \(\mathcal{E}_1^{\mu, \Phi}\)-dense in \(\mathcal{D}(\mathcal{E}^{\mu, \Phi})\), we obtain \(G^{\mu, \Phi}g=0\), and so \(g=0\). Thus \(\mathcal{C}\) is \(\|\cdot\|_{R(\sqrt{-\mathcal{L}^{\mu, \Phi}})}\)-dense in \(R(\sqrt{-\mathcal{L}^{\mu, \Phi}})\).

According to the proof of Proposition \ref{regularDF}, for \(\mu^- \in \mathcal{K}_\infty(G^{\mu^+})\), \(\mathcal{D}(\mathcal{E})\cap C_c\) is \(\mathcal{E}_1^{\mu, \Phi}\)-dense in \(\mathcal{D}(\mathcal{E}^{\mu, \Phi})\), so \(\mathcal{D}(\mathcal{E}) \cap C_c\) is \(\|\cdot\|_{R(\sqrt{-\mathcal{L}^{\mu, \Phi}})}\)-dense in \(R(\sqrt{-\mathcal{L}^{\mu, \Phi}})\).
\end{proof}

We consider another type of norm on  \(R(\sqrt{-\mathcal{L}^{\mu, \Phi}})\) as follows.
\begin{definition}
We define a semi-norm on \(R(\sqrt{-\mathcal{L}^{\mu, \Phi}})\) by
\[\llbracket f \rrbracket_{R(\sqrt{-\mathcal{L}^{\mu, \Phi}})} := \inf\left\{\|u\|_{L^2(E;m)} : u\in \mathcal{D}(\sqrt{-\mathcal{L}^{\mu, \Phi}}), \sqrt{-\mathcal{L}^{\mu, \Phi}}u=f\right\}.\]
\end{definition}

\begin{lemma}
\(\llbracket \cdot \rrbracket_{R(\sqrt{-\mathcal{L}^{\mu, \Phi}})}\) is isometric to the quotient norm on \(\mathcal{D}(\sqrt{-\mathcal{L}^{\mu, \Phi}})/\rm{Ker}(\sqrt{-\mathcal{L}^{\mu, \Phi}})\) induced by the \(L^2(E;m)\)-norm. In particular, \((R(\sqrt{-\mathcal{L}^{\mu, \Phi}}), \llbracket \cdot \rrbracket_{R(\sqrt{-\mathcal{L}^{\mu, \Phi}})})\) is a normed space.
\end{lemma}
\begin{proof}
To prove that the semi-norm \(\llbracket f \rrbracket_{R(\sqrt{-\mathcal{L}^{\mu, \Phi}})}\) is actually a norm, it is enough to show that Ker\((\sqrt{-\mathcal{L}^{\mu, \Phi}})\) is closed in \(L^2(E;m)\). We take a sequence \(\{g_n\}_n \subset \text{Ker}(\sqrt{-\mathcal{L}^{\mu, \Phi}})\) converging to \(g\) in \(L^2(E;m)\). Then \(\{g_n\}_n\) is an \(\mathcal{E}_1^{\mu, \Phi}\)-Cauchy sequence since \(\mathcal{E}^{\mu, \Phi}(g_n,g_n)= \|\sqrt{-\mathcal{L}^{\mu, \Phi}} g_n\|_{L^2(E;m)}^2=0\), so, by the closedness of \((\mathcal{E}^{\mu, \Phi}, \mathcal{D}(\mathcal{E}^{\mu, \Phi}))\), there exists \(v \in \mathcal{D}(\mathcal{E}^{\mu, \Phi})\) such that \(g_n\) converges to \(v\) in \(\mathcal{E}^{\mu, \Phi}_1\). Hence \(v=g\) holds and
\[\|\sqrt{-\mathcal{L}^{\mu, \Phi}} g\|_{L^2(E;m)}^2 = \mathcal{E}^{\mu, \Phi}(g,g)=\lim_n  \mathcal{E}^{\mu, \Phi}(g_n,g_n)=0,\]
and so \({\rm Ker}(\sqrt{-\mathcal{L}^{\mu, \Phi}})\) is a closed space of \(L^2(E;m)\).

The rest follow in the same way as the proof of Lemma \ref{quolem1}.
\end{proof}
Even when \(\mathcal{L}^{\mu, \Phi}\) is subcritical, \(\mathcal{D}(\sqrt{-\mathcal{L}^{\mu, \Phi}})\) is not closed in \(L^2(E;m)\), so \((R(\sqrt{-\mathcal{L}^{\mu, \Phi}}), \llbracket \cdot \rrbracket_{R(\sqrt{-\mathcal{L}^{\mu, \Phi}})})\) is not a Hilbert space in general. We also remark that if \(\sqrt{-\mathcal{L}^{\mu, \Phi}}\) is injective, then the norm \(\llbracket \cdot \rrbracket_{R(\sqrt{-\mathcal{L}^{\mu, \Phi}})}\) on \(R(\sqrt{-\mathcal{L}^{\mu, \Phi}})\) coincides with the norm \(\llbracket \cdot \rrbracket_{R(\sqrt{-\mathcal{L}^{\mu, \Phi}})}\) introduced in \cite{S26}. In particular, if \(\mathcal{L}^{\mu, \Phi}\) is subcritical, then \(\sqrt{-\mathcal{L}^{\mu, \Phi}}\) is injective.

We consider the following wave equation.
\begin{eqnarray}
\begin{cases}
\frac{\partial^2}{\partial t^2} w(x,t) = 
\mathcal{L}^{\mu, \Phi} w(x,t)\hspace{5mm}\text{for\ }(x,t)\in E\times (0,\infty)\\
\frac{\partial}{\partial t} w(x,0) = g(x)\hspace{5mm}\text{for\ }x\in E\\
w(x,0) = 0\hspace{5mm}\text{for\ }x\in E,
\end{cases}\label{eq:wave}
\end{eqnarray}
where \(g\in L^2(E;m)\). By the standard strategy for the solvability 
of abstract evolution equations of second-order 
(see, for example, Reed--Simon \cite[Section X.13]{RS75}), 
the existence and uniqueness of solutions to \eqref{eq:wave} are verified with 
the energy conservation law
\begin{eqnarray*}
\|\partial_t w(\cdot,t)\|_{L^2(E;m)}^2 +\|\sqrt{-\mathcal{L}^{\mu,\Phi}} w(\cdot,t)\|_{L^2(E;m)}^2 &=& \|\partial_t w(\cdot,0)\|_{L^2(E;m)}^2 +\|\sqrt{-\mathcal{L}^{\mu,\Phi}} w(\cdot,0)\|_{L^2(E;m)}^2\\
&=&\|g\|_{L^2(E;m)}^2
\end{eqnarray*}
(see also Engel--Nagel \cite[\S VI.3.c]{EN00}).
Here we consider the case where $g\in C_c$.

\begin{theorem}\label{wave_bdd}
Assume that \(\inf_{x\in K} G_1^{\mu, \Phi} v(x)>0\) for any compact set \(K\) and a strictly positive bounded function \(v \in L^1(E;m)\). If \(\mathcal{L}^{\mu, \Phi}\) is subcritical, then the solution 
$w$ to \((\ref{eq:wave})\) with $g\in C_c(E)$ is always uniformly bounded in $L^2(E;m)$.
\end{theorem}

\begin{proof}
Suppose that \(\mathcal{L}^{\mu, \Phi}\) is subcritical 
and $g\in C_c(E)$. 
Then in view of Theorem \ref{extSoba2}, 
there exists $g_*\in \mathcal{D}(\mathcal{E}^{\mu, \Phi})$ such that 
$g=\sqrt{-\mathcal{L}^{\mu, \Phi}}g_*$. 
Let $w_*$ be the unique solution of the problem
\begin{equation*}
\begin{cases}
\frac{\partial^2}{\partial t^2} w_*(x,t) = 
\mathcal{L}^{\mu, \Phi} w_*(x,t)\hspace{5mm}\text{for\ }(x,t)\in E\times (0,\infty)\\
\frac{\partial}{\partial t} w_*(x,0) = g_*(x)\hspace{5mm}\text{for\ }x\in E\\
w_*(x,0) = 0\hspace{5mm}\text{for\ }x\in E. 
\end{cases}
\end{equation*}
Then from the uniqueness of solutions to \eqref{eq:wave}, we can see that 
$w=\sqrt{-\mathcal{L}^{\mu, \Phi}}w_*$. 
Therefore the energy conservation law for $w_*$ implies the uniform upper bound for $w$
in the following way:
\begin{align*}
\llbracket g \rrbracket_{R(\sqrt{-\mathcal{L}^{\mu, \Phi}})}^2 = \|g_*\|_{L^2(E;m)}^2
&=
\|\partial_t w_*(\cdot,t)\|_{L^2(E;m)}^2 +\|\sqrt{-\mathcal{L}^{\mu,\Phi}} w_*(\cdot,t)\|_{L^2(E;m)}^2 
\\
&\geq 
\|\sqrt{-\mathcal{L}^{\mu,\Phi}} w_*(\cdot,t)\|_{L^2(E;m)}^2
\\
&\geq 
\|w(\cdot,t)\|_{L^2(E;m)}^2.
\end{align*}
The proof is completed.
\end{proof}

\begin{remark}
In \cite{S26}, Theorem \ref{wave_bdd} is proved in the case of \(\mathcal{L}=\Delta\), absolutely continuous \(\mu\) and \(\Phi(\lambda)=\lambda^\beta\) by using \(\|\cdot\|_{R(\sqrt{-\mathcal{L}^{\mu, \Phi}})}\) and \(\llbracket \cdot \rrbracket_{R(\sqrt{-\mathcal{L}^{\mu, \Phi}})}\).
\end{remark}

\begin{remark}
The boundedness of $L^2(E;m)$-norm 
is optimal in the following sense. 
In \cite{YZ06}, the lower bound for the $L^2$-norm of solutions 
to nonlinear wave equation of the form $\frac{\partial^2 t}{\partial t^2}u=\Delta u+|u|^p$ in $\mathbb{R}^N$ ($p>1$)
was obtained. As in the same way, one can also derive
the lower bound for $L^2$-norm of solutions to the linear wave equation 
$\frac{\partial^2 t}{\partial t^2}u=\Delta u$ in $\mathbb{R}^N$ 
(a precise description can be found in \cite{ISW19}). 
\end{remark}

Combining Theorem \ref{wave_bdd} with Corollary \ref{cor_extSoba2}, we obtain the following useful criterion.
\begin{theorem}\label{wave_bdd_SF}
We assume the condition {\rm (SF)} for \(\mathcal{L}\), and  \(\mu^+\in \mathcal{K}\). If \(\mathcal{L}^{\mu, \Phi}\) is subcritical, then the solution 
$w$ to \((\ref{eq:wave})\) with $g\in C_c(E)$ is always uniformly bounded in $L^2(E;m)$.
\end{theorem}

Next we propose a sufficient condition for subcriticality for a subordinated Schr\"{o}dinger operator via the knowledge of solutions to the corresponding wave equation. 

\begin{theorem}\label{oppose_wave_bdd}
We assume the condition {\rm (SF)} for \(\mathcal{L}\), and  \(\mu^+\in \mathcal{K}\). If, for every $g\in C_c(E)$, 
the unique solution $w$ of the corresponding wave equation 
\begin{eqnarray}
\begin{cases}
\frac{\partial^2}{\partial t^2} w(x,t) = 
\mathcal{L}^\mu w(x,t)\hspace{5mm}\text{for\ }(x,t)\in E\times (0,\infty)\\
\frac{\partial}{\partial t} w(x,0) = g(x)\hspace{5mm}\text{for\ }x\in E\\
w(x,0) = 0\hspace{5mm}\text{for\ }x\in E,
\end{cases}\label{eq:wave2}
\end{eqnarray}
is bounded in $L^2(E,m)$. Then for every $\beta\in (0,1)$, 
$(-\mathcal{L}^\mu)^{\beta}$ is subcritical.
\end{theorem}

\begin{proof} We set a Bernstein function \(\Phi_\beta(\lambda):=\lambda^\beta\) for \(\beta \in (0,1)\). Let $g\in C_c(E)$ be arbitrary fixed.  We recall that 
\[
e^{t\mathcal{L}^\mu}g
=
\frac{1}{2\pi^{\frac{1}{2}}t^{\frac{3}{2}}}
\int_0^\infty \sigma e^{-\frac{\sigma^2}{4t}}w(\sigma)\,d\sigma, \quad t>0.
\]
This formula is derived from the Hadamard transmutation formula
(see \cite[Lemma 3.6]{S26}). Therefore if $\sup_{t\geq 0}\|w(t)\|_{L^2(E,m)}\leq M$ for some $M\geq 0$, then 
we have for every $t\geq 1$, 
\[
\|e^{t\mathcal{L}^\mu}g\|_{L^2(E,m)}
\leq 
\frac{M}{2\pi^{\frac{1}{2}}t^{\frac{3}{2}}}
\int_0^\infty \sigma e^{-\frac{\sigma^2}{4t}}\,d\sigma
=
\frac{M}{\pi^{\frac{1}{2}}t^{\frac{1}{2}}}.
\]
This shows that
\begin{eqnarray*}
\langle g, G^{\mu, \Phi_\beta}g \rangle_m &=& \left\langle g, \int_0^\infty e^{t\mathcal{L}^\mu}g\,\Gamma(\beta)^{-1}t^{\beta -1} \,dt \right\rangle_m\\
& = & \Gamma(\beta)^{-1}\int_0^\infty \langle g, e^{t\mathcal{L}^\mu}g\rangle_m\,t^{\beta -1} \,dt \\
&= &\Gamma(\beta)^{-1} \int_0^\infty \|e^{\frac{t}{2}\mathcal{L}^\mu}g\|_{L^2(E,m)}^2 t^{\beta-1}\,dt<\infty
\end{eqnarray*}
for every $\beta\in (0,1)$. Here we used a fact \cite[Example 5.8]{SV09} that a \(0\)-potential measure for \(\Phi_\beta\) is \(\Gamma(\beta)^{-1}t^{\beta-1}\,dt\). By Theorem \ref{thm_range}, 
$g\in R(\sqrt{(-\mathcal{L}^\mu)^{\beta}})$.
As a consequence, 
we deduce $C_c(E)\subset R(\sqrt{(-\mathcal{L}^\mu)^{\beta}})$. 
By Corollary \ref{cor_extSoba2}, we can conclude that 
$(-\mathcal{L}^\mu)^{\beta}$ is subcritical.  
\end{proof}

\section{Examples}\label{sec:Example}

\begin{example}
Let \(E\) be a connected domain of \(\mathbb{R}^d\),  \(m\) be the Lebesgue measure and \(\mathcal{L}:=\Delta\). We consider the Laplace operator \(\Delta\) on \(E\) with the Dirichlet boundary condition. The corresponding irreducible regular Dirichlet form \((\mathcal{E}, \mathcal{D}(\mathcal{E}))\) on \(L^2(E;m)\) is
\begin{equation}
\begin{cases}
\mathcal{E}(f,g) := \frac{1}{2} \int_E \nabla f\cdot \nabla g\,dx, \\
\mathcal{D}(\mathcal{E}) := H^1_0(E):= \overline{C_c^\infty(E)}^{\mathcal{E}_1}
\end{cases}
\end{equation}
where the derivatives are taken in the Schwartz distribution sense. In this example, we extend the function \(f\) on \(E\) to \(\mathbb{R}^d\) by setting it to \(0\) on \(\mathbb{R}^d\setminus E\). Denote by \(L^p_{{\rm unif}}(\mathbb{R}^d)\) the space of functions \(f\) satisfying \[\sup_{x\in \mathbb{R}^d} \int_{\{|x-y|\leq 1\}} |f(y)|^p\,dy < \infty\]
and we set \(d\mu := V\, dx\) for \(V_+ \in L^\infty_{loc}\) and \(V_- \in L^p_{{\rm unif}}(\mathbb{R}^d)\) with \(p>d/2\) if \(d\geq 2\) or \(p\geq 1\) if \(d=1\). Then \(\mu_+ \in \mathcal{K}_{loc}\) since 
\[\sup_{x\in E} G_{\alpha}(\mu_{+}1_K )(x) \leq \|V_+1_K\|_{\infty}\sup_{x\in E} G_{\alpha}m (x) \leq \frac{\|V_+\|_{\infty}}{\alpha} \to 0\]
as \(\alpha \to \infty\) for any compact set \(K\), and \(\mu_- \in \mathcal{K}_{loc}\) by \cite[Theorem 1.4 (iii)]{AS82} (or \cite[Theorem 1.2]{KT07}). We assume that \((\mathcal{E}^{\mu^+}, \mathcal{D}(\mathcal{E}^{\mu^+}))\) is transient. For example, if either \(d\geq 3\) or \(m(\mathbb{R}^d\setminus E)>0\) or \(V_+ \not \equiv 0\) holds, then \((\mathcal{E}^{\mu^+}, \mathcal{D}(\mathcal{E}^{\mu^+}))\) is transient.

The following cases satisfy the condition (IB), that is,  \(b>0\) or \(\inf {\rm supp} \nu =0\). See \cite{S99, B99} for examples. Hence we can characterize criticality for the following operators by Theorem \ref{subodinated_subcri}, \ref{subodinated_cri}, \ref{subodinated_supercri}, and obtain the equivalent condition for subcriticality by Theorem \ref{extSoba2}. 

For \(\Phi(\lambda)= \lambda^{\alpha/2}\) with \(0<\alpha <2\), we can consider the operator \((-\Delta+V)^{\alpha/2}\) for \(0<\alpha <2\).

For a compound Poisson subordinator \(\Phi(\lambda)= a \frac{\lambda}{\lambda + c}\) with \(0<a,c\),  we can consider the operator \(a \frac{\Delta-V}{\Delta-V -c}\).

For a Gamma subordinator \(\Phi(\lambda)= a \log{(1+\lambda/c)}\) with \( a,c>0\), we can consider the operator \(a \log{(1-(\Delta-V)/c)}\).

For an inverse Gaussian subordinator \(\Phi(\lambda)= a (\sqrt{2\lambda + c^2}-c )\) with \(0<a,c\), we can consider the operator \(a (\sqrt{-2(\Delta-V) + c^2}-c )\).

For a relativistic stable subordinator \(\Phi(\lambda)= (\lambda+ m^{2/\alpha})^{\alpha/2}-m\) with \(0<\alpha <2\) and \(m>0\), we can consider the operator \((-\Delta+V+ m^{2/\alpha})^{\alpha/2}-m\). In the case of \(V=0\), a Markov process corresponding to \((-\Delta + m^{2/\alpha})^{\alpha/2}-m\) is called a relativistic \(\alpha\)-stable process. See \cite{SV09} for details.
\end{example}

\begin{example}[The Hardy inequality for \((-\Delta)^{\alpha/2}\)]\label{ex.Hardy}
We consider the Hardy inequality for \((-\Delta)^{\alpha/2}\) with \(0<\alpha \leq 2\).  Let \(\alpha < d\) and \(X\) be an \(\alpha\)-stable process on \(\mathbb{R}^d\), which corresponds to \((-\Delta)^{\alpha/2}\). The condition \(\alpha < d\) is needed for the transience of \(X\). We remark that \(X\) is Brownian motion if \(\alpha =2\). The associated irreducible transient regular Dirichlet form \((\mathcal{E}, \mathcal{D}(\mathcal{E}))\) on \(L^2(\mathbb{R}^d ;dx)\) is, for \(\alpha=2\),
\begin{eqnarray*}
\begin{cases}
\mathcal{E}(f,g):= \frac{1}{2} \int_{\mathbb{R}^d} \nabla f\cdot \nabla g \,dx \\
\mathcal{D}(\mathcal{E}):=H^1(\mathbb{R}^d):=\{f\in L^2(\mathbb{R}^d) : \nabla f \in L^2(\mathbb{R}^d)\}
\end{cases}
\end{eqnarray*}
where the derivatives are taken in the Schwartz distribution sense, and, for \(0<\alpha <2\),
\begin{eqnarray*}
\begin{cases}
\mathcal{E}(f,f):= \frac{\mathcal{A}(d,\alpha)}{2}\iint_{\mathbb{R}^d\times \mathbb{R}^d \setminus \rm{diag}} \frac{(f(x)-f(y))^2}{|x-y|^{d+\alpha}}\,dx\,dy,\\
\mathcal{D}(\mathcal{E}) := \left\{f\in L^2(\mathbb{R}^d) : \iint_{\mathbb{R}^d\times \mathbb{R}^d \setminus \rm{diag}} \frac{(f(x)-f(y))^2}{|x-y|^{d+\alpha}}\,dx\,dy<\infty \right\},
\end{cases}
\end{eqnarray*}
where \({\rm diag}:=\{(x,y)\in \mathbb{R}^d\times \mathbb{R}^d : x=y\}\) and
\[\mathcal{A}(d,\alpha):= \frac{\alpha 2^{\alpha-1} \Gamma((d+\alpha)/2)}{\pi^{d/2} \Gamma(1-\alpha/2)}.\]
In these cases, \(X\) satisfies the strong Feller condition (SF). The following Hardy inequality is well-known. See \cite{H77, DDM08} for example. 
\begin{eqnarray}
\lambda_* \int_{\mathbb{R}^d} \frac{|f(x)|^2}{|x|^2}\,dx \leq \mathcal{E}(f,f) \label{eq:ex_Hardy_1}
\end{eqnarray}
for \(f\in \mathcal{D}(\mathcal{E})\), where 
\[\lambda_*:=2^{\alpha-1}\left(\frac{\Gamma(\frac{d+\alpha}{4})}{\Gamma(\frac{d-\alpha}{4})}\right)^2.\]

Let \(d\mu_\lambda(x) = -\lambda |x|^{-2}\,dx\). Then a Schr\"{o}dinger form \((\mathcal{E}^{\mu_{\lambda_*}}, \mathcal{D}(\mathcal{E}^{\mu_{\lambda_*}}))\) is critical. For \(\lambda <\lambda_*\) (resp. \(\lambda >\lambda_*\)), a Schr\"{o}dinger form \((\mathcal{E}^{\mu_\lambda}, \mathcal{D}(\mathcal{E}^{\mu_\lambda}))\) is subcritical (resp. supercritical).

By Corollary \ref{subcrinomama}, for \(\lambda < \lambda_*\) and any Bernstein function \(\Phi\) satisfying (IB), \(-\Phi((-\Delta)^{\alpha/2} + \mu_\lambda)\) is subcritical. By Theorem \ref{wave_bdd_SF}, the solution to the wave equation \((\ref{eq:wave})\) for \(-\Phi((-\Delta)^{\alpha/2}+ \mu_\lambda)\) is uniformly bounded in \(L^2(\mathbb{R}^d)\). In particular, these hold for \(-((-\Delta)^{\alpha/2} + \mu_\lambda)^{\beta/2}\) with \(0<\alpha <2\), \(-a \frac{(-\Delta)^{\alpha/2} + \mu_\lambda}{(-\Delta)^{\alpha/2} + \mu_\lambda + c}\), \(-a \log{(1+((-\Delta)^{\alpha/2} + \mu_\lambda)/c)}\), \( -a (\sqrt{2((-\Delta)^{\alpha/2} + \mu_\lambda) + c^2}-c )\) with \(0<a,c\), and \(-((-\Delta)^{\alpha/2}+ \mu_\lambda+ m^{2/\beta})^{\beta/2}+m\) with \(0<\beta <2\) and \(m>0\).

We consider the criticality of a subordinated Schr\"{o}dinger form \((\mathcal{E}^{\mu_{\lambda_*}, \Phi}, \mathcal{D}(\mathcal{E}^{\mu_{\lambda_*}, \Phi}))\) for a \(\beta/2\)-stable subordinatior \(\Phi(x)=x^{\beta/2}\) with \(0<\beta <2\).

For a transition density function \(p^{\mu_{\lambda_*}}(t,x,y)\) of \(T_t^{\mu_{\lambda_*}}\), according to \cite{MS04}, \cite[Section 10.4]{G06}, for \(\alpha =2 < d\), it holds that
\begin{eqnarray}
\nonumber p^{\mu_{\lambda_*}}(t,x,y) & \asymp &\frac{1}{t^{d/2}} \left( 1+\frac{\sqrt{t}}{|x|} \right)^{\delta} \left( 1+\frac{\sqrt{t}}{|y|} \right)^{\delta} \exp{\left(-c\frac{|x-y|^2}{t} \right)}\\
&\asymp & \frac{1}{t^{d/2}} \left( 1+\frac{t^{\delta/2}}{|x|^\delta} \right) \left( 1+\frac{t^{\delta/2}}{|y|^\delta} \right)\exp{\left(-c\frac{|x-y|^2}{t} \right)} \label{eq:ex_Hardy_HKE_1}
\end{eqnarray}
where \(\delta = (d-2)/2\). According to \cite{BGJP19}, for \(0<\alpha <2\wedge d\), it holds that
\begin{equation}
p^{\mu_{\lambda_*}}(t,x,y) \asymp \left(1+\frac{t^{\delta/\alpha}}{|x|^{\delta}}\right)\left(1+\frac{t^{\delta/\alpha}}{|y|^{\delta}}\right) \left(t^{-d/\alpha} \wedge \frac{t}{|x-y|^{d+\alpha}}  \right), \label{eq:ex_Hardy_HKE_2}
\end{equation}
where \(\delta = (d-\alpha)/2\). Here and throughout this section, \(f(x) \asymp g(x)\exp{(-cu(x))}\) means that there exist \(C_1, C_2, c_1, c_2 >0\) such that \(C_1 g(x)\exp{(-c_1u(x))} \leq f(x) \leq C_2 g(x)\exp{(-c_2u(x))}\) holds for any \(x\). In order to verify subcriticality, it is enough to show the existence of \(0\)-order Green's kernel \(r^{\mu_{\lambda_*}, \Phi}\) for \(T_t^{\mu_{\lambda_*}, \Phi}\). 
By \cite[Example 37.19]{S99}, we have \[r^{\mu_{\lambda_*}, \Phi}(x,y):=\int_0^\infty  \int_0^\infty p^{\mu_{\lambda_*}}(s,x,y)\,d\eta_t(s)\,dt = C \int_0^\infty p^{\mu_{\lambda_*}}(s,x,y) \, s^{\beta/2-1}\,ds.\] 
Hence, for \(0< \alpha=2 < d\), by \((\ref{eq:ex_Hardy_HKE_1})\) and \(\beta/2 -1 -d/2 +\delta = \beta/2-2 < -3/2\), we obtain
\begin{eqnarray*}
r^{\mu_{\lambda_*}, \Phi}(x,y) & \asymp & \int_0^\infty s^{\beta/2 -1} \frac{1}{s^{d/2}} \left( 1+\frac{s^{\delta/2}}{|x|^\delta} \right) \left( 1+\frac{s^{\delta/2}}{|y|^\delta} \right)\exp{\left(-c\frac{|x-y|^2}{s} \right)}\,ds \\
& \asymp & \int_0^\infty  \left( s^{\beta/2 -1-d/2}+ s^{\beta/2 -1-d/2 + \delta/2}(|x|^{-\delta}+|y|^{-\delta})\right.\\
&& \hspace{20mm}+ \left. s^{\beta/2 -1-d/2 + \delta} |x|^{-\delta}|y|^{-\delta} \right) \exp{\left(-c\frac{|x-y|^2}{s} \right)}\,ds \\
& \asymp & |x-y|^{\beta -d} + |x-y|^{\beta -d+\delta}(|x|^{-\delta}+|y|^{-\delta}) + |x-y|^{\beta -d+2\delta}|x|^{-\delta}|y|^{-\delta}\\
&\asymp & \frac{1}{|x-y|^{d-\beta}} \left(1+\frac{|x-y|^\delta}{|x|^\delta} \right)\left(1+\frac{|x-y|^\delta}{|y|^\delta} \right).
\end{eqnarray*}

For \(0< \alpha <2 \wedge d\), we set \(c:= |x-y|^\alpha \) and, by \((\ref{eq:ex_Hardy_HKE_2})\), we obtain
\begin{eqnarray*}
r^{\mu_{\lambda_*}, \Phi}(x,y) & \asymp & \int_0^c s^{\beta/2 } c^{-(d+\alpha)/\alpha} (1+s^{\delta /\alpha}|x|^{-\delta} + s^{\delta /\alpha} |y|^{-\delta} + s^{2\delta /\alpha} |x|^{-\delta} |y|^{-\delta})\,ds\\
&&   +  \int_{c} ^\infty s^{\beta/2 -1 -d/\alpha} (1+s^{\delta /\alpha}|x|^{-\delta} + s^{\delta /\alpha} |y|^{-\delta} + s^{2\delta /\alpha} |x|^{-\delta} |y|^{-\delta})\,ds\\
& \asymp & c^{-(d+\alpha)/\alpha} \left[s^{\beta/2 +1} + s^{\beta/2 + 1+\delta /\alpha}|x|^{-\delta} + s^{\beta/2+1+\delta /\alpha} |y|^{-\delta} + s^{\beta/2 +1 + 2\delta /\alpha} |x|^{-\delta} |y|^{-\delta} \right]_0^c \\
&& + \left[s^{\beta/2 -d/\alpha} + s^{\beta/2 -d/\alpha+ \delta /\alpha}|x|^{-\delta} + s^{\beta/2 -d/\alpha+\delta /\alpha} |y|^{-\delta} + s^{\beta/2 -d/\alpha+ 2\delta /\alpha} |x|^{-\delta} |y|^{-\delta} \right]_c^\infty \\
&\asymp & c^{\beta /2- d/\alpha} + c^{\beta /2- d/(2\alpha)-1/2}(|x|^{-\delta} + |y|^{-\delta}) + c^{\beta /2- d/\alpha + 2\delta/\alpha}|x|^{-\delta} |y|^{-\delta}\\
&\asymp & \frac{1}{|x-y|^{d-\frac{\alpha \beta}{2}}} \left(1+\frac{|x-y|^\delta}{|x|^\delta} \right)\left(1+\frac{|x-y|^\delta}{|y|^\delta} \right).
\end{eqnarray*}
Here we used \(\beta /2 -d/\alpha + 2\delta/\alpha = \beta/2 -1 <0\).

In both cases of \(\alpha =2\) and \(\alpha <2\), there exists \(0\)-order Green's kernel \(r^{\mu_{\lambda_*}, \Phi}\) for \(\Phi((-\Delta)^{\alpha/2}+\mu_{\lambda_*}) = ((-\Delta)^{\alpha/2} - \lambda_* |x|^{-2})^{\beta/2}\) satisfying
\begin{equation}
r^{\mu_{\lambda_*}, \Phi}(x,y) \asymp \frac{1}{|x-y|^{d-\frac{\alpha \beta}{2}}} \left(1+\frac{|x-y|^\delta}{|x|^\delta} \right)\left(1+\frac{|x-y|^\delta}{|y|^\delta} \right), \label{eq:ex_Hardy_HKE_3}
\end{equation}
and so \(((-\Delta)^{\alpha/2} - \lambda_* |x|^{-2})^{\beta/2}\) is subcritical for any \(0< \alpha \leq 2\) and \(0<\beta <2\). We note that we can also obtain the subcriticality of \(((-\Delta)^{\alpha/2} -\lambda |x|^{-2})^{\beta/2}\) for \(\lambda < \lambda_*\) directly in a similar way to the above calculations.

Since \(X\) satisfies the strong Feller condition (SF), by Theorem \ref{wave_bdd_SF}, the solution \(W^{\alpha, \beta}(x,t)\) to the following wave equation \((\ref{eq:ex_Hardy_wave})\) is uniformly bounded in \(L^2(\mathbb{R}^d)\) for \(\lambda <\lambda_*, 0< \beta \leq 2\) and \(\lambda = \lambda_*, 0< \beta < 2\).
\begin{eqnarray}
\begin{cases}
\frac{\partial^2}{\partial t^2} w(x,t) = -((-\Delta)^{\alpha/2} - \lambda |x|^{-2})^{\beta/2} w(x,t)\hspace{5mm}\text{for\ }(x,t)\in  \mathbb{R}^d\times (0,\infty)\\
\frac{\partial}{\partial t} w(x,0) = g(x)\in C_c(\mathbb{R}^d)\hspace{5mm}\text{for\ }x\in  \mathbb{R}^d\\
w(x,0) = 0\hspace{5mm}\text{for\ }x\in \mathbb{R}^d
\end{cases}\label{eq:ex_Hardy_wave}
\end{eqnarray}

We note that unbounded solutions \(W^{2, 2}(x,t)\) to  \((\ref{eq:ex_Hardy_wave})\)  for \(\lambda =\lambda_*, \alpha= \beta =2\) are constructed in \cite[Proposition 1.7]{S26}.

By (\ref{eq:ex_Hardy_HKE_1}) and (\ref{eq:ex_Hardy_HKE_2}) and checking the existence of a \(0\)-order resolvent kernel of the subordinated Dircihlet form \((\mathcal{E}^{\mu_{\lambda_*}, h, \Phi}, \mathcal{D}(\mathcal{E}^{\mu_{\lambda_*}, h, \Phi}))\), we can decide the criticality and subcriticality for certain subordinators as follows. See \cite[Section 5.2.2, Theorem 5.17, Proposition 5.22]{SV09} for details of these  subordinators. We note that a potential density is a density function of a \(0\)-order potential measure.

A Gamma subordinator \(\Phi(\lambda)= a \log{(1+\lambda/c)}\) has a potential density \(\int_0^\infty c^{at}\Gamma(at)^{-1} s^{at-1}e^{-cs}\,dt\) comparable to \(1\) as \(s\to  \infty\), so
\[a \log{\left(1+\left((-\Delta)^{\alpha/2} - \lambda_* |x|^{-2}\right)/c \right)} \text{\ \ is\  critical.}\]

A relativistic stable subordinator \(\Phi(\lambda)= (\lambda+ m^{2/\beta})^{\beta/2}-m\) with \(0<\beta <2\) and \(m>0\) has a potential density \(e^{-m^{{2s}/{\beta}}} s^{\beta/2-1} \sum_{n=0}^\infty \frac{(ms^{\beta/2})^n}{\Gamma(\beta(1+n)/2)}\) comparable to \(1\) as \(s\to  \infty\), so 
\[\left(\left((-\Delta)^{\alpha/2} - \lambda_* |x|^{-2}\right)+ m^{2/\beta}\right)^{\beta/2}-m \text{\ \ is\  critical.}\]

A subordinator \(\Phi(\lambda)=\lambda^{\delta/2} \log{(1+\lambda)}^{\beta/2}\) with \(0<\delta<2, 0<\beta<2-\delta \) has a potential density comparable to \(s^{(\delta + \beta)/2-1}\) as \(s\to \infty\), so 
\[\left((-\Delta)^{\alpha/2} - \lambda_* |x|^{-2}\right)^{\delta/2} \log{\left(1+\left((-\Delta)^{\alpha/2} - \lambda_* |x|^{-2}\right)\right)}^{\beta/2} \text{\ \ is\  subcritical.}
\]

A subordinator \(\Phi(\lambda)=\lambda^{\delta/2} \log{(1+\lambda)}^{-\beta/2}\) with \(0<\delta<2, 0<\beta<\delta \) has a potential density comparable to \(s^{(\delta - \beta)/2-1}\) as \(s\to \infty\), so
\[\left((-\Delta)^{\alpha/2} - \lambda_* |x|^{-2}\right)^{\delta/2} \log{\left(1+\left((-\Delta)^{\alpha/2} - \lambda_* |x|^{-2}\right)\right)}^{-\beta/2} \text{\ \ is\  subcritical.}
\]

Bessel subordinators \(\Phi(\lambda)=\log{((1+\lambda)+\sqrt{( 1+\lambda )^2 -1})}\) and \(\Phi(\lambda)=\left(\log{((1+\lambda)+\sqrt{( 1+\lambda )^2 -1})} \right)^2\)  have potential density comparable to \(1\) and \(s^{-1/2}\) as \(s\to \infty\), respectively, so 

\[\log{\left((1+\left((-\Delta)^{\alpha/2} - \lambda_* |x|^{-2}\right))+\sqrt{\left( 1+\left((-\Delta)^{\alpha/2} - \lambda_* |x|^{-2}\right) \right)^2 -1}\right)} \text{\ \ is\  critical.}\]
However 
\[\left(\log{\left((1+\left((-\Delta)^{\alpha/2} - \lambda_* |x|^{-2}\right))+\sqrt{\left( 1+\left((-\Delta)^{\alpha/2} - \lambda_* |x|^{-2}\right)\right)^2 -1}\right)}\right)^2 \text{\ \ is\  subcritical.}\]
We can also obtain the result on uniformly boundedness of the solutions to wave equations \((\ref{eq:wave})\) with an initial function \(g\in C_c(\mathbb{R}^d)\) for \(\mathcal{L}^{\mu, \Phi} = \Phi((-\Delta)^{\alpha/2} - \lambda_* |x|^{-2})\) for above subordinators \(\Phi\) making \(\Phi((-\Delta)^{\alpha/2} - \lambda_* |x|^{-2})\) subcritical. 
\end{example}

\begin{example}[The trace Hardy inequality for \(-\Delta\)]\label{ex_traceHardy1}
We consider the upper half space \(\mathbb{R}^d_+:=\{x=(x', x_d) : x'\in \mathbb{R}^{d-1},\, 0< x_d \in \mathbb{R}\}\) for \(d\geq 3\), and the reflecting Brownian motion on \(\overline{\mathbb{R}^d_+}:= \{x=(x', x_d) : x'\in \mathbb{R}^{d-1},\, 0\leq  x_d \in \mathbb{R}\}\). The associated irreducible transient regular Dirichlet form \((\mathcal{E}, \mathcal{D}(\mathcal{E}))\) on \(L^2(\overline{\mathbb{R}^d_+};dx)\) is
\begin{eqnarray*}
\begin{cases}
\mathcal{E}(f,g):= \frac{1}{2} \int_{\mathbb{R}^d_+} \nabla f\cdot \nabla g \,dx \\
\mathcal{D}(\mathcal{E}):=H^1(\mathbb{R}^d_+;dx):=\{f\in L^2(\overline{\mathbb{R}^d_+}) : \nabla f \in L^2(\mathbb{R}^d_+)\}
\end{cases}
\end{eqnarray*}
where the derivatives are taken in the Schwartz distribution sense. See \cite[Example 2.2.4]{CF12} for details.

The following trace hardy inequality is known. See \cite{DJL21} for example.
\begin{eqnarray}
\lambda_* \int_{\partial \mathbb{R}^d_+} \frac{|f(x',0)|^2}{|x'|}\,dx' \leq \frac{1}{2} \int_{\mathbb{R}^d_+} |\nabla f|^2 \,dx \label{eq:ex_traceHardy_a}
\end{eqnarray}
for \(f\in \mathcal{D}(\mathcal{E})\), where 
\[\lambda_*:=\left(\frac{\Gamma(d/4)}{\Gamma((d-2)/4)}\right)^2.\]

Let \(v(x):=|x|^{-\beta}\) for \(\beta \in \mathbb{R}\) and we consider \(d\rho := v dx' d\delta_0\), where \(\delta_0\) is a Dirac measure at \(0\). Since the one point \(0\) has zero capacity, we prove that a family of compact sets \(F_k:=\{\frac{1}{n} \leq |x| \leq n\}\) constitutes a nest attached to \(\rho\). We take \(\varphi \in C^\infty(\mathbb{R}^d)\) satisfying \(\varphi=0\) on \(\mathbb{R}^{d-1} \times \{1 \leq |x_d|\}\) and \(\varphi=1\) on \(\mathbb{R}^{d-1} \times \{|x_d|\leq 1/2\}\). Then, for any \(f\in \mathcal{D}(\mathcal{E})\cap C_c\), we have
\[f(x',0)=-\int_0^1 \frac{\partial}{\partial x_d}\left(\varphi(x_d) f(x,x_d) \right)\,dx_d\]
and, by the Cauchy--Schwarz inequality, 
\begin{eqnarray*}
\int_{\mathbb{R}^d \cap F_n} |f(x)|^2\,d\rho(x) &=& \int_{\{\frac{1}{n} \leq |x'| \leq n\}} \frac{|f(x',0)|^2}{|x'|^\beta}\,dx'\\
&\leq & n^{|\beta|} \int_{\mathbb{R}^{d-1}} |f(x',0)|^2\,dx'\\
&\leq & n^{|\beta|} \int_{\mathbb{R}^{d-1}} \int_0^1 \left|\frac{\partial}{\partial x_d}\left(\varphi(x_d) f(x,x_d) \right)\right|^2\,dx_d\,dx'\\
&\leq & n^{|\beta|} C \, \mathcal{E}_1(f,f)
\end{eqnarray*}
for some positive constant \(C\). Hence \(1_{F_n}\rho\) is a Radon measure of finite energy integral, and so \(\rho\) is a smooth measure. Moreover \(\rho\) is a Radon measure if and only if \(\beta <d-1\). In particular, by considering \(\mathbb{R}^{d-1}\), \(|x'|^{-1}\,dx'\) appearing in the trace Hardy inequality is a Radon measure.

Let 
\[d\mu_\lambda(x) = -\lambda |x'|^{-1}\,dx'.\]
The trace Hardy inequality \((\ref{eq:ex_traceHardy_a})\) also follows from \cite[Theorem 5.6]{TU23}. By \cite[Lemma 5.1, 5.2]{TU23}, \((\mathcal{E}^{\mu_\lambda}, \mathcal{D}(\mathcal{E}) \cap C_c(\mathbb{R}^d\setminus \{0\}))\) is non-negative definite, the closure of \( \mathcal{D}(\mathcal{E}) \cap C_c(\mathbb{R}^d\setminus \{0\})\) coincides with that of \(\mathcal{D}(\mathcal{E}) \cap C_c(\mathbb{R}^d)\), and \(\mu_\lambda \in \mathcal{K}_{\rm loc}(\mathcal{E}|_{\mathbb{R}^d\setminus \{0\}})\). Hence we can consider a Schr\"{o}dinger form \((\mathcal{E}^{\mu_{\lambda}}, \mathcal{D}(\mathcal{E}^{\mu_{\lambda}}))\) and it holds that a Schr\"{o}dinger form \((\mathcal{E}^{\mu_{\lambda_*}}, \mathcal{D}(\mathcal{E}^{\mu_{\lambda_*}}))\) is critical. This method for deriving the criticality of \((\mathcal{E}^{\mu_{\lambda_*}}, \mathcal{D}(\mathcal{E}^{\mu_{\lambda_*}}))\) using \cite[Theorem 5.6]{TU23} is similar to the next example, so refer to that as well. By the Poincar\'{e} inequality, we have
\[\sup_{f\in \mathcal{D}(\mathcal{E})} \frac{\int_K |f| \, dx}{\sqrt{\mathcal{E}(f,f)}}< \infty \]
for any compact set \(K\). For \(\lambda >\lambda_*\), since it holds that \(\inf{\{\mathcal{E}(f,f) : \int |f|^2\,d\mu_\lambda =1 \}}>1\), \((\mathcal{E}^{\mu_\lambda}, \mathcal{D}(\mathcal{E}^{\mu_\lambda}))\) is subcritical by Theorem \ref{Theorem3.5_TU23} (\cite[Theorem 3.5]{TU23}). For a \(0\)-order resolvent kernel \(r_0\), a function \(\int r_0(x,y)|y|^{-d/2} dy\) attains \((\ref{eq:ex_traceHardy_a})\) and so, \((\mathcal{E}^{\mu_\lambda}, \mathcal{D}(\mathcal{E}^{\mu_\lambda}))\) is supercritical for \(\lambda > \lambda_*\). 

By Corollary \ref{subcrinomama}, for \(\lambda < \lambda_*\) and any Bernstein function \(\Phi\) satisfying (IB), \(-\Phi(-\Delta + \mu_\lambda)\) is subcritical and so, by Theorem \ref{wave_bdd_SF}, the solution to the wave equation \((\ref{eq:wave})\) for \(-\Phi(-\Delta + \mu_\lambda)\) is uniformly bounded in \(L^2(\mathbb{R}^d_+)\). In particular, these hold for \(-(-\Delta + \mu_\lambda)^{\alpha/2}\) with \(0<\alpha <2\), \(-a \frac{-\Delta + \mu_\lambda}{-\Delta + \mu_\lambda + c}\), \(-a \log{(1+(-\Delta + \mu_\lambda)/c)}\), \( -a (\sqrt{2(-\Delta + \mu_\lambda) + c^2}-c )\) with \(0<a,c\), and \(-(-\Delta + \mu_\lambda+ m^{2/\alpha})^{\alpha/2}+m\) with \(0<\alpha <2\) and \(m>0\).

By using a function 
\begin{equation}
h(x) = \frac{\Gamma(\frac{d-2}{4})^2}{4 \sqrt{\pi} \Gamma(\frac{d }{4})} \frac{1}{|x|^{\frac{d-2}{2}}} \,_2F_1\left(\frac{d-2}{4},\frac{d-2}{4},\frac{d-1}{2};\frac{|x'|^2}{|x|^2} \right), \label{eq:ex_resolventnu2}
\end{equation}
an \(h\)-transformed Dirichlet form \((\mathcal{E}^{\mu_{\lambda_*}, h}, \mathcal{D}(\mathcal{E}^{\mu_{\lambda_*}, h}))\) on \(L^2(\mathbb{R}^d_+;h^2dx)\) is recurrent and it holds that
\[\mathcal{E}^{\mu_{\lambda_*}, h}(f,f) = \frac{1}{2} \int_{\mathbb{R}^d} |\nabla f|^2\,h^2\,dx\]
for \(f\in \mathcal{D}(\mathcal{E}^{\mu_{\lambda_*}, h})\). Hence \((\mathcal{E}^{\mu_{\lambda_*}, h}, \mathcal{D}(\mathcal{E}^{\mu_{\lambda_*}, h}))\) coincides with a Dirichlet form associated with Brownian motion on the weighted manifold \((\mathbb{R}^d, h^2\,dx)\). Since \(\frac{d-1}{2} > 2\,\frac{d-2}{4}\), it holds that \( \,_2F_1\left(\frac{d-2}{4},\frac{d-2}{4},\frac{d-1}{2};1\right)<\infty\) by the Gaussian hypergeometric theorem and so \(h(x)\asymp |x|^{-\frac{d-2}{2}}\). By \cite[Corollary 6.11]{G06}, the heat kernel \(p^{\mu_{\lambda_*},h}\) of \((\mathcal{E}^{\mu_{\lambda_*}, h}, \mathcal{D}(\mathcal{E}^{\mu_{\lambda_*}, h}))\) admits the same estimate as (\ref{eq:ex_Hardy_HKE_1}), that is, for any \(x,y\in \mathbb{R}^d\) and \(t>0\),
\begin{eqnarray}
\nonumber p^{\mu_{\lambda_*}}(t,x,y) \asymp \frac{1}{t^{d/2}} \left( 1+\frac{t^{\delta/2}}{|x|^\delta} \right) \left( 1+\frac{t^{\delta/2}}{|y|^\delta} \right)\exp{\left(-c\frac{|x-y|^2}{t} \right)}
\end{eqnarray}
with \(\delta := (d-2)/2\). In particular, for \(|x|^2<t\), it holds that \(p^{\mu_{\lambda_*},h}(t,x,y)\asymp t^{-1} |x|^{2-d} \exp{(-c\frac{|x-y|^2}{t})}\). In the same way as Example \ref{ex.Hardy}, \((-\Delta - \lambda_* |x'|^{-1})^{\beta/2}\) is subcritical for any \(0<\beta <2\).

Since the heat kernel estimate coincides with that appearing in Example \ref{ex.Hardy}, we can obtain the same results on criticality and subcriticality, and uniform boundedness of wave equations for \(\Phi(-\Delta- \lambda |x'|^{-1})\) by subordinators \(\Phi\) appearing in Example \ref{ex.Hardy}.
\end{example}

\begin{example}[The trace Hardy inequality for \((-\Delta)^{\alpha/2}\)]
Let \(d\geq 2\). We consider \((-\Delta)^{\alpha/2}\) on \(\mathbb{R}^d\) for \(1<\alpha <d \wedge 2\). The assumption \(\alpha < d\) is need for transience and \(1<\alpha\) is need for the smoothness of a trace measure \(\mu_\lambda\) below. We define
\begin{eqnarray*}
\begin{cases}
\mathcal{E}(f,f):= \frac{\mathcal{A}(d,\alpha)}{2}\iint_{\mathbb{R}^d\times \mathbb{R}^d \setminus \rm{diag}} \frac{(f(x)-f(y))^2}{|x-y|^{d+\alpha}}\,dx\,dy,\\
\mathcal{D}(\mathcal{E}) := \left\{f\in L^2(\mathbb{R}^d) : \iint_{\mathbb{R}^d\times \mathbb{R}^d \setminus \rm{diag}} \frac{(f(x)-f(y))^2}{|x-y|^{d+\alpha}}\,dx\,dy<\infty \right\},
\end{cases}
\end{eqnarray*}
where \({\rm diag}:=\{(x,y)\in \mathbb{R}^d\times \mathbb{R}^d : x=y\}\) and
\[\mathcal{A}(d,\alpha):= \frac{\alpha 2^{\alpha-1} \Gamma((d+\alpha)/2)}{\pi^{d/2} \Gamma(1-\alpha/2)}.\]
Then \((\mathcal{E}, \mathcal{D}(\mathcal{E}))\) is a regular Dirichlet form on \(L^2(\mathbb{R}^d)\) and an associated Hunt process \(X\) is a symmetric \(\alpha\)-stable process. Since \(0<\alpha <d\), \((\mathcal{E}, \mathcal{D}(\mathcal{E}))\) is transient and its \(0\)-order resolvent kernel \(r(x,y)\) can be represented by
\begin{equation}\
r(x,y) = \frac{\Gamma((d-\alpha)/2)}{2^\alpha \pi^{d/2} \Gamma(\alpha/2)}\,\frac{1}{|x-y|^{d-\alpha}}.
\label{eq:ex_resolventkenrel}
\end{equation}
See \cite[Example 1.4.1]{FOT11}, \cite[2.2.2]{CF12} for example.

Similarly to \cite[Example 5.9]{TU23}, we obtain a trace Hardy inequality as follows. Let \(d\rho(x) := |x|^{-\beta}\, dx'\,d\delta_0(x_d)\) for \(x=(x', x_d) \in \mathbb{R}^{d-1}\times \mathbb{R}\) and a Dirac's delta measure \(\delta_0\) at \(0\). Since \(1<\alpha\), the one point \(0\) has zero capacity and so \(\rho\) is a smooth measure. By the Hardy-Littlewood-Sobolev inequality, for \(B(r):=\{x : |x| \leq r\} \) and \(1<p<q<\infty\) satisfying \(\frac{1}{p}+\frac{1}{q} = \frac{d+\alpha -2}{d-1}\) and \(\frac{d-1}{q}<\beta <\frac{d-1}{p}\), we have
\begin{eqnarray*}
\iint_{\mathbb{R}^d\times \mathbb{R}^d \setminus \rm{diag}} \frac{1_{B(r)}(x) 1_{B(r)^c}(y)}{|x-y|^{d-\alpha}}\,d\rho(x)\,d\rho(y) &=& \iint_{\mathbb{R}^{d-1}\times \mathbb{R}^{d-1} \setminus \rm{diag}} \frac{1_{B(r)}(x') 1_{B(r)^c}(y')}{|x'-y'|^{d-\alpha}} \frac{1}{|x'|^\beta}\frac{1}{|y'|^\beta}\,dx'\,dy' \\
&\leq &C \|1_{B(r)} |x'|^{-\beta}\|_{L^p(\mathbb{R}^{d-1})}\,\|1_{B(r)} |y'|^{-\beta}\|_{L^q(\mathbb{R}^{d-1})}\\
&=& C r^{(d-1)/p+(d-1)/q-2\beta}.
\end{eqnarray*}
If \(\beta = \frac{d+\alpha-2}{2}\), then \((d-1)/p+(d-1)/q-2\beta=0\) and so 
\[M:=\sup_{r>0} \iint_{\mathbb{R}^d\times \mathbb{R}^d \setminus \rm{diag}} \frac{1_{B(r)}(x) 1_{B(r)^c}(y)}{|x-y|^{d-\alpha}}\,d\rho(x) < \infty.\]
Let \(D:= \mathbb{R}^d\setminus \{0\}\) and \(X^D\) be a part process of \(X\) on \(D\), that is, \(X^D_t = X_t\) for \(t <\tau_D:=\inf\{t>0 : X_t \not \in D\}\) and \(X_t = \partial\) for \(t\geq \tau_D\). Denote by an associated Dirichlet form \((\mathcal{E}^D, \mathcal{D}(\mathcal{E}^D))\) on \(L^2(D)\), that is, \(\mathcal{E}= \mathcal{E}^D\) and \(\mathcal{D}(\mathcal{E}^D)=\{f\in \mathcal{D}(\mathcal{E}) : f=0 {\rm \ on\ }D^c\}\). Then, for \(\beta = \frac{d+\alpha-2}{2}\), it holds that \(\rho \in \mathcal{K}_{\rm loc}(\mathcal{E}^D)\). By \((\ref{eq:ex_resolventkenrel})\), we obtain  
\[R\rho(x'):=R\rho((x',0)) = \int r(x,y) d\rho(y) = \frac{\Gamma(\frac{d-\alpha}{4})^2}{2^\alpha \sqrt{\pi} \Gamma(\frac{d+\alpha-2}{4})^2} \frac{\Gamma(\frac{\alpha-1}{2})}{\Gamma(\frac{\alpha}{2})} \frac{1}{|x'|^{\frac{d-\alpha}{2}}}\]
and so
\[\frac{d\rho(x)}{R\rho(x')} = \lambda_* \frac{1}{|x'|^{\alpha-1}} dx' d\delta_0(x_d),\ \lambda_*:=\frac{2^\alpha \sqrt{\pi} \Gamma(\frac{d+\alpha-2}{4})^2\Gamma(\frac{\alpha}{2})}{\Gamma(\frac{d-\alpha}{4})^2 \Gamma(\frac{\alpha-1}{2})}.\]
We define \[d\mu_\lambda(x) := -\lambda \frac{1}{|x'|^{\alpha-1}} dx' d\delta_0(x_d),\]
then \(\mu_\lambda \in \mathcal{K}_{\rm loc}(\mathcal{E}^D)\) by \cite[Lemma 5.1]{TU23} and \((\mathcal{E}^{\mu_\lambda}, \mathcal{D}(\mathcal{E}) \cap C_c(\mathbb{R}^d\setminus \{0\}) )\) is non-negative definite by \cite[Lemma 5.2]{TU23}. The capacity of \(\{0\}\) is zero since we now assume \(d \geq 2\), the closure of \( \mathcal{D}(\mathcal{E}) \cap C_c(\mathbb{R}^d\setminus \{0\})\) coincides with that of \( \mathcal{D}(\mathcal{E}) \cap C_c(\mathbb{R}^d)\). Hence we can consider the Schr\"{o}dinger form \(\mathcal{E}^{\mu_\lambda}(f,f) = \mathcal{E}(f,f)- \int |f|^2 d\mu_\lambda(x)\). By \cite[Theorem 5.6]{TU23}, a Scr\"{o}dinger form \((\mathcal{E}^{\mu_{\lambda_*}}, \mathcal{D}(\mathcal{E}^{\mu_{\lambda_*}}))\) is critical. In particular, we obtain a trace Hardy inequality for \((-\Delta )^{\alpha/2}\),
\begin{equation}
\lambda_* \int_{\mathbb{R}^{d-1}} \frac{f(x')^2}{|x'|^{\alpha-1}} \, dx' \leq  \frac{\mathcal{A}(d,\alpha)}{2}\iint_{\mathbb{R}^d\times \mathbb{R}^d \setminus \rm{diag}} \frac{(f(x)-f(y))^2}{|x-y|^{d+\alpha}}\,dx\,dy \label{eq:ex_traceHardy1}
\end{equation}
for \(f\in \mathcal{D}(\mathcal{E})\).

Since \(R\rho\) attains \((\ref{eq:ex_traceHardy1})\), \((\mathcal{E}^{\mu_\lambda}, \mathcal{D}(\mathcal{E}^{\mu_\lambda}))\) is supercritical for \(\lambda > \lambda_*\). For any compact set \(K\) and \(f \in \mathcal{D}(\mathcal{E})\), we have 
\[\int_K |f|\,dm = \mathcal{E}(R1_K, |f|) \leq \sqrt{\mathcal{E}(f,f)}\,\sqrt{\mathcal{E}(R1_K,R1_K)}\]
and
\[\mathcal{E}(R1_K,R1_K) \leq \iint_{K\times K} \frac{1}{|x-y|^{d-\alpha}}\,dxdy < \infty.\]
For \(\lambda <\lambda_*\), it holds that \(\inf{\{\mathcal{E}(f,f) : \int |f|^2\,d\mu_\lambda =1 \}}>1\), and so \((\mathcal{E}^{\mu_\lambda}, \mathcal{D}(\mathcal{E}^{\mu_\lambda}))\) is subcritical by Theorem \ref{Theorem3.5_TU23} (\cite[Theorem 3.5]{TU23}).

By Corollary \ref{subcrinomama}, for \(\lambda < \lambda_*\) and any Bernstein function \(\Phi\) satisfying (IB), \(-\Phi((-\Delta)^{\alpha/2} + \mu_\lambda)\) is subcritical and so, by Theorem \ref{wave_bdd}, the solution to the wave equation \((\ref{eq:wave})\) for \(-\Phi((-\Delta)^{\alpha/2} + \mu_\lambda)\) is uniformly bounded in \(L^2(\mathbb{R}^d)\). In particular, these hold for \(-((-\Delta)^{\alpha/2} + \mu_\lambda)^{\beta/2}\) with \(0<\beta <2\), \(-a \frac{(-\Delta)^{\alpha/2} + \mu_\lambda}{(-\Delta)^{\alpha/2} + \mu_\lambda + c}\), \(-a \log{(1+((-\Delta)^{\alpha/2} + \mu_\lambda)/c)}\), \( -a (\sqrt{2((-\Delta)^{\alpha/2} + \mu_\lambda) + c^2}-c )\) with \(0<a,c\), and \(-((-\Delta)^{\alpha/2} + \mu_\lambda+ m^{2/\beta})^{\beta/2}+m\) with \(0<\beta <2\) and \(m>0\).

It is not easy to obtain the subcriticality or criticality for a subordinated Schr\"{o}dinger operator \(\Phi((-\Delta)^{\alpha/2} - \lambda_* |x'|^{-1})\) for the critical case. In the rest of this example, to obtain the subcriticality and criticality, we consider the heat kernel estimates of \((-\Delta)^{\alpha/2} - \lambda_* |x'|^{-1}\). Although we have obtained \(R\rho(x')\), we need to determine \(R\rho(x)\). By \((\ref{eq:ex_resolventkenrel})\) and  we obtain 
\begin{equation}
R\rho(x) = \int r(x,y) d\rho(y) = \frac{\Gamma(\frac{d-\alpha}{4})^2}{2^\alpha \sqrt{\pi} \Gamma(\frac{d+\alpha-2}{4}) \Gamma(\frac{\alpha}{2})} \frac{1}{|x|^{\frac{d-\alpha}{2}}} \,_2F_1\left(\frac{d-\alpha}{4},\frac{d-\alpha}{4},\frac{d-1}{2};\frac{|x'|^2}{|x|^2} \right) \label{eq:ex_resolventnu}
\end{equation}
where \(_2F_1\) is a hypergeometric function defined by
\[_2F_1(a,b,c ;z):=\frac{\Gamma(c)}{\Gamma(b) \Gamma(c-b)}\int_0 ^1 t^{b-1} (1-t)^{c-b-1} (1-zt)^{-a}\,dt.\]
Indeed, by using a polar coordinate transformation similarly to \cite[Section 2.4 (9)]{EMOT53} and calculations, we have
\begin{eqnarray*}
\lefteqn{ \int_{\mathbb{R}^{d-1}}\frac{1}{(|x'-y'|^2+|x_d|^2)^{(d-\alpha)/2}} \frac{1}{|y'|^{(d+\alpha -2)/2}} \,dy'}\\
&= & |S^{d-3}|  \int_0^\infty r^{\frac{d-\alpha}{2}-1}(r^2+|x|^2)^{-\frac{d-\alpha}{2}} \int_0^\pi \left(1-\frac{2|x'|r}{r^2+|x|^2} \cos{\theta}\right)^{-\frac{d-\alpha}{2}} \sin^{d-3}{\theta} \,d\theta \,dr\\
&=& |S^{d-2}|  \int_0^\infty r^{\frac{d-\alpha}{2}-1}(r^2+|x|^2)^{-\frac{d-\alpha}{2}}\,_2F_1\left(\frac{d-\alpha}{4}, \frac{d-\alpha+2}{4}, \frac{d-1}{2};\left(\frac{2|x'|r}{r^2+|x|^2}\right)^2 \right)\,dr\\
&=& \frac{|S^{d-2}| \Gamma\left(\frac{d-2}{2}\right)}{\Gamma\left(\frac{d-\alpha +2}{4}\right)\Gamma\left(\frac{d+\alpha -4}{4} \right)} \int_0^1 \int_0^\infty r^{\frac{d-\alpha}{2}-1} \left((r^2+|x|^2)^2-4|x'|^2r^2s \right)^{-\frac{d-\alpha}{4}} \,dr\,s^{\frac{d-\alpha-2}{4}}\,(1-s)^{\frac{d+\alpha}{4}-2}\,ds\\
&=& \frac{|S^{d-2}| \Gamma\left(\frac{d-2}{2}\right)}{\Gamma\left(\frac{d-\alpha +2}{4}\right)\Gamma\left(\frac{d+\alpha -4}{4} \right)}  \frac{|x|^{-\frac{d-\alpha}{2}}}{2}\\
&&\times \int_0^1 \int_0^1 u^{\frac{d-\alpha}{4}-1} (1-u)^{\frac{d-\alpha}{4}-1}\left(1-\frac{4|x'|^2}{|x|^2}su(1-u) \right)^{-\frac{d-\alpha}{4}} \,du \,s^{\frac{d-\alpha-2}{4}}\,(1-s)^{\frac{d+\alpha}{4}-2}\,ds\\
&=& \frac{|S^{d-2}|\Gamma\left(\frac{d-2}{2}\right)}{\Gamma\left(\frac{d-\alpha +2}{4}\right)\Gamma\left(\frac{d+\alpha -4}{4} \right)}\frac{|x|^{-\frac{d-\alpha}{2}}}{2} \frac{\Gamma\left(\frac{d-\alpha}{4}\right)^2}{\Gamma\left(\frac{d-\alpha}{2}\right)}\\
&&\hspace{10mm}\times \int_0^1\,_2F_1\left(\frac{d-\alpha}{4},\frac{d-\alpha}{4},\frac{d-\alpha+2}{4};\frac{|x'|^2 s}{|x|^2} \right)\,s^{\frac{d-\alpha-2}{4}}\,(1-s)^{\frac{d+\alpha}{4}-2}\,ds\\
&=& |x|^{-\frac{d-\alpha}{2}}\frac{\pi^{(d-1)/2}}{\Gamma(\frac{d-1}{2})} \frac{\Gamma\left(\frac{d-\alpha}{4}\right)^2}{\Gamma\left(\frac{d-\alpha}{2}\right)} \,_2F_1\left(\frac{d-\alpha}{4},\frac{d-\alpha}{4},\frac{d-1}{2};\frac{|x'|^2}{|x|^2} \right)
\end{eqnarray*}
and so (\ref{eq:ex_resolventnu}) follows from (\ref{eq:ex_resolventkenrel}). Here \(|S^n|\) is the volume of the unit ball on \(\mathbb{R}^{n+1}\).

Since \(R\rho \in \mathcal{D}_e(\mathcal{E})\) satisfies \(\mathcal{E}(R\rho, R\rho)=0\) and \(\mu_{\lambda_*}\) is a Radon measure, for \(u \in \mathcal{D}(\mathcal{E}) \cap C_c(D)\), we have
\[
\mathcal{E}^{\mu_{\lambda_*}}(u, R\rho) = \mathcal{E}(u, R\rho) -\int u\,R\rho\,d\mu_{\lambda_*} = \int u\,d\rho -\int u\,R\rho\,\frac{d\rho}{R\rho}=0
\]
For \(f \in \mathcal{D}(\mathcal{E}) \cap C_c(D)\), since \(R\rho\) is bounded and continuous on the support of \(f\), \(f^2 R\rho \in \mathcal{D}_e(\mathcal{E})\) follows from \cite[Exercise 1.1.10]{CF12}. Moreover it holds that \(f^2 R\rho \in \mathcal{D}(\mathcal{E})\) since \(f^2 R\rho \in L^2(\mathbb{R}^d)\). By an identity 
\[\big(f(x)g(x)-f(y)g(y)\big)^2 = \big(f^2(x)g(x)-f^2(y)g(y)\big)\big(g(x)-g(y)\big) + \big(f(x)-f(y)\big)^2g(x)g(y),\]
we have
\begin{eqnarray*}
\mathcal{E}^{\mu_{\lambda_*}}(fR\rho, fR\rho) &=& \mathcal{E}^{\mu_{\lambda_*}}(f^2R\rho, R\rho) + \frac{\mathcal{A}(d,\alpha)}{2}\iint_{\mathbb{R}^d\times \mathbb{R}^d \setminus \rm{diag}} \frac{(f(x)-f(y))^2}{|x-y|^{d+\alpha}}\,R\rho(x)dx\,R\rho(y)dy\\
&=& \frac{\mathcal{A}(d,\alpha)}{2}\iint_{\mathbb{R}^d\times \mathbb{R}^d \setminus \rm{diag}} \frac{(f(x)-f(y))^2}{|x-y|^{d+\alpha}}\,R\rho(x)dx\,R\rho(y)dy\\
&=& \frac{\mathcal{A}(d,\alpha)}{2}\iint_{\mathbb{R}^d\times \mathbb{R}^d \setminus \rm{diag}} (f(x)-f(y))^2 \frac{R\rho(y)dy}{|x-y|^{d+\alpha}\,R\rho(x)}\,R\rho(x)^2dx.
\end{eqnarray*}
Setting \(h:=R\rho\), since an \(\mathcal{E}_1\)-closure of \(\mathcal{D}(\mathcal{E}) \cap C_c(D)\) is \(\mathcal{D}(\mathcal{E})\), an \(h\)-transformed Dirichlet form \((\mathcal{E}^{\mu_{\lambda_*}, h}, \mathcal{D}(\mathcal{E}^{\mu_{\lambda_*}, h}))\) coincides with a Dirichlet form of a pure jump process on a weighted manifold \((\mathbb{R}^d, h^2dx)\) with a jump kernel
\[J^{\mu_{\lambda_*}, h}(dx,dy):=\frac{h^2(x)dx\,h^2(y)dy}{|x-y|^{d+\alpha}\,h(x)h(y)}.\]
Since \(\alpha >1\), we have \(\frac{d-1}{2}>2\,\frac{d-\alpha}{4}\) and so we have \(_2F_1(\frac{d-\alpha}{4}, \frac{d-\alpha}{4}, \frac{d-1}{2}; 1)< \infty\) by the Gaussian hypergeometric theorem \cite[15.4.20]{OLBC10}, and so \(h(x) \asymp |x|^{-\frac{d-\alpha}{2}}\). In this case, \((\mathcal{E}^{\mu_{\lambda_*}, h}, \mathcal{D}(\mathcal{E}^{\mu_{\lambda_*}, h}))\) coincides with the Dirichlet form of a transformed Schr\"{o}dinger form associated with the critical Hardy inequality appearing in \cite[(5.1)]{BGJP19}, the heat kernel \(p^{\mu_{\lambda_*}}\) for Schr\"{o}dinger form \((\mathcal{E}^{\mu_{\lambda_*}}, \mathcal{D}(\mathcal{E}^{\mu_{\lambda_*}}))\) associated with  the critical trace Hardy inequality enjoys (\ref{eq:ex_Hardy_HKE_2}), that is, it holds that
\begin{eqnarray}
p^{\mu_{\lambda_*}}(t,x,y) \asymp \left(1+\frac{t^{\delta/\alpha}}{|x|^{\delta}}\right)\left(1+\frac{t^{\delta/\alpha}}{|y|^{\delta}}\right) \left(t^{-d/\alpha} \wedge \frac{t}{|x-y|^{d+\alpha}}  \right)
\end{eqnarray}
By the same way as Example \ref{ex.Hardy}, \(((-\Delta)^{\alpha/2} - \lambda_* |x'|^{-1})^{\beta/2}\) is subcritical for any \(1< \alpha < 2, 0<\beta <2\). Combining these with Example \ref{ex_traceHardy1}, \(((-\Delta)^{\alpha/2} - \lambda |x'|^{-1})^{\beta/2}\) is subcritical for \(0< \lambda < \lambda_*, 1<\alpha\leq 2, 0<\beta \leq 2\) and \(\lambda =\lambda_*, 1<\alpha\leq 2, 0<\beta <2\). Hence, by Theorem \ref{wave_bdd}, the solution \(W^{\alpha, \beta}(x,t)\) to the following wave equation with a singular potential \((\ref{eq:ex_traceHardy_wave})\) is uniformly bounded in \(L^2(\mathbb{R}^d)\) for \(\lambda <\lambda_*, 1<\alpha \leq 2, 0< \beta \leq 2\) and \(\lambda = \lambda_*, 1<\alpha\leq 2, < \beta < 2\).
\begin{eqnarray}
\begin{cases}
\frac{\partial^2}{\partial t^2} w(x,t) = -((-\Delta)^{\alpha/2} - \lambda |x'|^{-1})^{\beta/2} w(x,t)\hspace{5mm}\text{for\ }(x,t)\in  \mathbb{R}^d\times (0,\infty)\\
\frac{\partial}{\partial t} w(x,0) = g(x)\in C_c(\mathbb{R}^d)\hspace{5mm}\text{for\ }x\in  \mathbb{R}^d\\
w(x,0) = 0\hspace{5mm}\text{for\ }x\in \mathbb{R}^d
\end{cases}\label{eq:ex_traceHardy_wave}
\end{eqnarray}
Moreover, the criticality and subcriticality for \(\Phi((-\Delta)^{\alpha/2} - \lambda_* |x'|^{-1})\) coincides with those appearing in Example \ref{ex.Hardy}. We can also obtain the result on uniformly boundedness of the solutions to wave equations \((\ref{eq:wave})\) with an initial function \(g\in C_c(\mathbb{R}^d)\) for \(\mathcal{L}^{\mu, \Phi} =\Phi((-\Delta)^{\alpha/2} - \lambda_* |x'|^{-1})\) for subordinators \(\Phi\) making \(\Phi((-\Delta)^{\alpha/2} - \lambda_* |x'|^{-1})\) subcritical. 
\end{example}

\begin{example}
We consider Laplace operators on spaces with varying dimension \(\mathbb{R}^d \cup \mathbb{R}^{d'}\) in \cite{O22}. Let \(d \geq 3\) and \(d'\geq 2.\) For small numbers \(\varepsilon, \varepsilon' >0\), we set \(\mathbb{R}_\varepsilon^d := \mathbb{R}^d \setminus B^d(0;\varepsilon)\) and \(\mathbb{R}_{\varepsilon'}^{d'} := \mathbb{R}^{d'} \setminus B^{d'}(0;\varepsilon')\), where \(B^d(0;\varepsilon)\) is a closed ball in \(\mathbb{R}^d\) with centre \(0\) and radius \(\varepsilon\). We identify \(\partial \mathbb{R}_\varepsilon^d \) and \(\partial \mathbb{R}_{\varepsilon'}^{d'}\) with a one point \(\{a^*\}\) and we define \(E:=\mathbb{R}_\varepsilon^d \cup \mathbb{R}_{\varepsilon'}^{d'} \cup \{a^*\}\). We use \(E\) instead of \(\mathbb{R}^d \cup \mathbb{R}^{d'}\) because Brownian motion on \(\mathbb{R}^d\) for \(d \geq 2\) does not hit to a point \(0\) and we cannot attach \(\mathbb{R}^d\) and \(\mathbb{R}^{d'}\) at a single point when considering the Laplace operator. We consider the natural metric \(\rho\) induced from Euclidean metrics on \(\mathbb{R}^d\) and \(\mathbb{R}^{d'}\), and the Lebesgue measure. In \cite{O22}, a regular Dirichlet form \((\mathcal{E}, \mathcal{D}(\mathcal{E}))\) on \(L^2(E;m)\) is defined as follows.
\begin{eqnarray*}
\begin{cases}
\mathcal{E}(f,g):=\frac{1}{2} \int_{\mathbb{R}_\varepsilon^d} \nabla_d(f|_{\mathbb{R}_\varepsilon^d})\cdot \nabla_d (g|_{\mathbb{R}_\varepsilon^d}) \,dx + \frac{1}{2} \int_{\mathbb{R}_{\varepsilon'}^{d'}} \nabla_{d'}(f|_{\mathbb{R}_{\varepsilon'}^{d'}})\cdot \nabla_{d'} (g|_{\mathbb{R}_{\varepsilon'}^{d'}}) \,dx,\\
\mathcal{D}(\mathcal{E}) := \{f \in L^2(E) : f|_{\mathbb{R}_\varepsilon^d} \in H^1(\mathbb{R}_\varepsilon^d), \ f|_{\mathbb{R}_{\varepsilon'}^{d'}} \in H^1(\mathbb{R}_{\varepsilon'}^{d'}),\ f(x)=f(a^*) \text{\  on\ } \partial \mathbb{R}_\varepsilon^d \cup \partial \mathbb{R}_{\varepsilon'}^{d'}\}.
\end{cases}
\end{eqnarray*}
Then the associated Hunt process \(X\) is called Brownian motion on the space with varying dimension. By \cite[Theorem 1.6]{O22}, for \(d\geq 3, d'=2\), its heat kernel \(p(t,x,y)\) has an on-diagonal estimate \(p(t,x,x) \leq C t^{-1}(\log{t})^{-2} + C t^{-d/2}\) for \(t \geq 1\), so \(\int_1^\infty p(t,x,x)dt< \infty\) and \((\mathcal{E}, \mathcal{D}(\mathcal{E}))\) is transient.  By \cite[Theorem 1.7]{O22}, for \(d\geq  d' \geq 3\), its heat kernel \(p(t,x,y)\) has an on-diagonal estimate \(p(t,x,x) \leq Ct^{-d/2} + C t^{-{d'}/2}\) for \(t \geq 1\), so \(\int_1^\infty p(t,x,x)dt< \infty\) and \((\mathcal{E}, \mathcal{D}(\mathcal{E}))\) is transient. See \cite{O22} for details of sharp heat kernel estimates on \(E\). 

For \(d\geq d' \geq 3\), by combining the Hardy inequalities on \(\mathbb{R}^d\) and \(\mathbb{R}^{d'}\), we have, for \(f\in \mathcal{D}(\mathcal{E})\),
\begin{equation}
\int_E \lambda_*(x) \frac{|f(x)|^2}{|x|^2} dx \leq \mathcal{E}(f,f) \label{eq:ex_BMVD_1}
\end{equation}
where \[\lambda_*(x) := \frac{(d-2)^2}{8} 1_{\mathbb{R}^d_{\varepsilon}}(x) + \frac{(d'-2)^2}{8} 1_{\mathbb{R}^{d'}_{\varepsilon'}}(x).\]
For \(d> d' = 2\), by \cite[Proposition 6.4]{O22} (see also \cite[Lemma 6.1]{GS09}), there exists a strictly positive function \(h\) on \(E\) such that \(h \asymp 1-|x|^{2-d} \asymp 1\) on \(\mathbb{R}^d_{\varepsilon}\) and \(h \asymp 1+\log{(|x|-\varepsilon')}\) on \(\mathbb{R}^{d'}_{\varepsilon'}\). We have \(\Delta \sqrt{h} = - \sqrt{h} \nabla(\log h)/4\) and 
\begin{equation}
\frac{1}{8} \int_E \frac{ |\nabla h(x)|^2}{|h(x)|^2} |f(x)|^2 dx \leq \mathcal{E}(f,f). \label{eq:ex_BMVD_2}
\end{equation}
for \(f \in \mathcal{D}(\mathcal{E})\) by \cite[Theorem (1.9)]{F00}. From \(h \asymp 1+\log{(|x|-\varepsilon')}\) on \(\mathbb{R}^{d'}_{\varepsilon'}\), it follows that \(\frac{ |\nabla h(x)|^2}{|h(x)|^2} \asymp (|x|\log{|x|})^{-2}\) on \(\mathbb{R}^{d'}_{\varepsilon'}\), and, from \(h \asymp 1-|x|^{2-d}\) on \(\mathbb{R}^d_{\varepsilon}\), it follows that \(\frac{ |\nabla h(x)|^2}{|h(x)|^2} \asymp |x|^{2-2d}\) on \(\mathbb{R}^{d}_{\varepsilon}\). Hence we obtain
\begin{eqnarray}
C\int_{\mathbb{R}^d_\varepsilon} \frac{|f(x)|^2}{|x|^{2d-2}} \,dx + C\int_{\mathbb{R}^{d'}_{\varepsilon'}} \frac{|f(x)|^2}{|x|^{2}(\log{|x|})^2}\, dx \leq \mathcal{E}(f,f) \label{eq:ex_BMVD_3}
\end{eqnarray}
for any \(f\in \mathcal{D}(\mathcal{E})\) and some \(C>0\). We set \(\mu_{c}:=-c|\nabla h|^2/(8|h|^2)\,dx\), then \(-\Delta -\mu_1\) is critical and \(-\Delta -\mu_c\) is subcritical for \(0<c<1\). For any Bernstein function \(\Phi\) satisfying (IB), \(-\Phi(-\Delta -\mu_c)\) is subcritical.

At the last of this example,  we give the following remark. Combining the Hardy inequalities on \(\mathbb{R}^d\) and the exterior domain \(\mathbb{R}^2_{\varepsilon'}\) (\cite{ACR02}), we have
\begin{eqnarray}
C\int_{\mathbb{R}^d_\varepsilon} \frac{|f(x)|^2}{|x|^{2}} \,dx + C\int_{\mathbb{R}^{d'}_{\varepsilon'}} \frac{|f(x)-f(a^*)|^2}{|x|^{2}(\log{|x|})^2}\, dx \leq \mathcal{E}(f,f) \label{eq:ex_BMVD_4}
\end{eqnarray}
for any \(f\in \mathcal{D}(\mathcal{E})\). When viewed across the entire space \(\mathbb{R}^d_\varepsilon \cup \mathbb{R}^{d'}_{\varepsilon'}\cup\{a^*\}\), the condition \(f|_{\partial \mathbb{R}^d_\varepsilon} = f(a^*)\) affects the order of \(|x|^{-2}\), and it becomes \(|x|^{2-2d}\).
\end{example}

\begin{example}
In \cite{CGL21}, Hardy's inequalities for local and non-local regular Dirichlet forms on metric measure spaces are obtained by using the Green operator. We consider the subcriticality of Schr\"{o}dinger forms on some fractal spaces by using Hardy's inequalities in \cite[Example 5.12]{CGL21}.

Let \(E\) be the Sierpinski gasket, \(d\) be a metric and \(m\) be a positive Radon measure with full support. There exist a strongly local regular Dirichlet form \((\mathcal{E}, \mathcal{D}(\mathcal{E}))\) on \(L^2(E;m)\) corresponding to the Laplace operator \(\Delta\) and its heat kernel \(p(t,x,y)\) satisfying
\begin{equation}
\frac{C_1}{t^{\alpha/\beta}} \exp{\left(-c_1 \left(\frac{d(x,y)^\beta}{t} \right)^{\frac{1}{\beta -1}} \right)} \leq p(t,x,y) \leq \frac{C_2}{t^{\alpha/\beta}} \exp{\left(-c_2 \left(\frac{d(x,y)^\beta}{t} \right)^{\frac{1}{\beta -1}} \right)} 
\label{eq:fractal}
\end{equation}
with \(\beta> \alpha\). See \cite{BP88} for details.
We take \(0<\delta < \alpha/\beta\) and consider the subordinated Dirichlet form \((\mathcal{E}^\delta, \mathcal{D}(\mathcal{E}^\delta))\) associated with \(-(-\Delta)^\delta\). Then, by \cite[Example 5.12]{CGL21}, the following critical Hardy inequality holds.
\begin{equation}
\lambda_* \int_E \frac{|f(x)|^2}{d(x_o,x)^{\beta \delta}} \,dm(x) \leq  \mathcal{E}^\delta(f,f) \label{eq:fractal2}
\end{equation}
for any \(f \in \mathcal{D(\mathcal{E}^\delta)}\), and a fixed point \(x_o \in E\) and some constant \(\lambda_*>0\). The same type of heat kernel estimate \((\ref{eq:fractal})\) and the Hardy inequality \((\ref{eq:fractal2})\) hold for p.c.f. fractals \cite{K01} and for generalized Sierpinski carpets \cite{BB99}.

For such fractals and fixed \(0< \delta <\alpha/\beta\), let \(d\mu_\lambda(x) := -\lambda \, d(x_o,x)^{-\beta \delta} \,dm(x)\). Then \(-(-\Delta)^\delta -\mu_\lambda\) is subcritical for \(\lambda > \lambda_*\). Hence, by Corollary \ref{subcrinomama}, for \(\lambda < \lambda_*\) and any Bernstein function \(\Phi\) satisfying (IB), \(-\Phi((-\Delta)^{\delta} + \mu_\lambda)\) is subcritical and so, by Theorem \ref{wave_bdd}, the solution to the wave equation \((\ref{eq:wave})\) for \(-\Phi((-\Delta)^{\delta} + \mu_\lambda)\) is uniformly bounded in \(L^2(E)\). In particular, these hold for \(-((-\Delta)^{\delta} + \mu_\lambda)^{\gamma/2}\) with \(0<\gamma <2\), \(-a \frac{(-\Delta)^{\delta} + \mu_\lambda}{(-\Delta)^{\delta} + \mu_\lambda + c}\), \(-a \log{(1+((-\Delta)^{\delta} + \mu_\lambda)/c)}\), \( -a (\sqrt{2((-\Delta)^{\delta} + \mu_\lambda) + c^2}-c )\) with \(0<a,c\), and \(-((-\Delta)^{\delta} + \mu_\lambda+ m^{2/\gamma})^{\gamma/2}+m\) with \(0<\beta <2\) and \(m>0\).

\end{example}

\appendix
\section{Basic definitions on Markov processes in Dirichlet form theory}\label{Appendix}
We summarize definitions and basic properties of Dirichlet form theory. For more details, see \cite{CF12,FOT11,O13}.

Let \(E\) be a locally compact separable metric space and \(m\) be a positive Radon measure with \(\text{ supp}(m)=E\). The state space \(E\) is equipped with the Borel \(\sigma\)-algebra \(\mathcal{B}(E)\). We take an isolated point \(\partial \not \in E\) called the cemetery point, and set \(E_\partial:=E\cup \{\partial\}\) equipped with \(\mathcal{B}(E_\partial):=\mathcal{B}(E) \cup \{B \cup \{\partial\} : B \in \mathcal{B}(E)\}\). The inner product in \(L^2(E;m)\) is denoted by \(\langle \cdot, \cdot \rangle_{m}\) and the \(L^2\)-norm is denoted by \(||\cdot ||_{L^2(m)}\). 

\begin{definition}[closed form and Dirichlet form]
Let \(\mathcal{D}(\mathcal{E})\) be a dense linear subspace of \(L^2(E;m)\) and \(\mathcal{E}\) be a non-negative definite symmetric bilinear form on \(\mathcal{D}(\mathcal{E})\times \mathcal{D}(\mathcal{E})\). We call \(\mathcal{D}(\mathcal{E})\) the domain of \(\mathcal{E}\). A non-negative definite symmetric bilinear form \((\mathcal{E}, \mathcal{D}(\mathcal{E}))\) is called a \textit{closed form} on \(L^2(E;m)\) if \(\mathcal{E}\) is complete with respect to the norm induced by \(\mathcal{E}_1\), where \(\mathcal{E}_{\alpha}(f,g):=\mathcal{E}(f,g)+\alpha \langle f,g\rangle_m\) for \(\alpha>0\).

Moreover, a closed form \((\mathcal{E}, \mathcal{D}(\mathcal{E}))\) is called a \textit{Dirichlet form} on \(L^2(E;m)\) if it is Markovian, that is, for any \(f\in \mathcal{D}(\mathcal{E})\), it holds that \(g:=(0\vee f)\wedge 1 \in \mathcal{D}(\mathcal{E})\) and \(\mathcal{E}(g,g)\leq \mathcal{E}(f,f)\).
\end{definition}
We note that \((\mathcal{D}(\mathcal{E}), \mathcal{E}_\alpha)\) is a Hilbert space for any closed form \((\mathcal{E}, \mathcal{D}(\mathcal{E}))\) and \(\alpha>0\).

It is well known that, for a strongly continuous contraction semigroup \(\{T_t\}_{t> 0}\) on \(L^2(E;m)\), the pair
\begin{eqnarray*}
\begin{cases}
\mathcal{D}(\mathcal{E}):=\{f \in L^2(E;m)\mid \lim_{t\searrow 0} \frac{1}{t}\langle f-T_tf, f \rangle_m <\infty\}\\
\mathcal{E}(f,g):=\lim_{t\searrow 0} \frac{1}{t}\langle f-T_tf, g \rangle_m\ \ \text{for\ }f,g\in \mathcal{D}(\mathcal{E})
\end{cases}
\end{eqnarray*}
is a closed form. Conversely, for a closed form \((\mathcal{E}, \mathcal{D}(\mathcal{E}))\) on \(L^2(E;m)\), there exists a strongly continuous contraction semigroup \(\{T_t\}_{t> 0}\) on \(L^2(E;m)\) such that \(G_{\alpha}f:= \int_0^{\infty} e^{-\alpha t} T_tf(x) dt \in \mathcal{D}(\mathcal{E})\) and \(\mathcal{E}_{\alpha}(G_{\alpha}f,g)=\langle f, g\rangle_m\) for \(\alpha >0\), \(f\in L^2(E;m)\) and \(g\in \mathcal{D}(\mathcal{E})\).

Furthermore, it is known that the generator of a strongly continuous contraction semigroup is a non-positive definite self-adjoint operator, and, for a non-negative definite self-adjoint operator \(-\mathcal{L}\) on \(L^2(E;m)\), \(\{T_t:=e^{\mathcal{L}t}\}_t\) is a strongly continuous contraction semigroup on \(L^2(E;m)\). In this case, the corresponding closed form is represented by \(\mathcal{D}(\mathcal{E})=\mathcal{D}(\sqrt{-\mathcal{L}})\) and \(\mathcal{E}(f,g)=\langle \sqrt{-\mathcal{L}}f, \sqrt{-\mathcal{L}}g \rangle_m = \langle -\mathcal{L}f, g \rangle_m\).

\begin{definition}[Markov process]
A quadruplet \(X=(\Omega, \mathcal{M}, \{X_t\}_{t\geq 0}, \{\mathbb{P}_x\}_{x\in E_\partial})\) is a Markov process on \(E\) if the following conditions hold.
\begin{enumerate}
\item[(M1)] For each \(x\in E_\partial\), \((\Omega, \mathcal{M}, \{X_t\}_{t\geq 0}, \mathbb{P}_x)\) is a stochastic process on \(E_\partial\), that is, \(\Omega, \mathcal{M}, \mathbb{P}_x\) is a probability space and \(X_t : \Omega \to E_\partial\) is a measurable map for each \(t\geq 0\).
\item[(M2)] For each \(t\geq 0\) and \(B\in \mathcal{B}(E_\partial)\), a map \(E_\partial \ni x \mapsto \mathbb{P}_x(X_t\in B) \in \mathbb{R}\) is measurable.
\item[(M3)] There exists a family of increasing sub \(\sigma\)-fields \(\{\mathcal{M}_t\}_{t\geq 0}\) of \(\mathcal{M}\) such that  \(X_t\) is \(\mathcal{M}_t\) measurable for each \(t\geq 0\), and \(\mathbb{P}_x(X_{s+t} \in B | \mathcal{M}_t) = \mathbb{P}_{X_t}(X_s\in B)\), \(\mathbb{P}_x\)-almost surely for any \(x\in E_\partial, s,t \geq 0\) and \(B\in \mathcal{B}(E_\partial)\). Here \(\mathbb{P}_x(\cdot|\cdot)\) is a conditional probability.
\item[(M4)] It holds that \(\mathbb{P}_x(X_0=x)=1\) for any \(x\in E_\partial\), and \(\mathbb{P}_\partial(X_t=\partial)=1\) for any \(t\geq 0\).
\end{enumerate}
\end{definition}
The condition (M3) is called the Markov property, and \(\{\mathcal{M}_t\}_{t\geq 0}\) in (M3) is called an admissible filtration. We call \(\sigma :\Omega \to [0,\infty]\) a stopping time if \(\{\sigma \leq t\} \in \mathcal{M}_t\) for each \(t\geq 0\). For a Markov process \(X\), we set \(\zeta(\omega):= \inf\{t \geq 0 : X_t(\omega)=\partial\}\). This random variable \(\zeta\) is a stopping time and we call \(\zeta\) a lifetime of \(X\).

\begin{definition}[Hunt process]
A Markov process \(X=(\Omega, \mathcal{M}, \{X_t\}_{t\geq 0}, \{\mathbb{P}_x\}_{x\in E_\partial})\) on \(E\) is called a Hunt process if the following conditions hold.
\begin{enumerate}
\item[(H1)] For \(t\geq \zeta(\omega)\), \(X_t(\omega) =\partial\) holds. For each \(t\geq 0\), there exists a map \(\theta_t :\Omega \to \Omega\) such that \(X_s\circ \theta_t = X_{s+t}\) for any \(s\geq 0\). For each \(\omega \in \Omega\), \(X_\cdot(\omega)\) is right continuous on \([0,\infty)\) and has left limits on \((0,\infty)\) in \(E_\partial\).
\item[(H2)]  An admissible filtration \(\{\mathcal{M}_t\}_{t\geq 0}\) in (M3) satisfies \(\bigcap_{s>t}\mathcal{M}_s = \mathcal{M}_t\) for each \(t\geq 0\) and, \(\mathbb{P}_\mu(X_{s+\sigma} \in B | \mathcal{M}_\sigma) = \mathbb{P}_{X_\sigma}(X_s\in B)\) holds \(\mathbb{P}_\mu\)-almost surely for each stopping time \(\sigma\), any \(s\geq 0\), \(B\in \mathcal{B}(E_\partial)\) and any probability measure \(\mu\) on \(E_\partial\).
\item[(H3)] For any increasing stopping times \(\{\sigma_n\}_n\) with \(\sigma := \lim_{n\to \infty} \sigma_n\), it holds that 
\(\mathbb{P}_\mu(\lim_{n\to \infty} X_{\sigma_m} = X_\sigma, \sigma<\infty) = \mathbb{P}_\mu(\sigma < \infty)\) for any probability measure \(\mu\) on \(E_\partial\).
\end{enumerate}
\end{definition}
The condition (H2) is called the strong Markov property, and a Markov process satisfying (H2) is called a strong Markov process. The condition (H3) is called quasi-left-continuity on \((0,\infty)\).

\begin{definition}[regular Dirichlet form]
A Dirichlet form \((\mathcal{E},\mathcal{D}(\mathcal{E}))\) is called \textit{regular} if \(\mathcal{D}(\mathcal{E}) \cap C_c(E)\) is \(\mathcal{E}_1\)-dense in \(\mathcal{D}(\mathcal{E})\) and \(||\cdot||_{\infty}\)-dense in \(C_c(E)\), where \(\|\cdot\|_\infty\) is the essential supremum with respect to \(m\).
\end{definition}

\begin{definition}[nest, polar set, quasi-continuous, \(m\)-inessential set]\label{quasi_notation}
\ \vspace{-3mm}\\
\begin{enumerate}
\item An increasing sequence of closed sets \(\{F_k\}_{k\geq 1}\) of \(E\) is called a \textit{nest} if \(\cup_{k\geq 1} \{f\in \mathcal{D}(\mathcal{E}) : f=0 \ m\text{-a.e.\  on\  }E \setminus F_k\}\) is \(\mathcal{E}_1\)-dense in \(\mathcal{D}(\mathcal{E})\).
\item  \(N\subset E\) is \textit{\(\mathcal{E}\)-polar} if there exists a nest \(\{F_k\}_{k\geq 1}\) such that \(N\subset \cap_{k\geq 1} (E\setminus F_k)\).
\item A statement depending on \(x\in A \subset E\) holds \(\mathcal{E}\)-\textit{quasi-everywhere} (q.e. in abbreviation) on \(A\) if there exists an \(\mathcal{E}\)-polar set \(N\subset A\) such that the statement holds for \(x\in A\setminus N\).
\item A function \(f\) is \textit{quasi-continuous} if there exists a nest \(\{F_k\}_{k\geq 1}\) such that the restriction of \(f\) to \(F_k\) is finite and continuous on \(F_k\) for each \(k\geq 1\).
\item  A subset \(B\subset E\) is a \textit{nearly Borel set} if, for any probability measure \(\mu\) on \(E\cup \{\partial\}\), there exist Borel sets \(B_1, B_2\) such that \(B_1\subset B \subset B_2\) and \(\mathbb{P}_{\mu}(X_t \in B_2\setminus B_1 \text{\ for\ some\ }t\geq 0)=0\).
\item A subset \(N\subset E\) is \(m\)\textit{-inessential} if \(N\) is an \(m\)-negligible nearly Borel set such that \(\mathbb{P}_x(\sigma_{N}<\infty)=0\) for \(x\in E\setminus N,\) where \(\sigma_{N}:=\inf\{t>0;X_t\in N\}\) is the first hitting time to \(N\).
\end{enumerate}
\end{definition}

We remark that an \(\mathcal{E}\)-polar set is \(m\)-negligible.

For a regular Dirichlet form \((\mathcal{E},\mathcal{D}(\mathcal{E}))\) on \(L^2(E;m)\), by Fukushima's theorem, there exists an \(m\)-symmetric Hunt process \(X=(\Omega, \mathcal{M}, \{X_t\}_{t\geq 0}, \{\mathbb{P}_x\}_{x\in E_\partial})\) on \(E\) associated with \((\mathcal{E},\mathcal{D}(\mathcal{E}))\). Here, \(X\) is \(m\)-symmetric if it holds that
\[\int_E P_tf(x) g(x)\,dm(x) = \int_E f(x) P_tg(x)\,dm(x)\]
for any non-negative Borel measurable functions \(f,g\), where 
\[P_tf(x):=\mathbb{E}_x[f(X_t)]:=\int_\Omega f(X_t(\omega))\, d\mathbb{P}_x(\omega).\]
We define the resolvents \(\{R_{\alpha}\}_{\alpha>0}\) by
\[R_{\alpha}f(x):=\int_0^{\infty}e^{-\alpha t}P_tf(x)dt\]
for \(f\in L^2(m)\) and \(\alpha>0.\) We remark that \(P_tf\) (resp. \(R_{\alpha}f\)) is a quasi-continuous version of \(T_tf\) (resp. \(G_{\alpha}f\)) for \(f\in \mathcal{D}(\mathcal{E})=\mathcal{D}(\sqrt{-\mathcal{L}})\), that is, for each \(t\geq 0\) and \(\alpha>0\), \(P_tf = T_tf \) and \(R_\alpha f = G_\alpha f\) \(m\)-almost everywhere.

\begin{example}
Let \(E:=\mathbb{R}^d\), \(m\) be a Lebesgue measure, and \(\mathcal{L}:=\Delta\). Then, the corresponding closed form \((\mathcal{E}, \mathcal{D}(\mathcal{E}))\) on \(L^2(\mathbb{R}^d; dx)\) is \(\mathcal{D}(\mathcal{E})=H^1(\mathbb{R}^d)\), the \(1\)-order Sobolev space, and \(\mathcal{E}(f,g)=\int \nabla f \, \nabla g\, dx\). In this case, \((\mathcal{E}, \mathcal{D}(\mathcal{E}))\) is a regular Dirichlet form on \(L^2(\mathbb{R}^d; dx)\) and its corresponding process is Brownian motion on \(\mathbb{R}^d\).
\end{example}

We define an \textit{extended Dirichlet space} \(\mathcal{D}_e(\mathcal{E})\) by the space of \(m\)-equivalence classes of all \(m\)-measurable functions \(f\) on \(E\) such that \(|f|<\infty, m\)-almost everywhere and there exists an \(\mathcal{E}\)-Cauchy sequence
\(\{f_n\}_{n\geq 1}\subset \mathcal{D}(\mathcal{E})\) such that \(\lim_{n\to \infty}f_n=f\) \(m\)-almost everywhere on \(E\). We can define \(\mathcal{E}(f,f)\) for \(f\in \mathcal{D}_e(\mathcal{E})\) by \(\mathcal{E}(f,f):= \lim_{n\to \infty}\mathcal{E}(f_n,f_n)\) for the above sequence \(\{f_n\}_n \subset \mathcal{D}(\mathcal{E})\). This definition is independent of the choice of an approximation sequence \(\{f_n\}_n\) of \(f\in \mathcal{D}_e(\mathcal{E})\). We remark that any function belonging to \(\mathcal{D}(\mathcal{E})\) has a quasi-continuous version, so without loss of generality, we may treat all functions in \(\mathcal{D}_e(\mathcal{E})\) as quasi-continuous functions.

\begin{definition}[Transience and Recurrence]
A regular Dirichlet form \((\mathcal{E}, \mathcal{D}(\mathcal{E}))\) on \(L^2(E;m)\) is {\it transient} if \((\mathcal{D}_e(\mathcal{E}), \mathcal{E})\) is a real Hilbert space. A regular Dirichlet form \((\mathcal{E}, \mathcal{D}(\mathcal{E}))\) on \(L^2(E;m)\) is {\it recurrent} if \(1\in \mathcal{D}_e(\mathcal{E})\) and \(\mathcal{E}(1,1)=0.\)
\end{definition}

We emphasize that we define the transience and recurrence for a form \((\mathcal{E}, \mathcal{D}(\mathcal{E}))\) corresponding to a stochastic process. There are many other definitions of transience and recurrence for Dirichlet forms, semigroup \(\{T_t\}_t\), \(\{P_t\}_t\), but most of them are equivalent for a regular Dirichlet form. For example, \((\mathcal{E}, \mathcal{D}(\mathcal{E}))\) is transient if and only if the existence of the \(0\)-order resolvent \(Gf:=(-\mathcal{L})^{-1}f\).

\begin{definition}[Irreducibility]
For a strongly continuous semigroup \(\{T_t\}_t\), a set \(A\) is \(\{T_t\}\)-invariant if \(T_t(1_A f)=1_A T_tf\) for any \(f\in L^2(E;m)\) and \(t>0\). A semigroup \(\{T_t\}_t\) is irreducible if, for any \(\{T_t\}\)-invariant set \(A\), either \(m(A)=0\) or \(m(A^c)=0\) holds. 
\end{definition}

We also say a closed form \((\mathcal{E}, \mathcal{D}(\mathcal{E}))\) is irreducible if its corresponding semigroup is irreducible.

\begin{proposition}[{cf. \cite[Proposition 2.1.3]{CF12}}]
An irreducible Dirichlet form is either transient or recurrent.
\end{proposition}

\subsection{Smooth measures, PCAFs and the Revuz correspondence}\label{Appendix_PCAF}
We provide definitions of a smooth measure and a PCAF, and their relationship. See \cite[Section 2.3, Section 4]{CF12} for details.
\begin{definition}[Smooth measure]
A positive Borel measure \(\mu\) on \(E\) is a \textit{smooth measure} if \(\mu\) charges no \(\mathcal{E}\)-polar set and, there exists a nest \(\{F_k\}_k\) such that \(\mu(F_k)<\infty\) for every \(k\geq 1.\) Denote by \(\mathcal{S}\) the family of all smooth measures.
\end{definition}

We remark that a smooth measure is not a Radon measure in general. The following is a subclass of Radon measures in \(\mathcal{S}\).
\begin{definition}[Smooth measure of finite energy integral]
A positive Radon measure \(\mu\) on \(E\) is called a measure of \textit{finite energy integral} if there exists a constant \(C>0\) such that, for any \(f \in \mathcal{D}(\mathcal{E})\), it holds that
\[\int_E |f(x)|\,d\mu(x) \leq  C \sqrt{\mathcal{E}_1(f,f)}.\]
Denote \(\mathcal{S}_0\) by the family of all Radon measures of finite energy integrals.
\end{definition}

By Theorem \cite[Theorem 2.3.7]{CF12}, \(\mathcal{S}_0 \subset \mathcal{S}\) holds. For any \(\mu \in \mathcal{S}_0\), by the Riesz representation theorem, for \(\alpha>0\), there exists a function \(U_\alpha \mu\) called an \(\alpha\)-\textit{potential} of \(\mu\) such that \(\int_E f\,d\mu = \mathcal{E}_\alpha(U_\alpha\mu,f)\) holds for each \(f \in \mathcal{D}(\mathcal{E})\). Moreover, denote by \(\mathcal{S}_{00}\) the family of all \(\mu \in \mathcal{S}_0\) such that \(U_1\mu\) is bounded and \(\mu(E)<\infty\). Then, for any \(\mu \in \mathcal{S}\), there exists a nest \(\{F_k\}_k\) such that \(1_{F_k}\mu \in \mathcal{S}_{00}\) for each \(k\), and \(\mu(\bigcap_{k=1}^\infty F_k^c)=0.\) See \cite[Section 2.3]{CF12} for details.

Let \(\mathcal{M}_{\infty}\) be the smallest \(\sigma\)-field including \(\{\mathcal{M}_t\}_t\), which is an admissible filtration of \(X\).
\begin{definition}[PCAF, positive continuous additive functional]\label{DefPCAF}
A \([-\infty,\infty]\)-valued stochastic process \(A=\{A_t\}_{t\geq 0}\) is called an \textit{additive functional} of \(X\) if there exist \(\Lambda \in \mathcal{M}_{\infty}\) and an \(m\)-inessential set \(N\subset E\) such that \(\mathbb{P}_x(\Lambda)=1\) for \(x\in E\setminus N\) and \(\theta_t \Lambda \subset \Lambda\) for any \(t>0\), and the following conditions hold.
\begin{enumerate}
\item[(A.1)] For each \(t\geq 0,\) \(A_t|_{\Lambda}\) is \(\mathcal{M}_t|_{\Lambda}\)-measurable.
\item[(A.2)] For any \(\omega \in \Lambda\), \(A_{\cdot}(\omega)\) is right continuous on \([0,\infty)\) and has left limits on \((0, \zeta(\omega))\), \(A_0(\omega)=0,\) \(|A_{t}(\omega)|<\infty\) for \(t<\zeta(\omega)\) and \(A_t(\omega)=A_{\zeta(\omega)}(\omega)\) for \(t\geq \zeta (\omega)\). Moreover the equation
\[A_{t+s}(\omega)=A_t(\omega)+A_s(\theta_t\omega)\ \ \text{for\ every\ }t,s\geq 0,\]
is satisfied.
\end{enumerate}

An additive functional \(A\) is called a \textit{positive continuous additive functional} (PCAF in abbreviation) if \(A\) is a \([0,\infty]\)-valued continuous process, and denote by \({\bf A}_c^+\) the family of all PCAFs.
\end{definition}
The set \(\Lambda\) appearing in Definition \ref{DefPCAF} is called the defining set of \(A\). A PCAF \(A\) is called a PCAF in the strict sense if \(N\) appearing in Definition \ref{DefPCAF} is empty. PCAFs \(A\) and \(B\) are called \textit{\(m\)-equivalent} if \(\int_E\mathbb{P}_x(A_t\neq B_t)dm(x)=0\) for any \(t>0\). An \(m\)-equivalence is equivalent to the existence of a common defining set \(\Lambda\) and a common \(m\)-inessential set \(N\) such that \(A_t(\omega)=B_t(\omega)\) for any \(t \geq 0\) and \(\omega \in \Lambda.\) 

It is known that PCAFs and smooth measures are in one-to-one correspondence in the following sense. This correspondence is called \textit{the Revuz correspondence}. Therefore, a smooth measure is also called the \textit{Revuz measure}. See \cite[Theorem 4.1.1]{CF12} for details. For a measure \(\rho,\) we set \(\mathbb{E}_\rho[\cdot]:=\int_E \mathbb{E}_x[\cdot] \,d\rho(x).\)
\begin{theorem}[The Revuz correspondence]
\((i)\) For a PCAF \(A\), there exists a unique smooth measure \(\mu\) such that
\begin{equation}
\int_E fd\mu = \lim_{t \to 0}\frac{1}{t} \mathbb{E}_m\left[ \int_0^t f(X_s) dA_s  \right] \label{eq:AppPCAF-1}
\end{equation}
for any non-negative Borel function \(f\) on \(E.\)\\
\((ii)\) For any smooth measure \(\mu\), there exists a PCAF \(A\) satisfying \((\ref{eq:AppPCAF-1})\) up to the \(m\)-equivalence.
\end{theorem}

For example, for a bounded positive Borel function \(f\), we set
\[A_t:=\int_0^t f(X_s)ds,\]
then \(A:=\{A_t\}_t\) is a PCAF and \(A\) corresponds to a smooth measure \(fdm\). As another example, when a capacity of \(x_{0}\in E\) is positive, the local time \(L^{x_0}\) is a PCAF and its corresponding smooth measure is a Dirac measure \(\delta_{x_0}\). Denote by \(A^\mu\) a PCAF corresponding to \(\mu\). We remark that \(R_\alpha\mu(x):=\mathbb{E}_x[\int_0^\infty e^{-\alpha t}\,dA_t^\mu]\) is a quasi-continuous version of \(U_\alpha\mu\) for any \(\mu \in \mathcal{S}_0\).

\bigskip
\noindent
{\large \bf Acknowledgments.}

\noindent
This work was supported by JSPS KAKENHI Grant Numbers  25K17270 (T.O.) and 26K06884 (M.S.).

\end{document}